\pdfoutput=1
\documentclass[preprint,12pt,a4paper]{elsarticle}

\biboptions{sort&compress}
\usepackage[margin=2.5cm]{geometry}

\usepackage{amsmath}
\usepackage{amsfonts}
\usepackage{amssymb}
\usepackage{amsthm}

\usepackage{graphicx}
\usepackage{epstopdf}
\usepackage{color}
\usepackage{array}
\usepackage{booktabs}
\usepackage{inputenc}
\usepackage{todonotes}
\usepackage[T1]{fontenc}   
\usepackage{epsfig}
\usepackage{float}
\usepackage{mathabx}
\usepackage{appendix}
\usepackage{bbm}
\usepackage{bbold}
\usepackage{mathrsfs}
\usepackage[normalem]{ulem}
\usepackage{enumitem}
\usepackage{adjustbox}
\usepackage[%
colorlinks=true,
urlcolor=blue, % \href{...}{...} external (URL)
filecolor=red, % \href{...} local file
linkcolor=blue!70!black, % \ref{...} and \pageref{...}
citecolor=green!40!black,
pdfpagemode=UseNone,
bookmarksopen=false]{hyperref} 
\usepackage{cleveref}
\usepackage{tikz,colortbl}
\usetikzlibrary{calc}
\usepackage{zref-savepos}
\usepackage{booktabs}

\newcounter{NoTableEntry}
\renewcommand*{\theNoTableEntry}{NTE-\the\value{NoTableEntry}}

\newtheorem{definition}{Definition}[section]
\newtheorem{remark}{Remark}[section]
\newtheorem{lemma}{Lemma}[section]
\newtheorem{prop}{Proposition}[section]

\newtheorem{theorem}{Theorem}[section]

\newcommand{\mat}[1]{\ensuremath{\mathbf{#1}}}
\renewcommand{\vec}[1]{{\boldsymbol{#1}}}

\newlist{case}{enumerate}{1}
\setlist[case,1]{label=\textbf{Method} \arabic*:}

\newlist{form}{enumerate}{1}
\setlist[form,1]{label=\textbf{Form} \arabic*:}

\newlist{casenew}{enumerate}{1}
\setlist[casenew,1]{label=\textbf{Approach} \arabic*:}

\newlist{condition}{enumerate}{1}
\setlist[condition,1]{label=\textbf{Approach} \arabic*:}

\date{\today}
\begin{document}
%\nocite{*}

%\begin{frontmatter}

\title{Harmonic analysis on directed graphs and applications: From Fourier analysis to Wavelets
\tnoteref{t1,t2}}

\tnotetext[t1]{This work was supported by the ANR-14-CE27-0001 GRAPHSIP grant and the ACADEMICS Grant given by the IDEXLYON project of the Universit\'{e} de Lyon, as part of the "Programme Investissements d'Avenir" ANR-16-IDEX-0005. } 
\tnotetext[t2]{Preliminary results were presented in \cite{sevi2017multiresolution,sevi:GlobalSIP:2018}.}

%\title[Harmonic analysis on directed graphs]{Harmonic analysis on directed graphs and applications: From Fourier analysis to Wavelets.}

%\author[Harry Sevi, Gabriel Rilling, Pierre Borgnat]{Harry Sevi$^{1,2}$, Gabriel Rilling$^{1}$, Pierre Borgnat$^2$}

% \thanks{This work was supported by the ANR-14-CE27-0001 GRAPHSIP grant and the ACADEMICS Grant given by the IDEXLYON project of the Universit\'{e} de Lyon, as part of the "Programme Investissements d'Avenir" ANR-16-IDEX-0005. \\
% Preliminary results were presented in \cite{sevi2017multiresolution,sevi:GlobalSIP:2018}.}
% \keywords{harmonic analysis, graph signal processing, Fourier analysis, wavelets, directed graphs, random walks, semi-supervised learning}

\author[2,3]{Harry Sevi\corref{cor1}\fnref{fn1}}
\ead{harry.sevi@ens-paris-saclay.fr}

\author[2]{Gabriel Rilling} \ead{gabriel.rilling@cea.fr}

\author[3]{Pierre Borgnat} \ead{pierre.borgnat@ens-lyon.fr}

%\address[1]{
%Universit\'e Paris-Saclay, ENS Paris-Saclay, CNRS, Centre Borelli, F-91190 Gif-sur-Yvette, France}

\address[2]{CEA, LIST, Laboratoire de Sciences des Donn\'ees et de la D\'ecision, 91400 Gif-sur-Yvette, France}

\address[3]{Univ Lyon, Ens de Lyon, CNRS, Laboratoire de Physique, F-69342 Lyon, France}

\cortext[cor1]{Corresponding author.}
%\fntext[fn1]{This research work was carried out at ENS de Lyon \& CEA.}
\fntext[fn1]{Now at Borelli Center, ENS Paris-Saclay.}

\begin{abstract} 
We introduce a novel harmonic analysis for functions defined on the vertices of a strongly connected directed graph of which the random walk operator is the cornerstone. As a first step, we consider the set of eigenvectors of the random walk operator as a non-orthogonal Fourier-type basis for functions over directed graphs. We found a frequency interpretation by linking the variation of the eigenvectors of the random walk operator obtained from their Dirichlet energy to the real part of their associated eigenvalues. From this Fourier basis, we can proceed further and build multi-scale analyses on directed graphs. We propose both a redundant wavelet transform and a decimated wavelet transform by extending the diffusion wavelets framework by Coifman and Maggioni for directed graphs. The development of our harmonic analysis on directed graphs thus leads us to consider both semi-supervised learning problems and signal modeling problems on graphs applied to directed graphs highlighting the efficiency of our framework.
\end{abstract}


%\end{frontmatter}

\maketitle
 \section{Introduction} 
 
 In a world where data available for scientific or social purposes accumulates at an exponential pace, managing, exploiting and analyzing this torrent of data has become one of the challenges of our era. Some of them take the form of graphs, structures that arise in various fields such as neuroscience, the Internet, genomic data, road transportations or social networks to name a few~\cite{barthelemy2011spatial,barrat2008dynamical}. Thus, there is a need to develop efficient mathematical and computational approaches to process and analyze such graphs and data on graphs.
 
Among the existing methods, a large number involve the (graph) Laplacian~\cite{belkin2003laplacian,belkin2004regularization,gine2006empirical,coifman2006diffusionmaps,von2007tutorial}:
\begin{itemize}
\item As a fundamental subject in mathematics and physics, the Laplacian is known and used, through its spectral study to extract relevant geometric properties from a manifold~\cite{rosenberg1997laplacian} or a graph~\cite{chung1997spectral} and plays a role in machine learning applications such as spectral clustering~\cite{von2007tutorial}. 
\item In continuous space, the eigenfunctions of the Laplace-Beltrami operator represent the generalization of the Fourier basis on manifolds~\cite{rosenberg1997laplacian,lablee2015spectral}.
\item In discrete space, the eigenvectors of the graph Laplacian form an orthonormal Fourier-type basis for functions on undirected graphs~\cite{shuman2013emerging,ortega2017graph}, when undirected graphs are considered as a discretization or sampling of a manifold~\cite{belkin2007convergence,hein2007graph,gine2006empirical,coifman2006diffusionmaps}. The eigenvalue associated with each eigenvector is then related to the notion of frequency~\cite{shuman2013emerging,RICAUD2019474}. 
\end{itemize}

Thanks to these properties, the graph Laplacian bridges the gap between spectral graph theory and signal processing through the mathematical framework called ``signal processing on graphs"~\cite{shuman2013emerging,ortega2017graph}. This framework aims at extending the concepts of classical signal processing such as filtering or sampling, for functions defined on the vertices of a graph. 

First developed in the context of undirected graphs, signal processing on graphs considers the graph Laplacian (combinatorial or normalized)~\cite{chung1997spectral} as its core element. The graph Laplacian is a symmetric positive semidefinite operator. By the spectral theorem, the graph Laplacian admits an orthonormal basis of eigenvectors
which can be considered as Fourier modes, the corresponding eigenvalues being associated with a notion of frequency.

The purpose of this paper is to provide some answers to the extension of the framework of signal processing on graphs to the case of directed graphs. The extension of signal processing on directed graphs is of interest because
many networks studied in the scientific (e.g. neurosciences) or social (e.g. social networks) fields are directed. Therefore, the analysis of information on directed graphs also requires the development of mathematical approaches and therefore an extension of the framework of signal processing on graphs to the case of directed graphs.

Formally, the direct use of the core element of graph signal processing, i.e. the graph Laplacian, is no longer appropriate. A directed graph is naturally represented by a non-symmetric adjacency matrix. It is always possible to naively define a graph Laplacian in the directed case, but the spectral properties associated with this undirected graph Laplacian (e.g. nonnegative real eigenvalues and the existence of orthonormal eigenbases) are no longer verified in the case of directed graphs. There is no simple consensus on a definition of a Laplacian for directed graphs for which the variation of eigenvectors is linked to a notion of frequency. This is the main challenge of signal processing on directed graphs and of this paper: which reference operator(s) should a Fourier analysis be built upon that would also generalize Laplacian-based Fourier analysis on undirected graphs ?

Hereafter, the random walk operator on graphs~\cite{lovasz1993random,coulhon1998random,aldous2002reversible} is proposed as being a suitable reference operator for extending the framework of signal processing on directed graphs. Like the graph Laplacian, the random walk operator is associated to the notion of diffusion. However, unlike the graph Laplacian, its definition is as straightforward on directed graphs and on undirected ones. Assuming that the random walk operator is diagonalizable, it potentially admits complex conjugate eigenvectors as well as associated complex conjugate eigenvalues. Our framework is based on the fact that the variational analysis of the eigenvectors of the random walk operator through their associated Dirichlet energy~\cite{coulhon1998random} is related to the real part of the eigenvalue~\cite{sevi2017multiresolution}. Therefore, the Dirichlet energy of a given eigenvector of the random walk operator on a directed graph can be considered as its frequency, using the same analogy as in~\cite{shuman2013emerging}, now for directed graphs. Provided with this notion of frequency, the eigenvectors of the random walk operator form a suitable Fourier basis on directed graphs. It is furthermore consistent with the undirected setting when the Fourier basis is based on the random-walk Laplacian operator (see section~\ref{freq_ana_dg}).

By now considering the random walk operator as a reference operator and its associated eigenvectors as a Fourier basis for functions defined on directed graphs, the last part of our harmonic analysis on directed graphs consists in building a multi-resolution analysis for functions defined on the vertices of a directed graph. Complex graphs, both directed and undirected, have structures at different scales. Motivated by the efficiency of multi-resolution wavelet analysis in traditional signal processing~\cite{mallat1989theory} and the multi-scale dimension of graph data, a number of multi-resolution graph constructions have emerged in recent years~\cite{coifman2006diffusion,maggioni2008diffusion,hammond2011wavelets,rustamov2013wavelets,dong2017sparse}. Here, we  propose a new harmonic analysis built around the random walk operator and whose multi-resolution constructions generalize spectral graph wavelets~\cite{hammond2011wavelets,maggioni2008diffusion} and diffusion wavelets~\cite{coifman2006diffusion,bremer2006diffusion} on directed graphs.

\subsection{Literature review and related work}
Wavelets~\cite{mallat1989theory,cohen1992biorthogonal,sweldens1996lifting,kingsbury2001complex,candes2004new,peyre2008orthogonal} have been thoroughly investigated for over two decades. Their efficiency and success in analyzing functions defined on the real line have led to their generalization for functions defined on higher dimensional spaces such as the sphere~\cite{schroder1995spherical,antoine1999wavelets,wiaux2008exact} or other manifolds~\cite{maggioni2008diffusion,coifman2006diffusion,maggioni2005biorthogonal,coulhon2012heat,geller2009continuous,dong2017sparse,rahman2005multiscale}. The development of wavelets on manifolds and the multi-scale features of graphs and data on graphs recently led to the extension of wavelets constructions to the discrete space setting. These wavelet-type constructions include wavelets on unweighted graphs~\cite{crovella2003graph}, lifting wavelets~\cite{jansen2009multiscale},
diffusion wavelets~\cite{coifman2006diffusion,maggioni2005biorthogonal,bremer2006diffusion}, diffusion polynomial frames~\cite{maggioni2008diffusion}, spectral graph wavelets~\cite{hammond2011wavelets,tremblay2014graph,narang2012perfect,de2019localized}, Haar-like wavelets~\cite{szlam2005diffusion,gavish2010multiscale,irion2014hierarchical,tremblay2016graph}, average interpolating wavelets ~\cite{rustamov2011average},
graph wavelets via deep learning~\cite{rustamov2013wavelets}, multi-scale pyramid transform~\cite{shuman2016multiscale}, tight wavelet frames on graphs~\cite{dong2017sparse,gobel2018construction,behjat2016signal} and intertwining wavelets on graphs via random forests ~\cite{avena2017intertwining}. The wavelet constructions mentioned above were all developed on undirected graphs. In the present work, we intend to construct a harmonic analysis on directed graphs. The first stage of the proposed harmonic analysis is the development of a Fourier analysis. 

In recent years, Sandryhaila and Moura generalized some fundamental concepts of traditional signal processing such as filtering to directed graphs using the adjacency matrix as the central component of their framework~\cite{sandryhaila2013discrete}. They also proposed the generalized eigenvectors of the adjacency matrix obtained by Jordan decomposition as a Fourier-type basis on directed graphs~\cite{sandryhaila2014discrete}. To the best of our knowledge, no wavelet design has been proposed in this framework, with the exception of the development of the filterbank approach generalized in \cite{TayOrtega2017} which can be applied to bipartite directed graphs and uses the adjacency matrix for a polyphase representation of the filters.

More recently, new Fourier-type bases have been proposed which are based on a novel measure of \emph{graph directed variation} (GDV) for functions defined on directed graphs, which is defined as the Lov\`asz extension of the graph cut size~\cite{sardellitti2017graph,shafipour2017digraph}. In both propositions, Fourier modes are defined as a set of orthonormal functions that approximately solve a non-convex optimization problem involving the GDVs of the Fourier modes. The two approaches differ insofar as the minimized objective in~\cite{sardellitti2017graph} is simply the sum of the GDVs of the Fourier modes whereas in~\cite{shafipour2017digraph} the objective is designed to spread the individual GDVs. It is worth stressing that because of the non-convexity of the objective, the Fourier modes are not properly defined by the optimization problems alone and actually depend on the specific algorithms used to approximately solve the problems. In both cases, the Fourier modes could also be defined on undirected graphs but they would not coincide with standard approaches based on Laplacian operators. The reasons are that the Fourier modes are not eigenvectors of any given operator and that the GDV on undirected graphs differs from the common quadratic variation measures used on undirected graphs. Unlike these approaches, the Fourier basis that we propose is generally not orthogonal but it does generalize the Laplacian-based Fourier basis for undirected graphs. Furthermore we associate a notion of frequency to the Fourier modes, eigenvectors of the random walk operator, via their Dirichlet energy which is conveniently related to the corresponding eigenvalues. Most of the Fourier-like bases approaches on directed graphs have been also discussed in~\cite{marques2020signal}.  

% as the solutions of non-convex optimization problems. On the one hand, Sardellitti et al. have proposed the construction of a Fourier basis on directed graphs as a set of orthonormal vectors obtained by minimizing the Lov\`asz extension of the graph cut size~\cite{sardellitti2017graph}. On the other hand, Shafipour et al. also have proposed the construction of a set of orthonormal vectors minimizing the spectral dispersion function~\cite{shafipour2017digraph}.
% Consequently, these bases proposed by~\cite{sardellitti2017graph,shafipour2017digraph} are based on a specific notion of variation. 
% Contrary to the approaches~\cite{sardellitti2017graph,shafipour2017digraph} which are not based on a specific operator, our Fourier basis is non-orthogonal and obtained without solving non-convex optimization problems. In addition, the study of the Dirichlet energy of the eigenvectors off the random walk operator enables the definition of a notion of frequency associated with the basis. 
% Furthermore, the basis is well-defined mathematically, unlike the previous two since they are based on non-necessarily optimal solutions of non-convex optimization problems. 

Mhaskar has recently proposed the first wavelet-type frame construction for functions defined on directed graphs~\cite{mhaskar2016unified} as an extension of the diffusion polynomial frames~\cite{maggioni2008diffusion}. Although our redundant wavelet construction on directed graphs is inspired in part by that of polynomial diffusion frames, it differs mainly in the choice of the Fourier basis and the considered operator. Mhaskar proposes in~\cite{mhaskar2016unified} the set of left-singular vectors of the weighted adjacency matrix of a directed graph as a Fourier basis while we propose the set of eigenvectors of the random walk operator.

Finally, Furutani et al.~\cite{furutani2019graph} proposed an extension of graph signal processing to directed graphs using an operator called the Hermitian Laplacian or Magnetic Laplacian~\cite{colin2013magnetic, berkolaiko2013nodal} as their reference operator.  The main properties of this latter operator are to be hermitian and preserve the directionality of the directed graph. The magnetic Laplacian has been highlighted trough machine learning applications such that community detection on directed graphs ~\cite{fanuel2017magnetic} and graph neural networks on directed graphs~\cite{zhang2021magnet}.
Another operators based on directed graphs such that the Hermitian adjacency matrix ~\cite{guo2017hermitian} has been investigated through clustering on directed  graphs~\cite{cucuringu2020hermitian,laenen2020higher}, however no Fourier analyses of these operators have been carried out.

\subsection{Contributions} 
The contributions of this article are the following:
\begin{enumerate}
\item \textbf{Construction of Fourier bases on directed graphs}.
We propose the random walk operator associated with a random walk on a directed graph as a reference operator. We propose  the eigenvectors of the random walk operator as a Fourier basis on directed graphs and determine a variational characterization of the eigenvectors of the random walk operator -- see Proposition~\ref{dir_eig}. 
\item \textbf{Construction of multi-resolution analyses on directed graphs}.
We propose an  overcomplete spectral wavelet transform on directed graphs in~\Cref{spectral_wavelet_dg}. This construction extends the framework of spectral wavelets on undirected graphs~\cite{hammond2011wavelets} and diffusion polynomial frames~\cite{maggioni2008diffusion}. We also propose a critically sampled wavelet transform in Section~\ref{diffusion_wavelets_dg} generalizing the framework of diffusion wavelets~\cite{coifman2006diffusion,bremer2006diffusion} to the directed case.
\item \textbf{Efficiency of the theoretical framework through applications}.
The development of our harmonic analysis on directed graphs leads us to consider semi-supervised learning problems with $\ell_2$~regularization in Section~\ref{sec:semi_sup_l_two_norm} and $\ell_1$ regularization in section~\ref{sec:semi_sup_l_one} and signal modeling by filtering on directed graphs in Section~\ref{sig_model_dg} highlighting the effectiveness of our theoretical framework.   
\end{enumerate}

\subsection{Outline of the paper}
This paper is structured as follows. We introduce the essential aspects of graph theory and review the foundations of graph signal processing in~\Cref{sect:gsp_foundations}. We present operators defined on directed graphs built from the random walk operator and their properties in ~\Cref{sect:oper_dg} and propose a new class of operators on directed graphs expressed as the convex combination between the random walk operator and its time-reversed version. \Cref{dgft} is dedicated to the presentation of a Fourier transform on directed graphs as a set of eigenvectors of the random walk operator. In \Cref{Fourier_analysis} we propose a Fourier-type analysis for functions defined on directed graphs by studying the variation of the random walk's eigenvectors. To illustrate our Fourier-type analysis, we study the machine learning problems in~\Cref{sect:applications_one} such as semi-supervised learning and signal modeling on directed graphs and show through numerical experiments the efficiency of our framework with respect to the existing approaches. ~\Cref{sect:mra} is dedicated to multiresolution analyses on directed graphs. We present a redundant wavelets construction, as well as a critically sampled wavelets construction with the random walk operator as a reference operator and we illustrate these multiresolution analyses through some examples. We conclude in~\Cref{sec:ccl}.

\section{Fundamentals of signal processing on graphs}
\label{sect:gsp_foundations}

Signal processing on graphs \cite{shuman2013emerging,sandryhaila2013discrete,ortega2017graph} is a mathematical framework in which the core concepts of harmonic analysis are generalized to functions defined on the vertices of a graph. In this section, we introduce the essential aspects of the signal processing on graphs framework and give some additional remarks. 

\subsection{Graph theory setup}
Let $\mathcal{G}=(\mathcal{V},\mathcal{E},w)$ be a weighted directed graph where $\mathcal{V}=\{v_1,\ldots,v_N\}$ is a finite set of vertices, $\mathcal{E} \subseteq \mathcal{V}\times \mathcal{V}$ is a set of directed edges. Each edge $e_{ij}=(v_i,v_j)$ is directed and represents a link from vertex $v_i$ to vertex $v_j$. The weight function $w:\mathcal{V}\times\mathcal{V}\rightarrow\mathbb{R}_{+}$ satisfies the following conditions :
\begin{itemize}
\item $w(v_i,v_j)> 0$ \textrm{if}  $e_{ij}\in\mathcal{E}$.
\item $w(v_i,v_j)=0$ \textrm{if}  $e_{ij}\notin\mathcal{E}$.
\end{itemize}
We denote by $|\mathcal{V}|=N$, the cardinality of the vertex set $\mathcal{V}$, that is the total number of vertices in $\mathcal{V}$. A weighted directed graph $\mathcal{G}$ can be entirely represented by its weighted adjacency matrix $\mathbf{W}=\{w_{ij}\}_{1\leq i,j\leq N} \in \mathbb{R}_{+}^{N\times N}$ where  $w_{ij}=w(v_i,v_j)$ are the weights associated to the respective edges $e_{ij}$. We define respectively the out-degree and the in-degree of a vertex $v_i\in \mathcal{V}$ by $d_{i}^{+}=\sum_{j=1}^{N}w_{ij}$ and $d_{i}^{-}=\sum_{j=1}^{N}w_{ji}$. For the sake of simplicity, we refer to the out-degree $d_{i}^{+}$ of a vertex $v_i\in\mathcal{V}$ as its degree that we denote by $d_{i}$.

A limitation of our framework is that it only applies to the case of weighted graphs with nonnegative weights, that is $\mathbf{W}=\{w_{ij}\}_{1\leq i,j\leq N} \in \mathbb{R}_{+}^{N\times N}$. Other approaches exist for the case of directed graphs with positive and negative weights but they are not compatible with the one presented here~\cite{bauer2012normalized,sandryhaila2013discrete,sandryhaila2014discrete}.

We assume through the theoretical sections of this paper that the directed graph $\mathcal{G}$ is \textbf{strongly connected} and the random walk operators are \textbf{diagonalizable}. These are two necessary conditions for our theoretical framework. 
If the directed graph is not strongly connected, our framework may still apply if we modify the graph to make it strongly connected. This can be achieved using e.g. the Google method~\cite{langville2011google}, which will be used in~\Cref{sig_model_dg}. Since diagonalizable matrices form a dense set, the case of a non-diagonalizable random walk operator is a priori very rare in practice and typically never occurs in numerical eigendecompositions.

\subsection{Graph signals}
Let $f:\mathcal{V}\mapsto \mathbb{C}$ be a function defined on the vertex set $\mathcal{V}$ of a given directed graph $\mathcal{G}$. We define a graph signal $\boldsymbol f$ as a column vector representation of the function $f$ applied at each node $v_i\in \mathcal{V}$, i.e.
\begin{equation*}
\boldsymbol{f}=[f(v_1),\ldots,f(v_N)]^{\top}\in \mathbb{C}^{N}.
\end{equation*}

\noindent We now define the space of functions defined over the vertices of a directed graph.

\begin{definition}[\cite{mugnolo2016semigroup}]
Let $\mathcal{G}=(\mathcal{V},\mathcal{E})$ be a directed graph and $\mu:\mathcal{V}\mapsto[0,\infty)$ be a function on $\mathcal{V}$ considered as a  measure on $\mathcal{V}$ by setting $\mu(U)=\sum_{x \in U} \mu(x)$, $U \subset \mathcal{V}$. For $q\in[1,\infty)$, we denote $\ell^{q}(\mu,\mathcal{V})$, the space of functions $f:\mathcal{V}\rightarrow \mathbb{C}$ such that
\begin{equation*}
\|\boldsymbol{f}\|_{\ell^{q}(\mu,\mathcal{V})}=
\begin{cases}
\bigg(\sum_{x\in \mathcal{V}}|f(x)|^{q}\mu(x)\bigg)^{1/q} <\infty, \quad  \: q\in[1,\infty).\\
\underset{{x\in \mathcal{V}}}{\max}\,|f(x)|\mu(x) < \infty, \qquad   q=\infty. 
\end{cases}
\end{equation*}
\end{definition}

We assume throughout this paper that the graph signals are defined in $\ell^2(\mathcal{V},\mu)$ which is the Hilbert space of functions defined over the vertex set $\mathcal{V}$ of $\mathcal{G}$  endowed with the inner product
\begin{equation*}
\langle\boldsymbol{f},\boldsymbol{g}\rangle_{\mu}=\sum_{x\in \mathcal{V}}\overline{f(x)}g(x)\mu(x),
\end{equation*}
for all $\boldsymbol f,\boldsymbol g \in \ell^2(\mathcal{V},\mu)$ and where $\overline{f(x)}$ denotes the complex conjugate of $f(x)$. 

\section{Linear filters on graphs}
\label{sec_lin_fi_gr}
A linear operator on graph is represented by a matrix $\mathbf{H} \in \mathbb{C}^{N\times N}$ that acts on a graph signal input $\mathbf{s} \in \mathbb{C}^{N}$ and produces a graph signal output $\tilde{\mathbf{s}} \in \mathbb{C}^{N}$ according to the matrix vector product 

\begin{equation*}
\tilde{\mathbf{s}}=\mathbf{Hs}.
\end{equation*}

Within the framework of graph signal processing and inspired by conventional signal processing, we are mainly interested in a type of graph filters that commutes with a reference operator $\mathbf{R}$
\begin{equation}
\mathbf{HR}=\mathbf{RH}.
\label{eq:commutating_filter}
\end{equation}

More specifically, we define a graph filter $\mathbf{H}$ as a finite polynomial sum of a reference operator $\mathbf{R}$, that is 
\begin{equation}
\mathbf{H}=\sum_{t=0}^T h_t\mathbf{R}^{t},
\quad h_t\in \mathbb{C}, \quad \forall t=0,\dots,T.
\label{eq:polynomial_filter}
\end{equation}
Such a graph filter $\mathbf{H}$ commutes with $\mathbf{R}$ but the converse is generally not true. The following theorem establishes a particular condition on $\mathbf{R}$ under which all commutating filters~\eqref{eq:commutating_filter} are polynomial filters~\eqref{eq:polynomial_filter}.

\begin{theorem}
\label{lsi_polysum}
\cite{sandryhaila2013discrete} Let $\mathbf{R}$ be a reference operator on a directed graph $\mathcal{G}$. We assume that $\mathbf{R}$ is diagonalizable. We also assume that the characteristic and minimum polynomials of $\mathbf{R}$ are equal, i.e. $p_{\mathbf{R}}(x)=m_{\mathbf{R}}(x)$. Then a graph filter $\mathbf{H}$ that commutes with $\mathbf{R}$ is expressed as a finite polynomial sum of $\mathbf{R}$ 
\begin{equation*}
\label{poly_sum}
\mathbf{H}=\sum_{t=0}^T h_t\mathbf{R}^{t},
\quad h_t\in \mathbb{C}, \quad \forall t=0,\dots,T.
\end{equation*}
\end{theorem}

\Cref{poly_sum} indicates that if the following three conditions are met:
\begin{enumerate}
\item $\mathbf{H}$ commutes with a reference operator $\mathbf{R}$.
\item $\mathbf{R}$ is diagonalizable.
\item Each of the eigenspaces of $\mathbf{R}$ is of dimension one.
\end{enumerate}
then the filter $\mathbf{H}$ can be expressed as a finite polynomial sum of $\mathbf{R}$. Following the terminology in~\cite{sandryhaila2013discrete} we refer to a polynomial filter~\eqref{eq:polynomial_filter} as a linear ``shift'' invariant (LSI) graph filter, or just graph filter for short.

\begin{remark}
This notion of invariance with respect to a reference operator on graphs may seem rather restrictive insofar as many diagonalizable reference operators on graphs may have eigenspaces of dimension larger than 1. However, it could also be argued that the set of all commutating filters may include unexpected elements and that the set of polynomial filters may be a more relevant generalization of the LSI filters in classical signal processing. As an example, let us consider the undirected cycle graph. Its (combinatorial) Laplacian is the circulant matrix with three main diagonals
\begin{equation*}
  % \begin{pmatrix}
  % 2 & -1&  \cdots & -1\\
  % -1 & \ddots & \ddots & \vdots \\
  % \vdots & \ddots & \ddots &  -1& \\
  % -1 & \cdots &  -1 & 2
  % \end{pmatrix},
  \begin{pmatrix}
  2 & -1& 0 & \cdots & 0 &-1\\
  -1 & \ddots & \ddots & \ddots & \ddots & 0 \\
  0 & \ddots & \ddots & \ddots & \ddots & \vdots \\
  \vdots & \ddots & \ddots & \ddots & \ddots & 0 \\
  0 & \ddots & \ddots & \ddots & \ddots &  -1& \\
  -1 & 0 & \cdots & 0 &  -1 & 2
  \end{pmatrix},
\end{equation*}
for which the discrete Fourier basis is an eigenbasis. However, most eigenspaces of the Laplacian matrix have dimension equal to 2 and are spanned by pairs of complex conjugate exponentials $v_j[k]=e^{ikj/N}$ and $v_{-j}[k]=e^{-ikj/N}$. The operator that turns a complex exponential into its conjugate, swapping  $\mathbf{v}_j$ with $\mathbf{v}_{-j}$ (i.e. swapping positive and negative frequencies), is usually not considered a LSI filter in classical signal processing. It commutes with the Laplacian but cannot be expressed as a polynomial of the Laplacian. In contrast, filters defined as polynomials of the Laplacian correspond to symmetric filters in classical signal processing, which seems more appropriate as a generic filter class on the undirected cycle graph.
\end{remark}

% Thus, it will always be possible to construct a filter on graph which is a finite polynomial sum of a reference operator and therefore always invariant to its reference operator. 

% This type of graph filter corresponds to a linear ``shift'' invariant (LSI) filter~\cite{sandryhaila2013discrete} to $\mathbf{R}$, where the reference operator $\mathbf{R}$ is referred to as ``shift'' to extend the terminology classically used in signal processing. This notion of invariance with respect to a reference operator on graphs seems rather restrictive insofar as many reference operators on graphs do not systematically admit eigenspaces of dimension equal to 1 but eventually eigenspaces of dimension greater than 1. Thus, it will always be possible to construct a filter on graph which is a finite polynomial sum of a reference operator and therefore always invariant to its reference operator. 

A graph filter can also be characterized by its eigenvalues on each eigenspace of $\mathbf{R}$. Let $\{\mathfrak{E}_j\}_{j=1}^m$ denote the eigenspaces of $\mathbf{R}$ and $\{\mathbf{E}_j\}_{j=1}^m$ the corresponding spectral projectors, characterized by $\mathbf{E}_i^2 =\mathbf{E}_i $ and $\mathbf{E}_i\mathfrak{E}_j=\delta_{ij}\mathfrak{E}_j$. A graph filter with respect to a diagonalizable operator $\mathbf{R}$ is a linear combination of the spectral projectors
\begin{equation}
\label{fil_sp_proj}
\mathbf{H}=\sum_{j=1}^m\gamma_j\mathbf{E}_j,
\quad \gamma_j \in \mathbb{C} ,\quad \mathbf{E}_j\in\mathbb{C}^{N\times N},\quad \forall j=1,\dots,m, 
\end{equation}
where $\gamma_j$ is the eigenvalue of $\mathbf{H}$ associated to the eigenspace $\mathfrak{E}_j$.

\section{Canonical operators on directed graphs}
\label{sect:oper_dg}

In this section, we introduce the fundamental linear operators and their properties in order to develop a harmonic analysis on directed graphs. 

\subsection{Random walk operators on graphs}

We define a random walk on a directed graph as follows.

\begin{definition}[\cite{aldous2002reversible,lovasz1993random}]
A random walk on a strongly connected directed graph $\mathcal{G}=(\mathcal{V},\mathcal{E},w)$ is a homogeneous Markov chain $\mathcal{X}=(X_{n})_{n\geq 0}$ not necessarily reversible with a finite state space $\mathcal{V}$ and whose transition probabilities are proportional to the edges' weights. The entries of the transition matrix $\mathbf{P}$, written as $p(x,y), \forall {x,y \in \mathcal{V}}$ are defined as
\begin{equation*}
p(x,y)=\mathbb{P}(X_{n+1}=y|X_n=x) = \frac{w(x,y)}{\sum_y w(x,y)}.
\end{equation*}
\end{definition}
A transition matrix must have non-negative entries and rows that sum up to one. This is guaranteed by the fact that the $w(x,y)$ are non-negative and that the denominators are always non-zero thanks to the graph being strongly connected. 

% The denominators are nonzero because the graph is strongly connected and $p(x,y)\geq 0, \forall x,y$
% Since The $w(x,y)$ are nonnegative, it can be verified that $\mathbf{P}$ is a proper transition matrix:
% \begin{equation*}
% p(x,y) \geq 0\quad \mbox{and} \quad \sum_{y\in \mathcal{V}}p(x,y)=1,  \quad \forall x,y\in\mathcal{V}.
% \end{equation*} 

From a graph theory point of view, the transition matrix $\mathbf{P} \in \mathbb{R}^{N\times N}$ is equal to 
\begin{equation*}
\mathbf{P}=\mathbf{D}^{-1}\mathbf{W},
\end{equation*}
where $\mathbf{D}=\operatorname{diag}\{d_1,\ldots,d_n\}$ is the diagonal matrix of the out-degrees of the vertices and $\mathbf{W}$ is the weighted adjacency matrix. For any graph signal $\boldsymbol f$, measure $\mu:\mathcal{V}\mapsto [0,\infty)$ and  states $x,y \in \mathcal{V}$, we adopt the following conventions~\cite{bremaud2013markov}:
 
\begin{align*}
\mathbf{P}\boldsymbol{f}(x)&=\sum_{y\in \mathcal{V}}p(x,y)f(y),\\
\boldsymbol\mu\mathbf{P}(y)&=\sum_{x\in\mathcal{V}}\mu(x)p(x,y).
\end{align*}
where $\boldsymbol{f}$ and $\boldsymbol{\mu}$ are the vector representations of $f$ and $\mu$ respectively. Let us consider the transition matrix $\mathbf{P}\in\mathbb{R}^{N\times N}$ and a Kronecker delta function $\boldsymbol\delta_k$ at vertex $k\in\mathcal{V}$, as a graph signal. If $\mathbf{P}$ acts on $\boldsymbol\delta_k$, that is $\mathbf{P}\boldsymbol\delta_k$, the mass at vertex $v_k$ in $\boldsymbol\delta_k$ propagates towards vertices where $p(v_i,v_k)\neq 0$, that is vertices with edges pointing \emph{to} $v_k$. In other words, the mass propagates towards $parent$ vertices and not $children$, which may seem counter-intuitive. When applied to an arbitrary signal $\mathbf{x}$, the random walk operator $\mathbf{P}$ thus behaves as a local averaging of the children's vertices. Given the importance of local averaging in signal processing, this suggests that the random walk operator $\mathbf{P}$ may be natural as a core element of graph signal processing. 

We define the irreducibility of a random walk on a graph as follows.
\begin{definition}
A random walk $\mathcal{X}$  is said to be irreducible if for any $x,y\in\mathcal{V}$, the probability from $x$ to reach $y$ is strictly positive, in other words:
\begin{equation*}
\forall x,y\in\mathcal{V},\exists m <\infty: \mathbb{P}(X_{n+m}=y|X_n=x)>0.
\end{equation*}
\end{definition}
Irreducibility is equivalent to the strong connectivity of $\mathcal{G}$, that is there is a (directed) path from any vertex $v_i \in \mathcal{V}$ to any vertex $v_j\in \mathcal{V}$. 
\begin{remark}
In the case of connected undirected graphs, the graph is always strongly connected and the associated random walk is always irreducible.
\end{remark}

We now recall the notions of recurrent and transient state. 
\begin{theorem}
Let $x\in\mathcal{V}$ be a vertex. We denote by $\nu_x$ the probability that the random walk $\mathcal{X}$ which starts at vertex $x$ returns to $x$ at least once, that is
\begin{equation*}
\nu_x=\mathbb{P}[\exists n\in\mathbb{N} : X_n=x].
\end{equation*}
Then the vertex $x$ is respectively recurrent or transient if and only if $\nu_x=1$ or $\nu_x<1$ . 
\end{theorem}

\begin{prop}[\cite{bremaud2013markov}]
\label{irrec_reccu}
An irreducible random walk $\mathcal{X}$ with a finite state space is always recurrent: all
states are recurrent.
\end{prop}

\begin{remark}
An irreducible random walk $\mathcal{X}$ with finite state space is positive recurrent. Consequently, the random walk $\mathcal{X}$ admits an unique stationary distribution~\cite{bremaud2013markov}.
\end{remark}

We define the periodicity of a random walk on a graph as follows.
\begin{definition}
\label{period_rw}
Let $\mathcal{X}$ be a random walk on $\mathcal{G}$. The period of a vertex $x\in\mathcal{V}$ is:
\begin{equation*}
\varrho(x)=\operatorname{gcd}\{n\in\mathbb{N}^{+}:p^{(n)}(x,x)>0\}.
\end{equation*}
\end{definition}

Thus starting in $x$, the Markov chain can return to $x$ only at multiples of the period $\varrho$. The state $x$ is aperiodic if $\varrho(x)=1$ and periodic is $\varrho(x)>1$. Having defined the notions of irreducibility and periodicity, we are able to define the notion of ergodicity of a random walk.

\begin{definition}
Let $\mathcal{X}$ be a random walk on $\mathcal{G}$. The random walk $\mathcal{X}$ is ergodic if it is irreducible and aperiodic. 
\end{definition}

We set out the proposition for the stationary distribution of an ergodic random walk on $\mathcal{G}$. 

\begin{prop}[\cite{montenegro2006mathematical}]
Let $\mathcal{G}$ be a directed graph with $\underline{finite}$ state space $\mathcal{V}$. If a random walk $\mathcal{X}$ with its transition matrix $\mathbf{P}$ is ergodic, i.e. irreducible and aperiodic, the measures $\mathbf{P}^{n}(x,.)$ converge towards the row vector $\boldsymbol\pi=[\pi(v_1),\cdots,\pi(v_N)]\in \mathbb{R}_{+}^N$ as $n\rightarrow \infty$, i.e. the \emph{unique} stationary distribution. In particular, $\boldsymbol\pi\mathbf{P}=\boldsymbol\pi$, with :
\begin{equation*}
\sum_{i=1}^N\pi(v_i)=1,\quad \pi(v_i)\geq 0.
\end{equation*}
\end{prop}

\subsubsection{Reversibility}
\label{reversibility_sect}
Given the discrete-time Markov chain $\mathcal{X}=(X_n)_{n\geq 0}$ and $M>0$ a  finite time horizon, we define the time reversed Markov chain as $\mathcal{X}^{*}=(X_n^{*})=(X_{M-n})$ for $n=0,\cdots,M$. We denote by $\mathbf{P}^{*}=\{p^{*}(x,y)\}_{x,y\in\mathcal{V}}$ the transition matrix associated with the Markov chain $\mathcal{X}^{*}$. It verifies
\begin{equation*}
p^{*}(x,y)=\mathbb{P}(X^{*}_{n+1}=y|X^{*}_{n}=x)=\frac{\mathbb{P}(X_{M-n-1}=y)}{\mathbb{P}(X_{M-n}=x)}p(y,x),\quad \forall x,y\in\mathcal{V}.
\end{equation*}
If we assume that $\mathcal{X}$ is ergodic with stationary distribution $\pi$, the time reversed Markov chain $\mathcal{X}^{*}$ is also ergodic with stationary distribution $\pi$ and the entries $p^{*}(x,y)\in\mathbf{P}^{*}$ are
\begin{equation*}
p^{*}(x,y)=\frac{\pi(y)}{\pi(x)}p(y,x),\quad \forall x,y\in\mathcal{V}
\end{equation*}
or in its transition matrix version
\begin{equation*}
\label{rev_rw}
\mathbf{P}^{*}=\boldsymbol\Pi^{-1}\mathbf{P}^{\top}\boldsymbol\Pi,
\end{equation*}
where $\boldsymbol\Pi=\textnormal{diag}\{\pi(v_1),\cdots,\pi(v_N)\}$ is the diagonal matrix of the stationary distribution.

Let us introduce the function space $\ell^2(\mathcal{V},\pi)$ endowed with its inner product
\begin{equation*}
\langle \boldsymbol f,\boldsymbol g\rangle_{\pi}=\sum_{x\in\mathcal{V}}\overline{f(x)}g(x)\pi(x).
\end{equation*}
In this space, $\mathbf{P}^{*}$  is the adjoint of $\mathbf{P}$, that is $\langle\boldsymbol{f}, \mathbf{P}\boldsymbol{g}\rangle_{\pi}=\langle \mathbf{P}^{*} \boldsymbol{f},\boldsymbol{g}\rangle_{\pi}$, for all $\boldsymbol f,\boldsymbol g \in \ell^2(\mathcal{V},\pi)$.

The ergodic random walk $\mathcal{X}$ with its transition matrix $\mathbf{P}$ is \underline{reversible} if and only if we have the following relation
\begin{equation*}
\mathbf{P=P^{*}}.
 \end{equation*}

\begin{remark}
\label{rev_undir}
Ergodic random walks on finite undirected graphs are reversible. That means that the transition matrix associated to the time reversed random walk $\mathcal{X}^{*}$,  $\mathbf{P}^{*}$ is equal to the transition matrix $\mathbf{P}$ associated to the original ergodic random walk $\mathcal{X}$  namely $\mathbf{P}^{*}=\mathbf{P}$.
\end{remark}

\begin{remark}
\label{undir_stat_dis}
In the undirected setting, the stationary distribution $\pi$ admits a closed form expression. Indeed, on a given weighted undirected graph $\mathcal{G}=(\mathcal{V},\mathcal{E},w)$ represented by its symmetric adjacency matrix $\mathbf{W}=\{w_{xy}\}_{x,y\in\mathcal{V}}\in \mathbb{R}_{+}^{N\times N}$ and the degree of a vertex $x\in\mathcal{V}$ is $c(x)=\sum_{y\in\mathcal{V}}w_{xy}$. As a result, the associated random walk is reversible with stationary distribution $\pi$ defined by $\pi(x)=c(x)/c_{\mathcal{G}}$ where $c_{\mathcal{G}}=\sum_{x\in\mathcal{V}}c(x)$. However, in the directed setting, the stationary distribution  $\pi$ does not admit an analytical form. In order to calculate it, we use iterative methods such as the power iteration method~\cite{golub2012matrix} or Markov Chain Monte Carlo(MCMC) methods~\cite{atchade2009adaptive}.
\end{remark}

\begin{remark}
Ergodic random walks defined on directed graphs are typically non-reversible. Nevertheless, reversible ergodic random walks may be constructed on directed graphs by modifying the original non-reversible ergodic random walk. Such a modification is discussed in section~\ref{add_revs}. 
\end{remark}

\subsubsection{Eigenvalue distribution} 
By the Perron-Frobenius theorem~\cite{bremaud2013markov}, if the random walk $\mathcal{X}$ is ergodic with a diagonalizable transition matrix $\mathbf{P}$, the diagonalization of $\mathbf{P}$ admits a simple dominant eigenvalue $\lambda_{\max}=1$. The other eigenvalues $\{\lambda \neq 1\}$ satisfy $|\lambda|<1$, which means that all eigenvalues different of $\lambda_{\max}$ lie within the unit circle.

\subsection{Random walk generalizations}
\label{rw_gen}

Given a random walk $\mathcal{X}$ on a directed graph $\mathcal{G}$ with transition matrix $\mathbf{P}$, we are able to build new types of random walks based on $\mathbf{P}$ with various purposes.

\subsubsection{Lazy random walks}
\label{lazyrw_sec}
The periodicity of a random walk $\mathcal{X}$ can be overcome by considering the lazy random walk version of $\mathcal{X}$. The transition matrix $\tilde{\mathbf{P}}$ associated to the lazy random walk $\tilde{\mathcal{X}}$, based upon $\mathbf{P}$ is expressed as
\begin{equation*}
\tilde{\mathbf{P}}=\frac{\mathbf{I}+\mathbf{P}}{2}. 
\end{equation*}
The lazy random walk can be seen as the random walk on a modified graph $\tilde{\mathcal{G}}$ where an edge from each vertex of $\mathcal{G}$ to itself is added with a weight equal to the vertex's degree in $\mathcal{G}$. The lazy random walk $\tilde{\mathcal{X}}$ is always aperiodic, hence when $\mathcal{G}$ is strongly connected, $\tilde{\mathcal{X}}$ is also ergodic with a stationary measure $\pi$ verifying $\boldsymbol{\pi} = \boldsymbol{\pi}\tilde{\mathbf{P}}=\boldsymbol{\pi}\mathbf{P}$.

More generally, let us define $\tilde{\mathcal{P}}$, the class of transition matrices associated with generalized lazy random walks by
\begin{equation*}
\tilde{\mathcal{P}}=\bigg\{\tilde{\mathbf{P}}_{\gamma}: \tilde{\mathbf{P}}_{\gamma}=(1-\gamma)\mathbf{P}+\gamma \mathbf{I}\bigg| \gamma\in[0,1)\bigg\}. 
\end{equation*}
We note that the elements $\tilde{\mathbf{P}}_{\gamma}\in\tilde{\mathcal{P}}$ share the same eigenspaces as $\mathbf{P}$, hence the graph filters with respect to $\tilde{\mathbf{P}}$ are also graph filters with respect to $\mathbf{P}$.

\subsubsection{Additive reversibilization}
\label{add_revs}
From a non-reversible ergodic random walk $\mathcal{X}$ with transition matrix $\mathbf{P}$ and unique stationary measure $\pi$, an additive reversibilization of $\mathcal{X}$ can be constructed~\cite{fill1991eigenvalue}, denoted by $\bar{\mathcal{X}}$, whose transition matrix is the average between $\mathbf{P}$ and its time reversed $\mathbf{P}^*$:

\begin{equation}
\label{add_rev}
\bar{\mathbf{P}}=\frac{\mathbf{P}+\mathbf{P}^{*}}{2}.
\end{equation}
$\bar{\mathcal{X}}$ is a reversible random walk with the same unique stationary distribution $\pi$. 

More generally, let us define $\bar{\mathcal{P}}$ the class of convex combinations between the random walk matrix $\mathbf{P}$ and its time reversed version $\mathbf{P}^{*}$ as
\begin{equation}
\label{conv_rw}
\bar{\mathcal{P}}=\bigg\{\bar{\mathbf{P}}_{\alpha}: \bar{\mathbf{P}}_{\alpha}=(1-\alpha)\mathbf{P}+\alpha\mathbf{P}^{*}\bigg| \alpha\in[0,1]\bigg\}.
\end{equation}
We note that the elements $\bar{\mathbf{P}}_{\alpha}\in\mathcal{P}$ share the same stationary distribution $\pi$ but do not have the same eigenspaces.
We also note that all the random walks associated to transition matrices $\bar{\mathbf{P}}_{\alpha}\in\bar{\mathcal{P}}$ are non-reversible except for $\alpha=1/2$, namely $\bar{\mathbf{P}}_{1/2}=\bar{\mathbf{P}}$.

\begin{remark}
Given an ergodic random walk $(\mathcal{X},\mathbf{P},\pi)$, we note that all lazy and reversibilized versions, respectively $\tilde{\mathbf{P}}_{\gamma}\in\tilde{\mathcal{P}}$ and $\bar{\mathbf{P}}_{\alpha}\in\bar{\mathcal{P}}$ share the same unique stationary distribution $\pi$.
\end{remark}

\begin{remark}
As a generalization of random walks, we may also consider the multiplicative reversibilization of the ergodic non-reversible random walk $\mathbf{P}$, that is $\mathbf{P}\mathbf{P}^{*}$ which is also a reversible Markov chain with stationary distribution $\pi$. It is useful e.g. for defining the convergence bounds of non-reversible Markov chains~\cite{fill1991eigenvalue}.
\end{remark}

\subsection{Laplacians on directed graphs}
In this section, we introduce several definitions of Laplacians on directed graphs, which are extended from the undirected cases case~\cite{anis2016efficient,girault2018irregularity,marques2020signal}.

\subsubsection{Normalized directed graph Laplacian}
The normalized Laplacian on a directed graph $\mathcal{G}$ is expressed in terms of the transition matrix $\mathbf{P}$ of an ergodic random walk $\mathcal{X}$ on $\mathcal{G}$ and is defined as follows. 

\begin{definition}[\cite{chung2005laplacians,butler2007interlacing}]
Let $\mathcal{G}=(\mathcal{V},\mathcal{E})$ be a directed graph with $|\mathcal{V}|=N$. Let $\mathcal{X}$ be an ergodic random walk on $\mathcal{G}$ with transition matrix $\mathbf{P}$ and unique stationary distribution $\pi$. The normalized Laplacian on $\mathcal{G}$ is defined by

\begin{equation}
\label{eq:dir_norm_laplacian}
\mathcal{L}=\mathbf{I}-\frac{\boldsymbol\Pi^{1/2}\mathbf{P}\boldsymbol\Pi^{-1/2}+\boldsymbol\Pi^{-1/2}\mathbf{P}^{\top}\boldsymbol\Pi^{1/2}}{2},
\end{equation}
where $\mathbf{I}$ is the identity matrix and  $\boldsymbol\Pi=\textnormal{diag}\{\pi(v_1),\ldots,\pi(v_N)\}$ is the diagonal matrix of the stationary distribution.
\end{definition}

\subsubsection{Random walk Laplacian}
Another definition is the random walk Laplacian $\mathbf{L}_{\operatorname{RW}}$, defined as
\begin{equation}
\label{eq:rw_dir_laplacian}
\mathbf{L}_{\operatorname{RW}}=\mathbf{I}-\bar{\mathbf{P}}.
\end{equation}

The normalized Laplacian on directed graphs $\mathcal{L}$ is connected to the random walk Laplacian through the following relation
\begin{equation*} 
\mathcal{L}=\boldsymbol\Pi^{1/2}\mathbf{L}_{\operatorname{RW}}\boldsymbol\Pi^{-1/2}.
\end{equation*}
% where the random walk Laplacian on directed graphs $\mathbf{L}_{\operatorname{RW}}$ is defined by 
More generally, we define  the random walk Laplacian on directed graphs associated to a transition matrix $\bar{\mathbf{P}}_{\alpha}\in\bar{\mathcal{P}}$ by
\begin{equation*}
\mathbf{L}_{\operatorname{RW},\alpha}=\mathbf{I}-\frac{\bar{\mathbf{P}}_{\alpha}+\bar{\mathbf{P}}_{\alpha}^{*}}{2}.
\end{equation*}
In the following proposition, we show that the random walk Laplacian $\mathbf{L}_{\operatorname{RW},\alpha}$ is equal to $\mathbf{L}_{\operatorname{RW}}$. 
\begin{prop}
For any $\bar{\mathbf{P}}_{\alpha}\in\bar{\mathcal{P}}$, we have
\begin{equation}
\mathbf{L}_{\operatorname{RW},\alpha}=\mathbf{L}_{\operatorname{RW}}, \quad \forall \alpha\in[0,1].
\end{equation}
\end{prop}

\begin{proof}
\begin{equation*}
\begin{split}
%\begin{align*}
\mathbf{L}_{\operatorname{RW},\alpha}&=\mathbf{I}-\frac{\bar{\mathbf{P}}_{\alpha}+\bar{\mathbf{P}}_{\alpha}^{*}}{2}\\
&=\mathbf{I}-\frac{\alpha\mathbf{P}+(1-\alpha)\mathbf{P}^{*}+\alpha\mathbf{P}^{*}+(1-\alpha)\mathbf{P}}{2}\\
&=\mathbf{I}-\frac{\alpha(\mathbf{P}+\mathbf{P}^{*})+(1-\alpha)(\mathbf{P}+\mathbf{P}^{*})}{2}\\
&=\mathbf{I}-\frac{\mathbf{P}+\mathbf{P}^{*}}{2}\\
&=\mathbf{L}_{\operatorname{RW}}. \qedhere
%\end{align*}
\end{split}
\end{equation*}
\end{proof}

\subsubsection{Combinatorial Laplacian}
Finally, we also introduce the combinatorial Laplacian $\mathbf{L}$, defined as
\begin{equation}
\label{eq:un_laplacian}
\mathbf{L}=\boldsymbol\Pi-\frac{\boldsymbol\Pi\mathbf{P}+\mathbf{P}^{\top}\boldsymbol\Pi}{2}.
\end{equation}
The latter is related to the random walk Laplacian through
\begin{equation*}
  \mathbf{L} = \boldsymbol\Pi \:\mathbf{L}_{\operatorname{RW}}. 
\end{equation*}

\begin{remark}
As mentioned in remark~\ref{rev_undir}, the transition matrix of a random walk on an undirected graph is equal to the transition matrix of its time reversed Markov chain. Furthermore, the stationary distribution $\pi$ admits a closed form as we mention in remark~\ref{undir_stat_dis}. As a result, the definitions \cref{eq:dir_norm_laplacian,eq:rw_dir_laplacian}, generalize the usual definitions for undirected graphs~\cite{chung1997spectral}. The combinatorial Laplacian on directed graphs \eqref{eq:un_laplacian} (thereafter denoted $\mathbf{L}_{directed}$) generalizes the one on undirected graphs up to a multiplicative factor: $\mathbf{L}_{undirected} = \mathbf{D}-\mathbf{W} = (\sum_i d_i)\mathbf{L}_{directed}$.
\end{remark}

\subsection{Canonical operators on directed graphs and Hilbert spaces}
\label{lino_hilsp}
Let $\mathcal{G}=(\mathcal{V},\mathcal{E})$ be a directed graph, $\ell^2(\mathcal{V})$ and $\ell^2(\mathcal{V},\pi)$ be the Hilbert spaces associated respectively with the counting measure and with the stationary measure $\pi$ of the ergodic random walk $\mathcal{X}$ on $\mathcal{G}$. Let $\varphi:\ell^2(\mathcal{V})\rightarrow\ell^2(\mathcal{V},\pi)$ be the linear mapping defined as follows
\begin{equation}
 \label{isometric_op}
 \varphi:\boldsymbol f\mapsto\boldsymbol\Pi^{-1/2}\boldsymbol f,\quad \forall \boldsymbol f\in\ell^2(\mathcal{V}).
 \end{equation}
 \begin{definition}[\cite{kubrusly2003hilbert}]
 Let $\mathsf{H},\mathsf{K}$ bet two Hilbert spaces. A linear transformation $\mathbf{V}:\mathsf{H}\mapsto \mathsf{K}$ is an isometry if and only if
 \begin{equation*}
 \|\mathbf{V}(\boldsymbol x- \boldsymbol y)\|_{\mathsf{H}}=\|\boldsymbol x-\boldsymbol y\|_{\mathsf{K}}, \quad \forall \boldsymbol x,\boldsymbol y\in\mathsf{H}.
 \end{equation*}
 \end{definition}
 \begin{prop}
 The linear transformation $\varphi:\ell^2(\mathcal{V})\mapsto\ell^2(\mathcal{V},\pi)$ is an isometry.
  \begin{proof}
 Given $\boldsymbol g,\boldsymbol h\in\ell^2(\mathcal{V})$, and the linear transformation $\varphi$, we have
\begin{equation*}
 \langle \varphi(\boldsymbol g),\varphi(\boldsymbol h)\rangle_{\pi}=\langle \boldsymbol \Pi^{-1/2}\boldsymbol g, \boldsymbol\Pi^{-1/2}\boldsymbol h\rangle_{\pi}=\langle \boldsymbol g,\boldsymbol h\rangle.\qedhere
 \end{equation*}
 \end{proof}
 \end{prop}

 Let us introduce the operator $\mathbf{T}\in\ell^2(\mathcal{V})$ defined as
 \begin{equation}
 \label{rw_ldeux}
\mathbf{T}=\boldsymbol\Pi^{1/2}\mathbf{P}\boldsymbol\Pi^{-1/2}. 
 \end{equation}
 \begin{definition}[\cite{kubrusly2003hilbert}]
 Let $\mathsf{H},\mathsf{K}$ bet two Hilbert spaces. An invertible bounded linear transformation $\mathbf{V}:\mathsf{H}\mapsto \mathsf{K}$ intertwines an operator  $\mathbf{M}:\mathsf{H}\mapsto\mathsf{H}$ to an operator $\mathbf{S}:\mathsf{K}\mapsto\mathsf{K}$ if
 \begin{equation*}
\mathbf{V} \mathbf{M}= \mathbf{S}\mathbf{V}.
 \end{equation*}
 The operators $\mathbf{M}$ and $\mathbf{S}$ are called similar. 
 \end{definition}
 \begin{prop}
 The linear operator $\varphi$ intertwines the operator $\mathbf{T}$ on $\ell^2(\mathcal{V})$ to the operator $\mathbf{P}$ on $\ell^2(\mathcal{V},\pi)$.
 \begin{proof}
 \begin{equation*}
 \varphi(\mathbf{T})=\boldsymbol\Pi^{-1/2}\mathbf{T}=\boldsymbol\Pi^{-1/2}\boldsymbol\Pi^{1/2}\mathbf{P}\boldsymbol\Pi^{-1/2}=\mathbf{P}\boldsymbol\Pi^{-1/2}.\qedhere
 \end{equation*}
 \end{proof}
 \end{prop}
Hence the operators $\mathbf{T}$ and $\mathbf{P}$ are similar with respect to $\varphi$. Identically, the random walk Laplacian $\mathbf{L}_{\operatorname{RW}}$ and the normalized Laplacian $\mathcal{L}$ are also similar with respect to~$\varphi$. This suggests that $\mathbf{P}$ and $\mathbf{L}_{\operatorname{RW}}$ are best viewed as operators on $\ell^2(\mathcal{V},\pi)$ while $\mathbf{T}$ and $\mathcal{L}$ are essentially the same operators on $\ell^2(\mathcal{V})$.

\section{Directed graph Fourier transform}
\label{dgft}

Let $\mathbf{P}$ be the transition matrix of the ergodic random walk $\mathcal{X}$. We assume $\mathbf{P}$ diagonalizable, that is $\mathbf{P}$ admits an eigenvalue decomposition 
\begin{equation*}
\mathbf{P}=\boldsymbol\Xi\boldsymbol\Theta\boldsymbol\Xi^{-1},
\end{equation*}
where $\boldsymbol\Xi=[\boldsymbol\xi_1,\cdots,\boldsymbol\xi_N]$ is an eigenbasis with elements $\boldsymbol\xi_j,j=1,\ldots,N$ and $\boldsymbol\Theta=\operatorname{diag}\{\vartheta_1,\cdots,\vartheta_N\}$ is the corresponding diagonal eigenvalue matrix, non necessarily distinct. Given a graph signal $\boldsymbol s$, its \underline{directed graph Fourier transform} denoted by $\hat{\boldsymbol s}=[\hat{s}_1,\ldots,\hat{s}_N]$ is
\begin{equation*}
\hat{\boldsymbol s}=\boldsymbol\Xi^{-1}\boldsymbol s.
\end{equation*}
The values $\{\hat{s}_j\}_{j=1}^N$ characterize the content of the graph signal $\boldsymbol s$ in the graph Fourier domain. 

The graph Fourier transform is not uniquely defined as it depends on the choice of eigenbasis $\boldsymbol\Xi$. First all eigenvectors $\boldsymbol\xi_j$ are defined up to a multiplicative scalar. This can be easily resolved by normalizing the vectors and e.g. making the first non-zero coefficient of each eigenvector real and positive. Second, whenever there are eigenspaces with dimension greater than~1, there are infinite choices for the orientations of the basis vectors within such eigenspaces. There is generally no solution to this second issue. This implies that the values of individual Fourier coefficients $\hat{s}_j$ within such eigenspaces are rather meaningless. However, it only impacts cases with eigenspaces of dimension greater than~1, which are not that common in applications.

The \underline{inverse directed graph Fourier transform} corresponding to an eigenbasis $\boldsymbol\Xi$ is given by 
\begin{equation*}
\boldsymbol s=\boldsymbol \Xi \hat{\boldsymbol s}.
\end{equation*}
It reconstructs the original signal from its frequency contents by forming a linear combination of eigenvectors weighted by the Fourier coefficients. 

More generally, any $\bar{\mathbf{P}}_{\alpha}\in\bar{\mathcal{P}}$ admits an eigenvalue decomposition 
\begin{equation*}
\bar{\mathbf{P}}_{\alpha}=\boldsymbol\Xi_{\alpha}\boldsymbol\Theta_{\alpha}\boldsymbol\Xi^{-1}_{\alpha}. 
\end{equation*}
Consequently, one can build an infinity of Fourier-type bases $\{\boldsymbol\Xi_{\alpha}\}_{\alpha\in[0,1]}$ for functions defined on directed graphs. Among these, the case $\alpha=1/2$ is particularly interesting insofar as $\bar{\mathbf{P}}_{1/2} = \bar{\mathbf{P}}$ is self-adjoint in $\ell^2(\mathcal{V},\pi)$. This implies that there is an orthonormal eigenbasis $\boldsymbol\Xi_{1/2}$ of $\ell^2(\mathcal{V},\pi)$ which yields the following theorem.

% and that the corresponding Fourier transform is an isometric morphism from $\ell^2(\mathcal{V},\pi)$ to $\ell^2(\{1,\ldots,N\})$.
\begin{theorem}[Generalized Parseval's Theorem]
Given an eigenbasis $\boldsymbol\Xi_{1/2}$ of $\bar{\mat{P}}$ that is orthonormal in $\ell^2(\mathcal{V},\pi)$, the corresponding Fourier transform is an isometric operator from $\ell^2(\mathcal{V},\pi)$ to $\ell^2(\{1,\ldots,N\})$
% \begin{equation}
%   \langle \vec{\Xi}_{1/2}^{-1} \vec{x},\vec{\Xi}_{1/2}^{-1} \vec{y} \rangle = \langle \vec{x},\vec{y} \rangle_\pi, \quad \forall \vec{x},\vec{y}\in\ell^2(\mathcal{V},\pi).
% \end{equation}
% \begin{equation}
% \langle\vec{\Xi}_{1/2}\vec{x},\vec{\Xi}_{1/2}\vec{y}\rangle_{\pi}=\langle\vec{x},\vec{y}\rangle, \quad \forall \vec{x},\vec{y}\in\ell^2(\{1,\ldots,N\}).
% \end{equation}
\begin{equation}
\langle\vec{\Xi}_{1/2}^{-1}\vec{x},\vec{\Xi}_{1/2}^{-1}\vec{y}\rangle=\langle\vec{x},\vec{y}\rangle_\pi, \quad \forall \vec{x},\vec{y}\in\ell^2(\mathcal{V},\pi).
\end{equation}
%  \begin{proof}
%  $\vec{\Xi}_{1/2}^\top\vec{\Pi}\vec{\Xi}_{1/2} = \mat{I}\Longrightarrow \vec{x}^\top\vec{\Xi}_{1/2}^{\top-1}\vec{\Xi}_{1/2}^{-1}\vec{y} = \vec{x}^\top\vec{\Xi}_{1/2}^{\top-1}\vec{\Xi}_{1/2}^\top\vec{\Pi}\vec{y} = \vec{x}^\top\vec{\Pi}\vec{y}$
%  \end{proof}

 % \begin{proof}
 % $\vec{\Xi}_{1/2}^\top\vec{\Pi}\vec{\Xi}_{1/2} = \mat{I}\Longrightarrow \langle\vec{\Xi}_{1/2}\vec{x},\vec{\Xi}_{1/2}\vec{y}\rangle_{\pi}=\vec{x}^{\top}\vec{\Xi}_{1/2}^{\top}\vec{\Pi}\vec{\Xi}_{1/2}\vec{y}=\vec{x}^{\top}\vec{y}=\langle\vec{x},\vec{y}\rangle$
 % \end{proof}
 \begin{proof}
 $\vec{\Xi}_{1/2}^\top\vec{\Pi}\vec{\Xi}_{1/2} = \mat{I}\Longrightarrow %
 \langle\vec{\Xi}_{1/2}^{-1}\vec{x},\vec{\Xi}_{1/2}^{-1}\vec{y}\rangle=\vec{x}^{\top}\vec{\Xi}_{1/2}^{-\top}\vec{\Xi}_{1/2}^{-1}\vec{y}=\vec{x}^{\top}\vec{\Pi}\vec{y}=\langle\vec{x},\vec{y}\rangle_\pi$
 \end{proof}
\end{theorem}
Another interesting property of $\bar{\mathbf{P}}$ is that its eigenspaces are the same as those of the random walk Laplacian $\mathbf{L}_{\operatorname{RW}}$~\eqref{eq:rw_dir_laplacian}.

\begin{remark}
 \label{remark_undir_one}
In the undirected setting, a suitable Fourier-type basis is usually defined as an orthonormal eigenbasis of the combinatorial Laplacian $\mathbf{L}$  or its normalized counterpart $\mathcal{L}$ \cite{chung1997spectral}. These Fourier-type bases are both orthonormal for $\ell^2(\mathcal{V})$. Another less common choice is a Fourier-type basis based upon the random walk Laplacian $\mathbf{L}_{\operatorname{RW}}=\mathbf{I}-\mathbf{P}$. As in the directed case with $\bar{\mat{P}}$, it leads to a Fourier basis that is orthonormal in $\ell^2(\mathcal{V},\pi)$, where $\pi$ is proportional to the degrees of the vertices (see remark~\ref{undir_stat_dis}). Some properties of the different definitions are investigated in~\cite{girault2018irregularity}.

\end{remark}

\section{Fourier analysis on directed graphs}
\label{Fourier_analysis}
In this section, we propose a new analysis on directed graphs that is different from the existing approaches on directed graphs~\cite{sandryhaila2014discrete,sardellitti2017graph,shafipour2017digraph,mhaskar2016unified}. It is based on the study of variations in the eigenvectors of the random walk operator. Before discussing it in greater detail, we introduce the elements that allow us to study the variation of signals on directed graphs.

\subsection{Regularity of signals on graphs}
\label{reg_gs}

The behavior of a graph signal over a directed graph (or undirected) can be analyzed by measuring its regularity or smoothness. We first define the length of the graph gradient  at a given vertex as follows.

\begin{definition}[\cite{coulhon1998random}]
Let $\mathcal{G}=(\mathcal{V},\mathcal{E})$ be a directed graph and $\nu:\mathcal{E}\rightarrow [0,\infty)$ be a positive measure defined on the edge set $\mathcal{E}$. The length of the graph gradient of a graph signal $\boldsymbol f$ at vertex $x\in\mathcal{V}$ on an arbitrary graph $\mathcal{G}$ under the measure $\nu$ is

\begin{equation*}
|\nabla\boldsymbol{f}(x)|=\bigg(\frac{1}{2}\sum_{y\in\mathcal{V},(x,y)\in \mathcal{E}}|f(x)-f(y)|^2\nu(x,y)\bigg)^{1/2}. 
\end{equation*}
\end{definition}

Intuitively, the length of the graph gradient measures the smoothness of a graph signal around a given vertex. We now introduce the Dirichlet energy as a measure of the smoothness of a signal over a strongly connected graph.

\begin{definition}[\cite{montenegro2006mathematical}]
The Dirichlet energy of a graph signal  $\boldsymbol f$ associated to the ergodic random walk $\mathcal{X}$ with transition matrix $\mathbf{P}$ and stationary distribution  $\pi$ on a strongly connected graph $\mathcal{G}$ is 

\begin{align}
\label{dir_nrj}
\mathcal{D}_{\pi,\mathbf{P}}^2(\boldsymbol{f})&=\frac{1}{2}\sum_{(x,y)\in\mathcal{E}}\pi(x)p(x,y)|f(x)-f(y)|^2,\\
&=\langle \boldsymbol{f},\mathbf{L}_{\operatorname{RW}}                                                                                                                                                                                                                                                                                                                                                                                     \boldsymbol{f}\rangle_{\pi} \nonumber.
\end{align}

\end{definition}

As we can appreciate, $\mathcal{D}_{\pi,\mathbf{P}}^2(\boldsymbol{f})=\||\nabla\boldsymbol{f}|\|_{2,\nu}^2$ where $\nu(x,y)=\pi(x)p(x,y)$ where $p(x,y)$ is the $(x,y)$ entry of the transition matrix $\mathbf{P}$ and $\pi(x)$ is the stationary distribution at vertex $x$. For all $\bar{\mathbf{P}}_{\alpha}\in\bar{\mathcal{P}}$, the associated Dirichlet energy is the same:
\begin{equation*}
\mathcal{D}_{\pi,\bar{\mathbf{P}}_{\alpha}}^2(\boldsymbol{f})=\mathcal{D}_{\pi,\mathbf{P}}^2(\boldsymbol{f})=\langle \boldsymbol{f},\mathbf{L}_{\operatorname{RW}}                                                                                                                                                                                                                                                                                                                                                                                     \boldsymbol{f}\rangle_{\pi}.
\end{equation*}

We also introduce the Rayleigh quotient of a graph signal $\boldsymbol f$ associated to the Dirichlet energy $\mathcal{D}_{\pi,\mathbf{P}}^2(\boldsymbol{f})$ as

\begin{equation*}
\mathcal{R}_{\pi,\mathbf{P}}(\boldsymbol{f})=\frac{\mathcal{D}_{\pi,\mathbf{P}}^2(\boldsymbol{f})}{\|\boldsymbol{f}\|_{\pi}^2}.
\end{equation*}

For any $\bar{\mathbf{P}}_{\alpha}\in\bar{\mathcal{P}}$, we have
\begin{equation*}
\mathcal{R}_{\pi,\mathbf{P}}(\boldsymbol{f})=\mathcal{R}_{\pi,\bar{\mathbf{P}}_{\alpha}}(\boldsymbol{f}).
\end{equation*}

\subsection{Frequency analysis on directed graphs}
\label{freq_ana_dg}

The following proposition settles a key connection between the regularity of the eigenvectors of the transition matrix $\mathbf{P}$ of the ergodic random walk $\mathcal{X}$ and their associated eigenvalues. 

\begin{prop}\label{dir_eig}
Let $\boldsymbol \xi\in\mathbb{C}^N$ 
be an eigenvector of a transition matrix $\mathbf{P}$ of an ergodic random walk $\mathcal{X}$, with stationary distribution $\pi$, associated to the eigenvalue $\vartheta \in \mathbb{C}$. The Rayleigh quotient of $\boldsymbol\xi$ is given by
\begin{equation*}
\mathcal{R}_{\pi,\mathbf{P}}(\boldsymbol\xi) =1-\mathfrak{Re}(\vartheta),
\end{equation*}
 where $\mathfrak{Re}(\vartheta)$ denotes the real part of $\vartheta \in \mathbb{C}$. 
 \begin{proof}
\begin{align*}
\mathcal{R}_{\pi,\mathbf{P}}(\boldsymbol\xi)&=\frac{1}{\|\boldsymbol\xi\|_{\pi}^2}\bigg(\langle \boldsymbol \xi, \boldsymbol \xi\rangle_{\pi}-\frac{1}{2}\langle \boldsymbol \xi,\mathbf{P}\boldsymbol \xi\rangle_{\pi}-\frac{1}{2}\langle \boldsymbol \xi,\mathbf{P}^{*}\boldsymbol \xi\rangle_{\pi}\bigg)\\
&=\frac{1}{\|\boldsymbol\xi\|_{\pi}^2}\bigg(\|\boldsymbol \xi\|_{\pi}^2-\frac{\vartheta+\bar{\vartheta}}{2}\|\boldsymbol \xi\|_{\pi}^2\bigg)\\
% &=||\boldsymbol \xi||_{\pi}^2-\frac{1}{2}\eta||\boldsymbol \xi||_{\pi}^2-\frac{1}{2}\overline{\eta}||\boldsymbol \xi||_{\pi}^2\\
\mathcal{R}_{\pi,\mathbf{P}}(\boldsymbol\xi)&=\big[1-\mathfrak{Re}(\vartheta)\big].\qedhere
\end{align*}
\end{proof}
\end{prop}

The latter proposition thus indicates that the smoothness of any eigenvector $\boldsymbol\xi$ of $\mathbf{P}$, as described by its Rayleigh quotient, is associated to the real part of its respective eigenvalue $\vartheta$. More generally, the smoothness of any eigenvector $\boldsymbol\xi_{\alpha}$ of $\bar{\mathbf{P}}_{\alpha}\in\bar{\mathcal{P}}$ is characterized exactly in the same manner as in \Cref{dir_eig}, i.e.
\begin{equation*}
\mathcal{R}_{\pi,\mathbf{P}}(\boldsymbol\xi_{\alpha}) =1-\mathfrak{Re}(\vartheta_{\alpha}).
\end{equation*}

Therefore, we are now able to associate to each eigenvector $\boldsymbol\xi$ of $\mathbf{P}$ and generally to each eigenvector $\boldsymbol\xi_{\alpha}$ of a given $\bar{\mathbf{P}}_{\alpha}\in\bar{\mathcal{P}}$, a value $\omega$ characterizing its variation that we call \emph{frequency} expressed intuitively as 
\begin{equation}
\label{freq_def}
 \omega = 1- \mathfrak{Re}(\vartheta), \quad \omega \in [0,2].
\end{equation}

\begin{remark}
In the undirected setting, the random walk is reversible such that the Dirichlet energy associated to the random walk operator $\mathbf{P}$ is 
\begin{equation*}
\mathcal{D}_{\pi,\mathbf{P}}(\boldsymbol f)=\langle \boldsymbol f, \mathbf{L}_{\operatorname{RW}}\boldsymbol f\rangle_{\pi},
\end{equation*}
where $\mathbf{L}_{\operatorname{RW}}=\mathbf{I}-\mathbf{P}$. The random walk operator $\mathbf{P}$ is self-adjoint in $\ell^2(\mathcal{V},\pi)$. As a result, given an eigenvector $\boldsymbol\xi$ of $\mathbf{P}$ associated to an eigenvalue $\vartheta$, the associated Rayleigh quotient is 
\begin{equation*}
\mathcal{R}_{\pi,\mathbf{P}}(\boldsymbol\xi) =1-\vartheta.
\end{equation*}
Therefore, the variation of the eigenvectors of $\mathbf{P}$ is directly related to their respective eigenvalues. A similar result holds when the Fourier basis is based on the combinatorial Laplacian $\mathbf{L}=\mathbf{D}-\mathbf{W}$ thanks to the following identity
\begin{equation*}
\mathcal{D}_{\pi,\mathbf{P}}(\boldsymbol f)= \langle \boldsymbol f, \mathbf{L}_{\mathrm{RW}}\boldsymbol f\rangle_\pi = \langle \boldsymbol f, \mathbf{L}\boldsymbol f\rangle.
\end{equation*}
As $\mathbf{L}$ is a symmetric semi-definite operator, its eigenvalues are non-negative and real. Hence the Rayleigh quotient of an eigenvector $\boldsymbol{\phi}$ of $\mathbf{L}$ associated to an eigenvalue $\lambda$ is $\mathcal{R}_{\pi,\mathbf{P}}(\boldsymbol\phi) = \lambda$, which is the definition of frequency considered in~\cite{shuman2013emerging}.
\end{remark}

\subsubsection{On the content of subspaces associated to the random walk}
\label{mfrws}

The eigenvalues of an ergodic transition matrix $\mathbf{P}$ associated to a random walk on a strongly connected directed graph $\mathcal{G}$ are either real or come as complex-conjugate  pairs $(\vartheta,\bar{\vartheta})\in\mathbb{C}^2, \vartheta=\alpha+\boldsymbol i\beta$. 
% In the general case, a mono-frequency subspace can be composed of eigenspaces corresponding to any number of pairs of complex-conjugate eigenvalues and possibly one real eigenvalue. 
In the following, we examine how these individual eigenspaces can be understood with respect to each other. 

Let us first consider pairs of eigenspaces related to complex-conjugate eigenvalues $\alpha\pm\boldsymbol i\beta$. Assuming that the eigenspaces have dimension equal to 1, the eigenvectors can be chosen as complex-conjugate~$\boldsymbol\xi,\bar{\boldsymbol\xi}\in \mathbb{C}^N$. These eigenvectors share the same frequency $1-\alpha$ and can be seen as a generalization of the Fourier modes at frequencies $\omega$ and $-\omega$ in one-dimensional classical Fourier analysis. One can also define the corresponding real-valued Fourier modes 
\begin{equation*}
\boldsymbol\xi^{cos}=\frac{\boldsymbol\xi+\bar{\boldsymbol\xi}}{2},\qquad \boldsymbol\xi^{sin}=\frac{\boldsymbol\xi-\bar{\boldsymbol\xi}}{2\boldsymbol i},
\end{equation*}
which generalize the cosine and sine functions to graph signals. 
% A graph mono-frequency subspace spanned by only two complex-conjugate eigenvectors is thus analogous to the subspace spanned by the cosine and sine functions at a given frequency in classical Fourier analysis. 
As with the discrete Fourier transform, there are also frequencies for which the $\omega$ and $-\omega$ frequency subspaces are the same (the zero frequency and possibly the $1/2$ frequency in discrete signal processing). The case of real eigenvalues in graph signal processing can be seen as a generalization of these.

When the eigenspaces have dimension greater than 1, the situation is similar to the multidimensional classical Fourier analysis setting where the classical frequency is a vector $\vec{\omega}=[\omega_1,\ldots,\omega_n]^\top$. In graph signal processing, the $\omega$ frequency would then be a generalization of the norm of the classical frequency vector $\|\vec{\omega}\|$.

% Let us now consider two eigenvalues with the same real part but different non-conjugate imaginary parts. 
Continuing the analogy with classical Fourier analysis, in more than one dimensions Fourier modes with a given frequency $\omega$ are all oscillations at the same frequency but with \emph{different orientations}. We argue that the case in graph signal processing of eigenvalues with the same real part but different non-conjugate imaginary parts may be a generalization of \emph{different orientations}.
% can be seen as linear combinations of, say, vertical and horizontal Fourier modes with the same frequency, which are two \emph{different orientations}. 
This is illustrated in section \ref{torus_sec} on the directed toroidal graph.

\subsection{Random walk graph filters}
We define a random walk graph filter as a LSI filter with respect to the reference operator $\mathbf{P}$:
\begin{equation*}
\mathbf{H} = \sum_{t=0}^T h_t\mathbf{P}^{t}.
\end{equation*}
Assuming that $\mathbf{P}$ is diagonalizable, we can represent a graph filter $\mathbf{H}$ as a linear combination of the spectral projectors $\{\mathbf{E}_{\vartheta_k}\}_{k=1}^m$ associated with $\mathbf{P}$, as defined in~\eqref{fil_sp_proj}:
\begin{equation*}
\mathbf{H}=\sum_{k=1}^m\gamma_k\mathbf{E}_{\vartheta_k},\quad \gamma_k\in\mathbb{C},\quad k=1,\ldots,m.
\label{eq:spectraldomain_filter}
\end{equation*}

The filter is characterized in the Fourier domain by its spectral response coefficients, or Fourier coefficients, $\gamma_k$, $k=1,\ldots,m$. As in classical signal processing, the application of a filter to a signal in the Fourier domain is obtained by multiplying the filter's Fourier coefficients with the signal's corresponding Fourier coefficients. A particularity here is that a filter may only have one Fourier coefficient for a Fourier subspace of dimension larger than 1 while the signal is generally described by as many coefficients as the dimension of the subspace.

From a filter design point of view, we may also be interested in filters defined by a frequency response. This is a special case of \eqref{eq:spectraldomain_filter} where the spectral coefficients only depend on the frequency: $\gamma_k = h(\omega_k) =  h(1-\real{\theta_k})$, where $h:[0,2]\mapsto\mathbb{C}$ is the frequency response of the filter.

Let us define the projector associated to the frequency $\omega$ as
\begin{equation}
 \mathbf{S}_{\omega}=\sum_{\vartheta : 1-\mathfrak{Re}(\vartheta)=\omega} \mathbf{E}_{\vartheta}
  \label{eq:mono_freq_proj}
\end{equation}
It is a projector that projects on the sum of the eigenspaces corresponding to the frequency $\omega$. A filter based on a frequency response $h$ can be expressed as
\begin{equation}
 \label{rwgf_freq_resp}
 \mathbf{H}_{\boldsymbol\omega}=\sum_{\omega\in\boldsymbol{\omega}} h(\omega)\mathbf{S_{\omega}}, \quad\omega\in[0,2]. 
\end{equation}

A major advantage of such a formulation is that it allows one to shift, contract or dilate a filter along the frequency axis by simply applying these operations to its frequency response. Unfortunately, this is not possible for all graph filters because a filter may have different Fourier coefficients $\gamma_k$ for different eigenspaces even when they share the same frequency $\omega$.

\subsection{Fourier analysis on finite groups: a graph signal perspective}
\label{fa_fg}

In order to show the consistency of our Fourier analysis with respect to  traditional signal processing, we depict the directed counterparts of two well-known objects: the cycle graph associated with the cyclic group $(\mathbb{Z}/ n\mathbb{Z})$ and the toroidal graph that is the direct product of cyclic groups, i.e. $\mathbb{T}=(\mathbb{Z}/ n_1\mathbb{Z})\oplus\dots\oplus (\mathbb{Z}/ n_r\mathbb{Z})$ \cite{terras1999fourier}.

\subsubsection{Fourier analysis on the directed cycle graph}
\label{fa_dcg}
A directed cycle graph $\mathcal{C}_N$ is a graph with $N$ vertices containing a single cycle through all the vertices and were all the edges are directed in the same direction.
The directed cycle graph $\mathcal{C}_N$ is represented by its adjacency matrix $\mathbf{C}_N$ which is circulant
\begin{equation*}
\mathbf{C}_N=
\begin{bmatrix}
 0&0&\ldots&0&1\\
 1&0&\ldots&0&0\\
 0&\ddots&\ddots&\vdots&\vdots\\
 \vdots&\ddots&\ddots&0&0\\
 0&\ldots&0&1&0
\end{bmatrix}
\end{equation*}

As the out-degree matrix of $\mathcal{C}_N$ is $\mathbf{D}=\mathbf{I}_N$, the transition matrix for the random walk random walk $\mathcal{X}$ on $\mathcal{C}_N$ is $\mathbf{P}=\mathbf{C}_N$.
% we define a random walk $\mathcal{X}$ on $\mathcal{C}_N$ whose transition matrix is characterized by its adjacency matrix, namely $\mathbf{P}=\mathbf{C}_N$. 
$\mathbf{P}$ is diagonalizable, i.e. $\mathbf{P}=\boldsymbol\Xi\boldsymbol\Theta\boldsymbol\Xi^{-1}$ with $\boldsymbol\Theta=\operatorname{diag}\{\vartheta_1,\cdots,\vartheta_N\}\in\mathbb{C}^{N\times N}$ the eigenvalue matrix  
where

\begin{equation*}
\vartheta_k=e^{2\pi i (k-1)/N},\, k=1,\dots,N.
\end{equation*}
and $\boldsymbol\Xi=[\boldsymbol\xi_1,\cdots,\boldsymbol\xi_N]\in\mathbb{C}^{N\times N}$ is the discrete Fourier basis where 
\begin{equation*}
\boldsymbol\xi_k=\frac{1}{\sqrt{N}}[1,\vartheta_k^1,\dots,\vartheta_k^{N-1}]^{\top}\quad k=1,\cdots,N.
\end{equation*}

The random walk $\mathcal{X}$ is irreducible and periodic. As our frequency analysis only deals with ergodic random walks, we need $\mathcal{X}$ to be aperiodic. In order to overcome the periodicity issue, we consider a $\gamma$-lazy random walk $\tilde{\mathcal{X}}$ characterized by its transition matrix $\tilde{\mathbf{P}}_{\gamma}\in\tilde{\mathcal{P}}$ defined at section \ref{lazyrw_sec}. For such a lazy random walk, the stationary distribution $\pi$ exists and is constant. As a consequence, the Rayleigh quotient associated to a given eigenvector $\boldsymbol\xi_k$ is given by 

\begin{equation*}
\mathcal{R}^{(\gamma)}_{\pi,\mathbf{P}}(\boldsymbol\xi_k)=(1-\gamma)\bigg[1-\cos\bigg(\frac{2\pi(k-1)}{N}\bigg)\bigg]=\tilde \omega^{(\gamma)}_k\,,\quad k=1,\cdots,N.
\end{equation*}

Since the $\gamma$-lazy random walks $\tilde{\mathcal{X}}$ have the same eigenspaces as the non-lazy random walk, taking the limit $\gamma\rightarrow 0$ of these Rayleigh quotients allows us to define a frequency for the eigenspaces of the non-lazy random walk too.

\begin{equation*}
\omega_k=\bigg[1-\cos\bigg(\frac{2\pi(k-1)}{N}\bigg)\bigg],\quad k=1,\cdots,N.
\end{equation*}

Thus, the ordering of the frequencies $\{\omega_k\}_{k=1}^N$ associated with the eigenvectors $\{\boldsymbol\xi_k\}_{k=1}^N$ by the Rayleigh quotient coincides exactly with the classical signal processing approach. However, the actual values of the frequencies are warped by the function $x\mapsto 1-\cos(x)$ compared to their classical counterparts. As the warping could in theory be reversed by modifying our definition of frequency, the notion of frequency that we propose is consistent with the classical notion of frequency.

Finally, the trick we used here with taking the limit $\gamma\rightarrow 0$ of $\gamma$-lazy random walks could be used for any graph where the random walk is irreducible but periodic. This effectively allows us to define a meaningful notion of frequency for any graph that is strongly connected.
% By taking the limit $\gamma\rightarrow 0$, we could define frequencies for non-lazy random walks. More generally, this trick may be used to define a notion of frequency for any graph where the random walk is irreducible but periodic.
% More generally, our Fourier analysis works perfectly with irreducible random walks by taking the same kind of limit.

\subsubsection{Fourier analysis on the directed toroidal graph}
\label{torus_sec}
A directed toroidal graph $\mathcal{T}_{m,n}$ is the Cartesian product of the directed cycle graphs $\mathcal{C}_m$ and $\mathcal{C}_n$, namely $\mathcal{T}_{m,n}=\mathcal{C}_m\square \mathcal{C}_n$. We introduce necessary definitions for the understanding of the section.

\begin{definition}[\cite{kaveh2013optimal}]
Let $\mathcal{G}=(\mathcal{U},\mathcal{E})$ and $\mathcal{H}=(\mathcal{V},\mathcal{F})$ be two graphs with respective vertex sets $\mathcal{U}=(u_1,\dots,u_n)$ and $\mathcal{V}=(v_1,\dots,v_m)$.
The Cartesian product of $\mathcal{G}$ and $\mathcal{H}$ is the graph $\mathcal{G} \square \mathcal{H}$ with vertex set $\mathcal{U} \times \mathcal{V}$ in which two vertices $x=(u_i,v_j)$ and $y=(u_p,v_q)$ are adjacent if and only if either $u_i=u_p$ and $(v_j,v_q)\in \mathcal{F}$ or $v_j=v_q$ and $(u_i,u_p)\in \mathcal{E}$. 
\end{definition}

\begin{definition}[\cite{kaveh2013optimal}]
Let $\mathcal{G}=(\mathcal{U},\mathcal{E})$ and $\mathcal{H}=(\mathcal{V},\mathcal{F})$ be two graphs. The adjacency matrix of the Cartesian product $\mathcal{G} \square \mathcal{H}$, denoted $\mathbf{A}_{\mathcal{G} \square \mathcal{H}}$ is
\begin{equation*}
\mathbf{A}_{\mathcal{G} \square \mathcal{H}}=\mathbf{A}_{\mathcal{G}} \otimes\mathbf{I}_{|\mathcal{V}|}+\mathbf{I}_{|\mathcal{U}|}\otimes \mathbf{A}_{\mathcal{H}}. 
\end{equation*}
where $\otimes$ is the Kronecker product symbol. 

\end{definition}

\begin{lemma}
Suppose $\lambda_1,\cdots,\lambda_n$ are eigenvalues of $\mathbf{A}_{\mathcal{G}}$ and $\mu_1,\cdots,\mu_m$ are eigenvalues of $\mathbf{A}_{\mathcal{H}}$. Then the eigenvalues of $\mathbf{A}_{\mathcal{G} \square \mathcal{H}}$ are all $\lambda_i+\mu_j$ for $1\leq i\leq n$ and $1\leq j \leq m$. Moreover, if $u$ and $v$ are eigenvectors for $\mathbf{A}_{\mathcal{G}}$ and $\mathbf{A}_{\mathcal{H}}$ with eigenvalues $\lambda$ and $\mu$ respectively, then the vector $\mathfrak{w}=u\otimes v$ is an eigenvector of $\mathbf{A}_{\mathcal{G} \square \mathcal{H}}$ with eigenvalue $\lambda+\mu$.
\end{lemma}

\subsubsection*{Directed toroidal graph}
Let $\mathcal{T}_{m,n}=\mathcal{C}_m \square \mathcal{C}_n$ be the directed toroidal graph characterized by its adjacency matrix $\mathbf{A}_{\mathcal{T}_{m,n}}=\mathbf{C}_m \otimes\mathbf{I}_{n}+\mathbf{I}_{m}\otimes \mathbf{C}_n$. The eigenvalues of $\mathbf{C}_m$ are $\lambda_k=e^{2i\pi(k-1)/m}$ for $k \in\ldbrack 1,m \rdbrack$ and the eigenvectors are denoted by $[\boldsymbol\phi_1,\ldots,\boldsymbol\phi_m]$. Identically, the eigenvalues of $\mathbf{C}_n$ are $\mu_\ell=e^{2i\pi(\ell-1)/n}$ for $\ell \in \ldbrack 1,n \rdbrack$ and the eigenvectors are denoted by $[\boldsymbol\psi_1,\ldots,\boldsymbol\psi_n]$. The directed toroidal graph $\mathcal{T}_{m,n}$ is a directed 2-regular graph. Consequently, the out-degree matrix of $\mathcal{T}_{m,n}$ is $\mathbf{D}_{\mathcal{T}_{m,n}}=2\mathbf{I}_{m\times n}$. 
We thus define a random walk $\mathcal{X}$ on $\mathcal{T}_{m,n}$ with a diagonalizable transition matrix $\mathbf{P}=\mathbf{D}_{\mathcal{T}_{m,n}}^{-1}\mathbf{A}_{\mathcal{T}_{m,n}}$. We specify its spectral properties in the following lemma. 

\begin{lemma}
Let $\mathcal{T}_{m,n}$ be a directed toroidal graph and $\mathcal{X}$ be a random walk associated with the diagonalizable transition matrix $\mathbf{P}$. The directed toroidal graph $\mathcal{T}_{m,n}$ is a directed 2-regular graph. Consequently, the spectrum of $\mathbf{P}$ is
\begin{equation*}
\sigma(\mathbf{P})=\bigg\{\zeta_{i,j}=\frac{\lambda_i+\mu_j}{2},\, i \in \ldbrack 1,m \rdbrack, j \in \ldbrack 1,n \rdbrack \bigg\}, 
\end{equation*}
with associated eigenvectors

\begin{equation*}
\bigg\{\boldsymbol\phi_i\otimes\boldsymbol\psi_j, i \in \ldbrack 1,m \rdbrack, j \in \ldbrack 1,n \rdbrack \bigg\}.
\end{equation*}
\end{lemma}

For the cycle graph, we observed that our notion of frequency was consistent with the classical notion of frequency used in signal processing up to some monotonous transform. This is only approximately true for the toroidal graph. Indeed, it is usually possible to find pairs of eigenvectors of the toroidal graph for which the frequency order differs compared to the classical case, although it does not seem to occur when the frequency difference is large enough. An example where classical frequencies and graph frequencies behave slightly differently is presented below.

% In two dimensions, the classical frequency of a wave form $(x,y)\mapsto \exp(i \omega_x x + \omega_y y)$ associated with the frequencies $\omega_x$ and $\omega_y$ in the $x$ and $y$ directions would be $\omega = \sqrt(\omega_x^2 + \omega_y^2)$. 
% which is notably different from $\omega = 1-(\cos\omega_x +\cos\omega_y)/2$, which is verified by the . 

\begin{figure}
\begin{center}
\includegraphics[width=0.65\columnwidth]{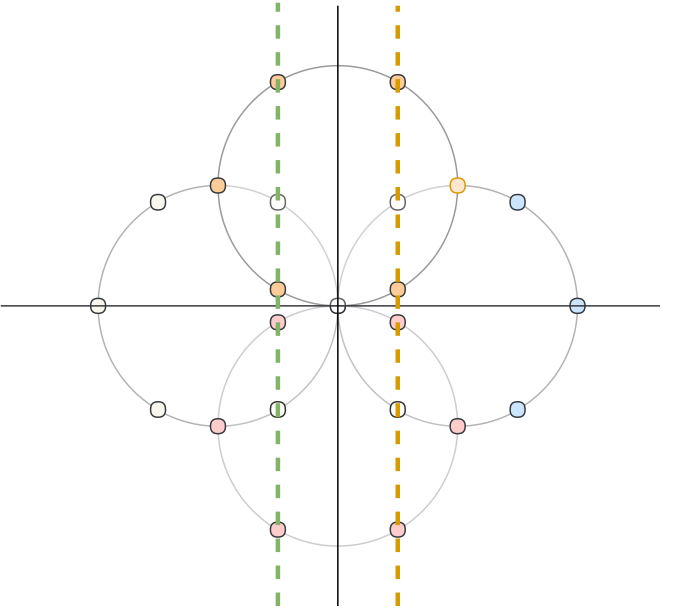}
\end{center}
 \caption{Eigenvalue distribution of the random walk matrix of a directed toroidal graph $\mathcal{T}_{6,4}$. Eigenvalues on the dotted green line: $\mathfrak{Re}(\lambda^{g})=-0.25, \forall \lambda^{g}\in\boldsymbol\Lambda_g$. Eigenvalues on the dotted orange line: $\mathfrak{Re}(\lambda^{o})=0.25, \forall \lambda^{o}\in\boldsymbol\Lambda_o$.}
 \label{fig:eig_torus}
 \end{figure}
 
 As in \cref{fa_dcg}, the associated random walk $\mathcal{X}$ is irreducible and periodic. In order to overcome this periodicity issue, we consider a $\gamma$-lazy random walk $\tilde{\mathcal{X}}$ with transition matrix $\tilde{\mathbf{P}}_{\gamma}\in\tilde{\mathcal{P}}$.
 
 Figure \ref{fig:eig_torus}
 illustrates the eigenvalue distribution of the transition matrix of a directed toroidal graph. As we can appreciate, the distribution of eigenvalues is remarkable. We denote by $\boldsymbol\Lambda_{g}$, the set of eigenvalues located on the green dotted line and $\boldsymbol\Lambda_{o}$ the set of eigenvalues located on the orange dotted line. Sets $\boldsymbol\Lambda_{g}$, and $\boldsymbol\Lambda_{o}$, include respectively eigenvalues having the same real part and different imaginary parts i.e. $\mathfrak{Re}(\boldsymbol\Lambda_{g})=-0.25$ and $\mathfrak{Re}(\boldsymbol\Lambda_{o})=0.25$. In order to illustrate the insights discussed in section \ref{mfrws}, we exhibit in figure~\ref{fig:eig_tore_graph} the two-dimensional representation of some eigenvectors from the transition matrix of a directed toroidal graph $\mathcal{T}_{36,60}$ corresponding to a graph frequency $\omega=(7-\sqrt{5})/8$. The eigenvalues associated with these eigenvectors have identical real parts but different non conjugate imaginary parts. Analytically, the Rayleigh quotients of the selected eigenvectors are equal, as explained in section~\ref{mfrws}. As we can see in figure~\ref{fig:eig_tore_graph}, the selected eigenvectors are all (discretized) complex exponential waveforms with what looks like identical or very similar frequencies. The main difference between them appears to be their orientation, which seems related to the different imaginary parts of the eigenvalues. On a side note, we also observe that the graph frequencies of these 4 eigenvectors are identical but this is not exactly true for their classical frequencies, which are equal to $0.260$, $0.269$, $0.260$ and $0.269$ from top to bottom respectively.

  % This orientation difference is associated to the influence of the imaginary parts of the eigenvalues of the selected eigenvectors. 

\begin{figure}
  \centering
  \includegraphics[]{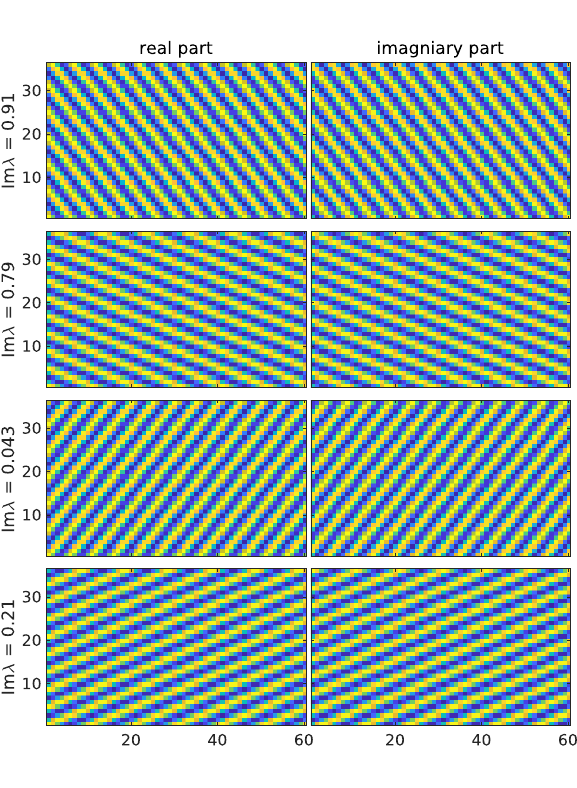}
  \caption{Two-dimensional representation of the eigenvectors for the toroidal graph $\mathcal{T}_{36,60}$ with a graph frequency equal to $(7-\sqrt{5})/8$. Each row shows one eigenvector, with its real part on the left hand side and its imaginary part on the right hand side. Among the 8 eigenvectors with this frequency, we only show those corresponding to eigenvalues with a non-negative imaginary part (shown on the left hand side).}
  \label{fig:eig_tore_graph}
\end{figure}

% \begin{figure}
%  \includegraphics[width=0.90\columnwidth]{torusimnew.pdf}
%  \caption{Two-dimensional representation of the eigenvectors $\boldsymbol\xi_{28,49}$ and $\boldsymbol\xi_{28,50}$ from the transition matrix of the directed toroidal $\mathcal{T}_{54,36}$. The axis is the label of a vertex in $\mathbf{C}_m$ (resp. $\mathbf{C}_m $) in the $x$-direction (resp. $y$-direction).}
% \label{eig_tore_graph}
%  \end{figure}

\section{Applications}
\label{sect:applications_one}
This section is devoted to illustrating our theoretical framework through examples of semi-supervised learning and signal modeling on directed graphs. It turns our that  in the following applications, we are able to work either in the Hilbert space $\ell^2(\mathcal{V})$ or
$\ell^2(\mathcal{V},\pi)$ where $\pi$ is the unique stationary distribution of an ergodic random walk. Furthermore, by the results obtained in section~\ref{lino_hilsp}, we define the isometries $\varphi$ and $\varphi^{-1}$ going from $\ell^2(\mathcal{V})$ to $\ell^2(\mathcal{V},\pi)$. 
\begin{equation}
 \label{isometric_op_inv}
 \varphi:\boldsymbol f\mapsto\boldsymbol\Pi^{-1/2}\boldsymbol f,\quad\varphi^{-1}:\boldsymbol f\mapsto\boldsymbol\Pi^{1/2}\boldsymbol f,\quad \forall \boldsymbol f\in\ell^2(\mathcal{V},\pi).
\end{equation}

For the sake of simplicity, we will always map signals living in  $\ell^2(\mathcal{V},\pi)$ to $\ell^2(\mathcal{V})$, so that we can work in $\ell^2(\mathcal{V})$ where the inner product coincides with the one between vectors in $\mathbb{C}^N$. In practice, given a graph signal $\boldsymbol f$ measured in the real world. We must first decide in which space the signal lives. This is a somewhat arbitrary choice between the signal living in a space associated with the counting measure or in a space associated with the stationary measure $\pi$ of an ergodic random walk. If we decide that the signal lives in $\ell^2(\mathcal{V},\pi)$, we then convert the signal into $\ell^2(\mathcal{V})$ using the application $\varphi^{-1}$, we perform all the necessary operations in $\ell^2(\mathcal{V})$ and then go back to the  $\ell^2(\mathcal{V},\pi)$ space using $\varphi$. In $\ell^2(\mathcal{V})$, the operator $\mathbf{T}$ defined in~\cref{rw_ldeux} corresponds to the random walk operator and $\mathcal{L}$ is the operator such that $\langle \boldsymbol f,\mathcal{L} \boldsymbol f\rangle$ is the Dirichlet energy.

\subsection{Semi-supervised learning on directed graph via $\ell_2$-regularization}
\label{sec:semi_sup_l_two_norm}
We discuss the semi-supervised learning problem on directed graphs with a regularization term of type $\ell_2$. The following problem aims to show the efficiency and relevance of the Dirichlet energy \eqref{dir_nrj} as a regularization term for signal processing on directed graphs.

\subsubsection{Definition of the problem}
Let $\mathcal{G}=(\mathcal{V},\mathcal{E})$ be a strongly connected directed graph where $\mathcal{V}=(v_1,\ldots,v_N)$ is the vertex set, $\mathcal{E}$ is the edge set and with cardinality $|\mathcal{V}|=N$. 
The directed graph $\mathcal{G}$ is represented by its adjacency matrix $\mathbf{W}\in\mathbb{R}_{+}^{N\times N}$. Let $\mathcal{X}$ be the associated random walk on $\mathcal{G}$. The random walk $\mathcal{X}$ is represented by its random walk operator $\mathbf{P}\in\mathbb{R}_{+}^{N\times N}$ with unique stationary distribution $\pi$.

We also introduce $\mathcal{G}_{\operatorname{sym}}$ the symmetrized version of the directed graph $\mathcal{G}$. The undirected graph $\mathcal{G}_{\operatorname{sym}}$ is represented by its adjacency matrix $\mathbf{W}_{\operatorname{sym}}=(\mathbf{W}+\mathbf{W}^{\top})/2\in\mathbb{R}_{+}^{N\times N}$. Let $\mathcal{X}_{\operatorname{sym}}$ be the associated random walk on $\mathcal{G}_{\operatorname{sym}}$. The random walk $\mathcal{X}_{\operatorname{sym}}$ is represented by its random walk operator $\mathbf{P}_{\operatorname{sym}}\in\mathbb{R}_{+}^{N\times N}$ with unique stationary distribution $\pi_{\operatorname{sym}}$.

Let $y:\mathcal{V}\rightarrow\{-1,1\}$ be a ``label'' function defined on the vertex set $\mathcal{V}$. From the function $y$, we obtain the graph signal $\boldsymbol y$ defined as follows
\begin{equation*}
\boldsymbol{y}=[y(v_1),\ldots,y(v_N)]^{\top}.
\end{equation*}

Given the values $y(v_k)_{\{1\leq k\leq N\}}$ on only a subset of labeled vertices $\mathcal{U}\subset\mathcal{V}$, the aim is to estimate the labels of the remaining unlabeled vertices. 
This is the semi-supervised learning problem on graphs~\cite{zhou2005learning,belkin2004regularization,zhu2005semi}. We introduce the formulation of the problem in the next section.  

% , and the classical approach \cite{belkin2004regularization} is to solve the following problem.

\subsubsection{Considered approaches}
The standard approach with $\ell_2$ regularization is based on the following optimization problem:
\begin{equation}
\label{semi_sup_graph_gnl}
\underset{\boldsymbol{f}}{\operatorname{argmin}}\big\{c\|\mathbf{M}_{l}(\boldsymbol{f}-\boldsymbol{y})\|^2+\varrho_1\|\mathbf{M}_{u}\boldsymbol f\|^2+\varrho_2\mathcal{S}(\boldsymbol{f})\big\},  \quad c,\varrho_1,\varrho_2 > 0.
\end{equation}
The term $\|\mathbf{M}_{l}(\boldsymbol{f}-\boldsymbol{y})\|^2$ is the data fidelity term with $\mathbf{M}_{l}=\operatorname{diag}\{m_1,\ldots,m_N\}$ a diagonal matrix where $m_i=\mathbb{1}_{v_i\in\mathcal{U}}$, for all $i=1\ldots,N$ (hence 0 on vertices with unknown labels). In the second term, $\|\mathbf{M}_{u}\boldsymbol f\|^2$, $\mathbf{M}_u$ selects the unlabeled vertices and is defined as $\mathbf{M}_u=\mathbf{I}_N-\mathbf{M}_l$. This second term is a Tikhonov-like regularization term. The third term, $\mathcal{S}(\boldsymbol f)$, is the variational regularization term derived from the graph
 \begin{equation*}
 \mathcal{S}(\boldsymbol f)=\langle \boldsymbol f,\mathbf{X}\boldsymbol f\rangle, \quad \mathbf{X}\in\mathbb{R}^{N\times N}.
 \end{equation*}

The semi-supervised learning on graphs problem is formulated in the same manner for the directed graph $\mathcal{G}$ or its symmetrized version $\mathcal{G}_{\operatorname{sym}}$, except that the variational regularization term $\mathcal{S}(\boldsymbol f)$ differs. 
In all these situations, for learning two classes, the solution is  obtained by constraining the previous continuous problem of having labels taking values only in $\{+1,-1\}$; a relaxed version of the problem is to solve the continuous optimization problem and take the sign of its solution. 
We now compare three methods for the directed case and the symmetrized case. 

\paragraph{(a) Semi-supervised learning on $\mathcal{G}$, the directed graph:}
\begin{case}
\item The first method is a version of eq.~\eqref{semi_sup_graph_gnl}, that was previously investigated by~\cite{zhou2005learning}. Assuming the graph signal lives in $\ell^2(\mathcal{V})$, the optimization problem is

\begin{equation}
\label{semi_sup_dir_zhu}
\underset{\boldsymbol{f}}{\operatorname{argmin}}\big\{c\|\mathbf{M}_{l}(\boldsymbol f-\boldsymbol y)\|^2+c\|\mathbf{M}_{u}\boldsymbol f\|^2+\varrho_2\langle \boldsymbol f,\mathcal{L} \boldsymbol f\rangle\big\},  \quad c,\varrho_2 > 0.
\end{equation}
with $\mathcal{L}$ the directed normalized Laplacian \eqref{eq:dir_norm_laplacian}. The term $\langle \boldsymbol f,\mathcal{L} \boldsymbol f\rangle$ is the Dirichlet energy in $\ell^2(\mathcal{V})$.

\item  For this method, we also use the formulation~\eqref{semi_sup_dir_zhu}, except that we now assume that the graphs signal lives in $\ell^2(\mathcal{V},\pi)$. Hence, we define $\tilde{\boldsymbol y} \in\ell^2(\mathcal{V})$ as follows 
\begin{equation*}
\tilde{\boldsymbol y}=\varphi^{-1}(\boldsymbol y)=\boldsymbol\Pi^{1/2}\boldsymbol y, \quad \boldsymbol y\in\ell^2(\mathcal{V},\pi).
\end{equation*}
The optimization problem is therefore 
\begin{equation}
\label{semi_sup_dir_zhu_pi}
\underset{\tilde{\boldsymbol f}}{\operatorname{argmin}}\big\{c\|\mathbf{M}_{l}(\tilde{\boldsymbol f}-\tilde{\boldsymbol y})\|^2+c\|\mathbf{M}_{u}\tilde{\boldsymbol f}\|^2+\varrho_2\langle \tilde{\boldsymbol f},\mathcal{L} \tilde{\boldsymbol f}\rangle\big\},  \quad c,\varrho_2 > 0.
\end{equation}

\Cref{semi_sup_dir_zhu_pi} is the same as \cref{semi_sup_dir_zhu} except for $\boldsymbol y$ which is now $\tilde{\boldsymbol y}$. The optimization problem~\eqref{semi_sup_dir_zhu_pi} can be rewritten as
\begin{equation}
\label{semi_sup_dir_re}
\underset{\boldsymbol{f}}{\operatorname{argmin}}\big\{c\|\mathbf{M}_{l}(\boldsymbol f-\boldsymbol y)\|_{\pi}^2+c\|\mathbf{M}_{u}\boldsymbol f\|_{\pi}^2+\varrho_2\langle \boldsymbol f,\mathbf{L}_{\operatorname{RW}} \boldsymbol f\rangle_{\pi}\big\},  \quad c,\varrho_2 > 0.
\end{equation}
with $\mathbf{L}_{\operatorname{RW}}$ the random walk Laplacian on directed graphs~\eqref{eq:rw_dir_laplacian}. The term $\langle \boldsymbol f,\mathbf{L}_{\operatorname{RW}} \boldsymbol f\rangle_{\pi}$ is then the Dirichlet energy~\eqref{dir_nrj}. 

\item In the framework of "Discrete Signal Processing on Graphs"~\cite{sandryhaila2013discrete,sandryhaila2014discrete}, the semi-supervised learning problem for directed graphs is formulated as follows
\begin{equation}
\label{semi_sup_moura}
\underset{\boldsymbol{f}}{\operatorname{argmin}} \big\{c\|\mathbf{M}_{l}(\boldsymbol{f}-\boldsymbol{y})\|^2+\varrho_2\|\boldsymbol f-\mathbf{W}^{norm}\boldsymbol f\|_2^2\big\}, \quad c,\varrho_2 > 0,
\end{equation}
 with $\mathbf{W}^{norm}$, a normalized version of the the adjacency  matrix $\mathbf{W}$ so that its spectral norm is equal to one. 
 \end{case}

\paragraph{Semi-supervised learning on $\mathcal{G}_{\operatorname{sym}}$, the symmetrized graph:} 
We consider the same methods applied now on $\mathcal{G}_{\operatorname{sym}}$. 
\begin{case}
\item  Given the graph signal $\boldsymbol y\in\ell^2(\mathcal{V})$, the optimization problem is
\begin{equation}
\label{semi_sup_dir_zhu_sym}
\underset{\boldsymbol{f}}{\operatorname{argmin}}\big\{c\|\mathbf{M}_{l}(\boldsymbol f-\boldsymbol y)\|^2+c\|\mathbf{M}_{u}\boldsymbol f\|^2+\varrho_2\langle \boldsymbol f,\mathcal{L}_{\operatorname{sym}} \boldsymbol f\rangle\big\},  \quad c,\varrho_2 > 0.
\end{equation}
with $\mathcal{L}_{\operatorname{sym}}$ the normalized Laplacian defined as
\begin{equation*}
\mathcal{L}_{\operatorname{sym}}=\mathbf{I}-\boldsymbol\Pi_{\operatorname{sym}}^{1/2}\mathbf{P}_{sym}\boldsymbol\Pi_{\operatorname{sym}}^{-1/2}, \quad \boldsymbol\Pi_{\operatorname{sym}}=\operatorname{diag}\{\pi_{sym}(v_1),\ldots,\pi_{sym}(v_N)\}.
\end{equation*} 
\item  The optimization problem is  
\begin{equation}
\label{semi_sup_dir_zhu_pi_sym}
\underset{\tilde{\boldsymbol f}}{\operatorname{argmin}}\big\{c\|\mathbf{M}_{l}(\tilde{\boldsymbol f}-\tilde{\boldsymbol y})\|^2+c\|\mathbf{M}_{u}\tilde{\boldsymbol f}\|^2+\varrho_2\langle \tilde{\boldsymbol f},\mathcal{L}_{\operatorname{sym}} \tilde{\boldsymbol f}\rangle\big\},  \quad c,\varrho_2 > 0.
\end{equation}
This can be rewritten as
\begin{equation}
\label{semi_sup_dir_re_sym}
\underset{\boldsymbol{f}}{\operatorname{argmin}}\big\{c\|\mathbf{M}_{l}(\boldsymbol f-\boldsymbol y)\|_{\pi}^2+c\|\mathbf{M}_{u}\boldsymbol f\|_{\pi}^2+\varrho_2\langle \boldsymbol f,\mathbf{L}_{\operatorname{RW,\operatorname{sym}}} \boldsymbol f\rangle_{\pi}\big\},  \quad c,\varrho_2 > 0.
\end{equation}
with $\mathbf{L}_{\operatorname{RW,\operatorname{sym}}}$ the random walk Laplacian defined as
\begin{equation*}
\mathbf{L}_{\operatorname{RW,\operatorname{sym}}}=\mathbf{I}-\mathbf{P}_{\operatorname{sym}}.
\end{equation*}
 \end{case}

\subsubsection{Solutions in closed form}
 The problem \eqref{semi_sup_graph_gnl} is quadratic and convex and therefore admits a closed form solution. 
\paragraph{Semi-supervised learning on $\mathcal{G}$}
\begin{case}
\item The solution of the optimization problem~\eqref{semi_sup_dir_zhu} is
\begin{equation}
\boldsymbol f^{*}=(\mathbf{I}+\gamma\mathcal{L})^{-1}\mathbf{M}_{l}\boldsymbol y,\quad \gamma=\varrho_2/c, \quad \boldsymbol f^{*}\in\ell^2(\mathcal{V}).
\end{equation}
\item The solution of the optimization problem~\eqref{semi_sup_dir_zhu_pi} is
\begin{equation}
\boldsymbol f^{*}=(\mathbf{I}+\gamma\mathbf{L}_{\operatorname{RW}})^{-1}\mathbf{M}_{l}\boldsymbol y,\quad \gamma=\varrho_2/c, \quad \boldsymbol f^{*}\in\ell^2(\mathcal{V},\pi).
\end{equation}
\item The solution of the optimization problem~\eqref{semi_sup_moura} is
\begin{equation*}
\boldsymbol f^{*}=\big(\mathbf{M}_{l}+\gamma \mathbf{R}_{M}^{-1}\big)\boldsymbol y,\quad \gamma=\varrho_2/c, \quad \boldsymbol f^{*}\in\ell^2(\mathcal{V}).
\end{equation*}
with $\mathbf{R}_{M}=(\mathbf{I}-\mathbf{W}^{norm})^{\top}(\mathbf{I}-\mathbf{W}^{norm})$.
\end{case}

\paragraph{Semi-supervised learning on $\mathcal{G}_{\operatorname{sym}}$}
\begin{case}

\item The solution of the optimization problem~\eqref{semi_sup_dir_zhu_sym} is
\begin{equation}
\boldsymbol f^{*}=(\mathbf{I}+\gamma\mathcal{L}_{\operatorname{sym}})^{-1}\mathbf{M}_{l}\boldsymbol y,\quad \gamma=\varrho_2/c, \quad \boldsymbol f^{*}\in\ell^2(\mathcal{V}).
\end{equation}

\item The solution of the optimization problem~\eqref{semi_sup_dir_zhu_pi_sym} is
\begin{equation}
\boldsymbol f^{*}=(\mathbf{I}+\gamma\mathbf{L}_{\operatorname{RW,sym}})^{-1}\mathbf{M}_{l}\boldsymbol y,\quad \gamma=\varrho_2/c, \quad \boldsymbol f^{*}\in\ell^2(\mathcal{V},\pi_{\operatorname{sym}}).
\end{equation}
\end{case}

Remember that in each case, we take the approximated solution for 2 classes which keeps only the sign of the solution as the label of the inferred class: $\boldsymbol y^* = \mathrm{sgn}(\boldsymbol f^{*})$. 
We summarize the different methods in Table~\ref{ssl_methods}. 

\begin{table}[H]
\begin{center}
\begin{tabular}{ c c c}
\toprule
         Method & $\mathcal{G}$ & $\mathcal{G}_{\operatorname{sym}}$ \\ 
       \midrule
        Method 1 & $\mathcal{L}$ & $\mathcal{L}_{\operatorname{sym}}$\\ 
        
       Method 2 & $ \mathbf{L}_{\operatorname{RW}}$ & $\mathbf{L}_{\operatorname{RW,sym}}$\\
       
       Method 3 & $\mathbf{R}_{M}$ ~\cite{sandryhaila2013discrete} &  \\
       \bottomrule
\end{tabular} 
\end{center}
 \caption{Table of the different operators associated with the semi-supervised learning methods according to $\mathcal{G}$ and $\mathcal{G}_{\operatorname{sym}}$.}
 \label{ssl_methods}
\end{table}

\subsubsection{Experiments}
\label{expe_polblog}
Let us consider the dataset of the political blogs of the 2004 US presidential campaign~\cite{adamic2005political}. The dataset consists of 1224  blogs, each associated with a political orientation, either republican or democrat. The hyperlinks between blogs provide the dataset with a graph structure $\mathcal{G}=(\mathcal{V},\mathcal{E})$, where each vertex $v\in\mathcal{V}$ is associated to a blog and an edge between two vertices $\{v_i,v_j\}$ indicates the presence of hyperlinks from blog $i$ to blog $j$. The political orientations of the blogs are modeled by a graph signal $\boldsymbol f_0=\{f_0(v_1),\cdots,f_0(v_N)\}$ where $f_0(v_i)=+1$ if blog $i$ has a Democrat orientation and $f_0(v_i)=-1$ if it is Republican. The directed graph $\mathcal{G}$ is not strongly connected, which is a problem with our framework that requires strongly connected graphs. As a result, we consider the largest strongly connected subgraph of $\mathcal{G}$, denoted by $\mathcal{G}'=(\mathcal{V}',\mathcal{E}')$ which is made up of $|\mathcal{V}'|=N'=793$ vertices (hence roughly 65\% of the vertex set $\mathcal{V}$) and its associated graph signal $\boldsymbol f_0'$. The respective performances of the approaches \cref{semi_sup_dir_zhu,semi_sup_dir_re,semi_sup_moura} and their symmetrized counterparts~\cref{semi_sup_dir_zhu_sym,semi_sup_dir_re_sym} are compared in~\cref{fig:ssl_lscc}.

% The respective performances of the approaches \eqref{ssl_sevi},~\eqref{ssl_moura},~\eqref{semi_sup_sym} are compared on the directed graph $\mathcal{G}'$. 

\subsubsection*{Numerical simulations}
\begin{figure}
 \includegraphics[width=1\columnwidth]
 {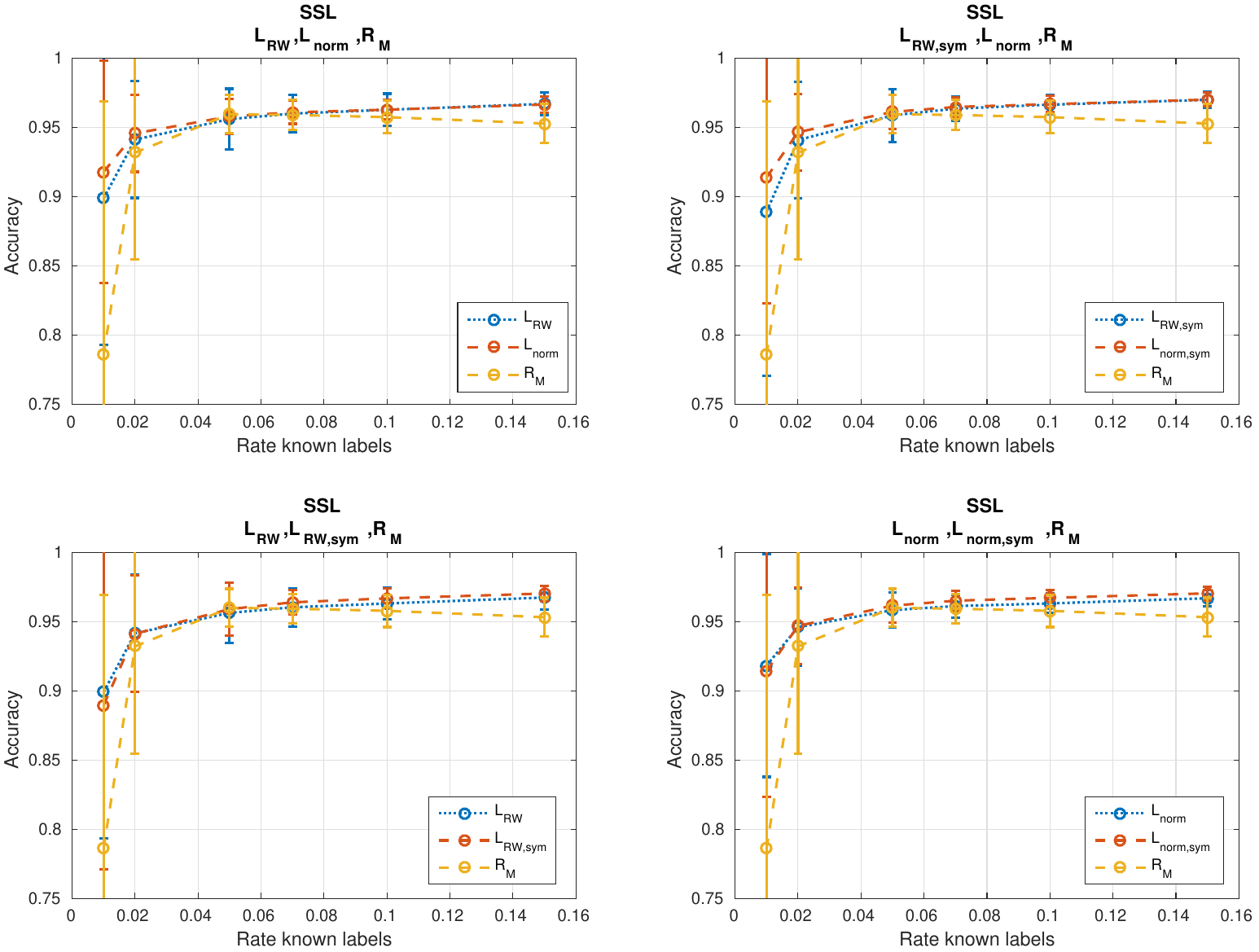}
 \caption{Semi-supervised learning on $\mathcal{G}'$ of the political blogs of the 2004 US presidential campaign. The notations $\operatorname{L_{norm}}$ and $\operatorname{L_{norm,sym}}$ correspond respectively to $\mathcal{L}$ and $\mathcal{L}_{\operatorname{sym}}$                                                                       .}
\label{fig:ssl_lscc}
 \end{figure}

Figure~\ref{fig:ssl_lscc} shows the performance of the semi-supervised approaches~\cref{semi_sup_dir_zhu,semi_sup_dir_re,semi_sup_moura,semi_sup_dir_zhu_sym,semi_sup_dir_re_sym}, in terms of the proportion of correct matching between $\boldsymbol y^*$ and $\boldsymbol f_0$ (or $\boldsymbol f'_0$). For each method, several values of $\gamma\in[0,10]$ were tested and we kept the one giving the best performance over 500 realizations. The lines displayed in~\cref{fig:ssl_lscc} are averaged estimation accuracy over 500 realizations and the error bars show the corresponding standard deviations.   

Firstly, we remark that for all proportions of known labels, the semi supervised approaches based on $\mathcal{L}, \mathbf{L}_{\operatorname{RW}}$ and its symmetrized counterparts $\mathcal{L}_{\operatorname{sym}}, \mathbf{L}_{\operatorname{RW,sym}}$ perform better than the semi-supervised approaches based on $\mathbf{R}_{M}$. On the top left of~\cref{fig:ssl_lscc}, we compare the semi-supervised approaches involving the operators $\mathcal{L}, \mathbf{L}_{\operatorname{RW}}$ and $\mathbf{R}_M$. The semi-supervised approach associated with $\mathcal{L}$ performs the best compared to the approaches involving $\mathbf{L}_{\operatorname{RW}},\mathbf{R}_M$ for low proportions of known labels. Furthermore, $\mathcal{L}$ performs slightly better than the approach involving $\mathbf{L}_{\operatorname{RW}}$. On the top right of~\cref{fig:ssl_lscc}, we compare the semi-supervised approaches involving the operators $\mathcal{L}_{\operatorname{sym}}, \mathbf{L}_{\operatorname{RW,sym}}$ and $\mathbf{R}_M$. Similarly the top left figure, the semi-supervised approach associated with $\mathcal{L}_{\operatorname{sym}}$ performs slightly better than the approach involving $\mathbf{L}_{\operatorname{RW}}$. On the bottom left of~\cref{fig:ssl_lscc}, we compare the semi-supervised approaches involving the operators $\mathbf{L}_{\operatorname{RW}},\mathbf{L}_{\operatorname{RW,sym}}$ and $\mathbf{R}_{M}$. The approach involving  $\mathbf{L}_{\operatorname{RW}}$ and $\mathbf{L}_{\operatorname{RW,sym}}$ perform quite similarly for any proportion of known labels. The same can be observed in the bottom right of~\cref{fig:ssl_lscc} comparing the semi-supervised approaches involving $\mathbf{L}_{\operatorname{RW}},\mathbf{L}_{\operatorname{RW,sym}}$ and $\mathbf{R}_{M}$. This performance comparison leads us to the following conclusions:
\begin{itemize}
\item The semi supervised learning approaches involving Dirichlet energies based upon $\mathcal{L}$ and $\mathbf{L}_{\operatorname{RW}}$ (and their symmetrized counterparts) yield better performance compared to the approach proposed by Sandryhaila and Moura~\cite{sandryhaila2013discrete} based upon $\mathbf{R}_{M}$.
Consequently, semi-supervised learning approaches on directed graphs with Dirichlet energies based on the random walk operator seem more appropriate. 
\item The Hilbert space where we assume the graph signal lives seems to have an influence on the performance. Here, we obtain better performance when assuming that the graph signal $\boldsymbol y$ belongs to $\ell^2(\mathcal{V})$ rather than $\ell^2(\mathcal{V},\pi)$, or equivalently when using $\mathcal{L}$ in the regularization term rather than $\mathcal{L}_{\operatorname{RW}}$.
% considering the graph signal $\boldsymbol y$ belongs to $\ell^2(\mathcal{V})$, or equivalently considering the operator $\mathcal{L}$ in the regularization term yields better performance than considering that the graph signal $\boldsymbol y$ belongs to $\ell^2(\mathcal{V},\pi)$ or equivalently considering that the operator $\mathcal{L}_{\operatorname{RW}}$ is in the regularization term.  
% \item Semi-supervised approaches on directed graphs are more advantageous than the approaches using symmetrized graphs. In the example, for a low rate of known labels, the performance based upon a Dirichlet energy involving $\mathcal{L}$ or $\mathbf{L}_{\operatorname{RW}}$ were slightly better than their symmetric counterparts $\mathcal{L}_{\operatorname{sym}}$ and $\mathbf{L}_{\operatorname{RW,sym}}$. 
\end{itemize}

\subsection{Signal modeling on directed graphs via filtering}
\label{sig_model_dg}

In this section, we consider a way to model the relationships between the values of a graph signal using a graph filter. The model is expressed as a graph filter that takes some values from the graph signal and reconstructs the other values.  A possible application of such a model could be the lossy compression of the signal, where knowing only the graph, the coefficients of the graph filter and a few signal values enable the reconstruction of the whole signal. The major difference with the previous application is that we assume the whole graph signal known to learn the filter.

\subsubsection{Problem formulation}

Let $\mathcal{G}=(\mathcal{V},\mathcal{E})$ be a directed graph with cardinality $|\mathcal{V}|=N$ and $\mu:\mathcal{V}\rightarrow \mathbb{R}_{+}$ a positive measure. The directed graph $\mathcal{G}$ is characterized by its adjacency matrix $\mathbf{W}\in\mathbb{R}_+^{N\times N}$. Let $f_0\in\ell^2(\mathcal{V},\mu)$ be a known graph function and $\boldsymbol f_0=[f_0(v_1),\ldots,f_0(v_N)]^{\top}$, the corresponding graph signal. Let $\boldsymbol y$ be a random sampling of $f_0$ according to some known sampling strategy, where for a given sampling realization $y_j = f_0(v_j)$ or $y_j = 0$ depending on whether vertex $j$ is selected in that sampling realization. The goal of the problem is to find a graph filter $\mathbf{H}$ that optimally reconstructs $\boldsymbol f_0$ from the samples in $\boldsymbol y$. More specifically, we consider graph filters expressed as finite polynomials of a reference operator $\mathbf{R}$ with order $K$. 
\begin{equation}
% \label{poly_sum}
 \mathbf{H}=\sum_{k=0}^K\theta_k\mathbf{R}^{k},\quad\theta_k\in\mathbb{R},\quad  k=0,\cdots,K. 
 \end{equation}
 The reconstruction of $\boldsymbol f_0$ from $\boldsymbol y$ is then
\begin{equation}
  \boldsymbol{\hat f_0} = \mathbf{H}\boldsymbol y = \sum_{k=0}^K\theta_k\mathbf{R}^{k}\boldsymbol y.
\end{equation}

In order to find an optimal filter $\widehat{\mathbf{H}}=\sum_{k=0}^K{\hat \theta}_{k}\mathbf{R}^{k}$, we propose to minimize the expected quadratic reconstruction error:
\begin{equation}
\label{signal_rec_prob_form}
{\boldsymbol {\hat\theta}} = \underset{\boldsymbol\theta=\{\theta_k\}_{k=0}^K\in\mathbb{R}^{K+1}}{\operatorname{argmin}}\mathbb{E}\Big\|\boldsymbol f_0-\sum_{k=0}^K\theta_k\mathbf{R}^k\boldsymbol y\Big\|_{\mu}^2 ,
\end{equation}

Finally, we only consider sampling strategies where samples in $\boldsymbol y$ are selected independently according to a given distribution $\boldsymbol \delta$, that is $ y_j = f_0(v_j)$ with probability $\delta_j$ and $y_j = 0$ otherwise. The random variables $y_j, j=1,\ldots,N$ can be expressed as $y_j=\varepsilon_j f_0(v_j)$,  where the $\varepsilon_j$ are independent and Bernoulli distributed with parameters $\delta_j$: $\varepsilon_j\sim \operatorname{Ber}(\delta_j)$. The random variable $\bar{\varepsilon}$ defined as
\begin{equation}
\label{mean_Y}
p=\frac{1}{N}\sum_{i=1}^N \varepsilon_i,
\end{equation}
is the proportion of known values of $\boldsymbol{f}_0$ collected in $\boldsymbol y$.

\subsubsection{Solution in closed form}
The solution of the problem~\eqref{signal_rec_prob_form} is the following 

\begin{equation}
\boldsymbol{\widehat{\theta}}=\mathbf{Z}^{-1}\mathbf{M}\mathbf{Q}^{\top}\boldsymbol f_0,
\label{signal_rec_prob_form_sol} 
\end{equation}
where $\mathbf{M}=\operatorname{diag}\{\mu(v_1),\cdots,\mu(v_N)\}$ is the diagonal matrix of the associated measure $\mu$, $\mathbf{Z}$ the matrix where each entry $Z_{k\ell}\in\mathbf{Z}$ is expressed as
\begin{equation*}
Z_{k\ell}=\operatorname{Tr}\big([\mathbf{R}^{k}]^{\top}\mathbf{M}\mathbf{R}^{\ell}\mathbb{E}\{\boldsymbol{y y^{\top}} \}\big),\quad \forall \{k,\ell\} \in \ldbrack 0,K \rdbrack^{2},
\end{equation*}
and $\mathbf{Q}=[\boldsymbol q_{0},\cdots,\boldsymbol q_{K}]\in\mathbb{R}^{N\times (K+1)}$ where each vector $\boldsymbol q_{j}$ is
\begin{equation*}
q_{j}=\mathbf{R}^{j}\mathbb{E}\{\boldsymbol y \},\quad\forall j\in \ldbrack 0,K \rdbrack.
\end{equation*}

Finally, according to the definition of $\boldsymbol y$, we have $\mathbb{E}\{ y_j \} = \delta_jf_0(v_j)$, $\mathbb{E} \{y_j^2\} = \delta_jf_0^2(v_j)$ and $\mathbb{E} \{y_iy_j\} = \delta_i\delta_jf_0(v_i)f_0(v_j)$ for $i\neq j$, which would allow us to replace the above terms $\mathbb{E}\{\boldsymbol y\}$ and $\mathbb{E}\{\boldsymbol y\boldsymbol y^\top\}$ with their values. For the experiments below, we will assume that the sampling strategy is unknown and estimate $\mathbb{E}(\boldsymbol y)$, $\mathbb{E}(\boldsymbol{y y^{\top}})$ empirically.

% Consequently, the resulting graph filter $\widehat{\mathbf{H}}$ is 
 
% \begin{equation*}
% \widehat{\mathbf{H}}=\sum_{k=0}^K\hat{\theta}_k\mathbf{R}^{k},\quad\hat{\theta}_k\in\mathbb{R},\quad  k=0,\cdots,K. 
% \end{equation*}
% \item 
%  \begin{equation*}
% \widehat{\mathbf{H}}=\sum_{j=1}^m\hat{\gamma}_j\mathbf{E}_j,
% \quad \hat{\gamma}_j \in \mathbb{C} ,\quad \mathbf{E}_j\in\mathbb{C}^{N\times N},\quad \forall j=1,\dots,m, 
% \end{equation*}
% \end{form}

\subsubsection{An exploration of modeling accuracy with different graph operators}
The reconstruction quality using the previous approach depends on several factors: the choice of reference operator $\mathbf{R}$, the order of the filter $K$ and the random sampling strategy defined by the Bernoulli parameters $\delta_k$. Here, we evaluate primarily the influence of $\mathbf{R}$ and compare two possible sampling strategies. As an example, we also consider the dataset of the political blogs of the 2004 US presidential campaign~\cite{adamic2005political}, described in section~\ref{expe_polblog}. 

As before, we will use the largest strongly connected directed subgraph $\mathcal{G}'$, represented by its adjacency matrix $\mathbf{W}$. On $\mathcal{G}'$, we consider the following operators
\begin{itemize}
\item $\mathbf{W}_{norm}$ the normalized version of the adjacency matrix $\mathbf{W}$ whose spectral norm is equal to one. 
\item The random walk operator $\mathbf{P}$ associated with an ergodic random walk $\mathcal{X}$ with unique stationary distribution $\pi$. 
\item $\bar{\mathbf{P}}$ the additive reversibilization of $\mathbf{P}$. 
\item $\mathbf{T}=\boldsymbol\Pi^{1/2}\mathbf{P}\boldsymbol\Pi^{-1/2}$ the operator similar to $\mathbf{P}$.
\item $\bar{\mathbf{T}}=\boldsymbol\Pi^{1/2}\bar{\mathbf{P}}\boldsymbol\Pi^{-1/2}$ the operator similar to $\bar{\mathbf{P}}$.
\item $\bar{\mathbf{P}}_{\alpha} \in\bar{\mathcal{P}}$ with $\bar{\mathcal{P}}=\big\{\bar{\mathbf{P}}_{\alpha}: \bar{\mathbf{P}}_{\alpha}=(1-\alpha)\mathbf{P}+\alpha\mathbf{P}^{*}\big| \alpha\in[0,1]\big\}$ 
\item  $\bar{\mathbf{T}}_{\alpha} \in\bar{\mathcal{T}}$ with $\bar{\mathcal{T}}=\big\{\bar{\mathbf{T}}_{\alpha}: \bar{\mathbf{T}}_{\alpha}=(1-\alpha)\boldsymbol\Pi^{1/2}\mathbf{P}\boldsymbol\Pi^{-1/2}+\alpha\boldsymbol\Pi^{1/2}\mathbf{P}^{*}\boldsymbol\Pi^{-1/2}\big| \alpha\in[0,1]\big\}$
\end{itemize}

 We also consider $\mathcal{G}_{\operatorname{sym}}'$, the symmetrized version of $\mathcal{G}'$. The undirected graph $\mathcal{G}_{\operatorname{sym}}'$ is represented by its adjacency matrix $\mathbf{W}_{\operatorname{sym}}=(\mathbf{W}+\mathbf{W}^{\top})/2$. On $\mathcal{G}_{\operatorname{sym}}'$, we define the following operators
 \begin{itemize}
 \item $\mathbf{W}_{\operatorname{sym}}^{norm}$ the normalized version of the adjacency matrix $\mathbf{W}_{\operatorname{sym}}$ whose spectral  norm is equal to one. 
 \item The random walk operator $\mathbf{P}_{\operatorname{sym}}$ corresponding to the ergodic random walk $\mathcal{X}_{\operatorname{sym}}$ with stationary distribution $\pi_{\operatorname{sym}}$.
 \item $\mathbf{T}_{\operatorname{sym}}=\boldsymbol\Pi_{\operatorname{sym}}^{1/2}\mathbf{P}_{\operatorname{sym}}\boldsymbol\Pi_{\operatorname{sym}}^{-1/2}$ the operator similar to $\mathbf{P}_{\operatorname{sym}}$.
 \end{itemize}
 
 Finally, we consider $\mathcal{G}$, the full graph that is not strongly connected. Consequently, we can not build directly a random walk operator on $\mathcal{G}$ unless we transform $\mathcal{G}$ into a strongly connected graph, as required by our framework. 
% Indeed, some vertices in $\mathcal{G}$ have an out-degree equal to zero and we cannot create a transition matrix directly from the adjacency matrix. 
We thus consider two classic approaches that make the graph strongly connected and the associated random walk ergodic:
\begin{casenew}
\item Rank-one perturbation: from the original adjacency matrix $\mathbf{W}$, we construct a new adjacency matrix $\mathbf{W}_{\epsilon}$ as follows
\begin{equation*}
 \mathbf{W}_{\epsilon}=\mathbf{W}+\epsilon \mathbf{J},
\end{equation*}
where $\mathbf{J}=\mathbf{11}^{\top}/N$ is a rank-one matrix and $\epsilon$ is small. The weak perturbation of the adjacency matrix $\mathbf{W}$ by the matrix $\mathbf{J}$ ensures that the random walk on $\mathcal{G}$ is ergodic with stationary measure $\pi_{\epsilon}$ and its associated transition matrix $\mathbf{P}_{\epsilon}$ is well-defined. For our experiments, we choose $\epsilon=10^{-4}$. 
\item Construction of the Google matrix of $\mathcal{G}$~\cite{langville2011google,RevModPhys.87.1261}. This is achieved in two steps. Firstly, we construct an adjacency matrix $\tilde{\mathbf{W}}$ from $\mathbf{W}$ by adding a weight one from all dangling nodes, that is nodes with no
out-edges towards all the nodes in the graph. From $\tilde{\mathbf{W}}$ we can construct the transition matrix $\mathbf{S}$. Secondly,
we define the Google matrix $\mathbf{P}_{G}$ as
\begin{equation*}
\mathbf{P}_{G}=(1-\gamma)\mathbf{S}+\gamma \mathbf{J}.
\end{equation*}
where $\gamma=0.85$~\cite{langville2011google,RevModPhys.87.1261}.
\end{casenew}
 
 We will compare the reconstruction accuracy using a proportion of correctly reconstructed labels obtained via

\begin{equation*}
\boldsymbol{\hat{f}}=\operatorname{sgn}(\widehat{\mathbf{H}} \boldsymbol y). 
\end{equation*}

\paragraph{Subproblem 1}
We solve the problem~\eqref{signal_rec_prob_form} by learning a polynomial graph filter with $K=10$ on $\mathcal{G}'$ and $\mathcal{G}_{\operatorname{sym}}'$ with each of the reference operators listed above. In this subproblem, the random variables $y_j,j=1,\ldots,N'$ are distributed according to one of the two following cases 
\begin{itemize}
\item uniform sampling: $y_j=\varepsilon_jf_{0}'(v_j)$ where $\varepsilon_j\sim \operatorname{Ber}(p)$ with $p$ the proportion of known labels.
\item $\pi$-weighted sampling: $y_j=\varepsilon_j f_{0}'(v_j)$ where $\varepsilon_j\sim \operatorname{Ber}(\alpha \pi_j)$, such that $\sum_j\mathbb{E}\{\varepsilon_j\}=pN'$ with $p$ the proportion of known labels.
\end{itemize}

The proportions of correctly reconstructed labels are measured for various $p$ values using 500 realizations of $\boldsymbol{y}$ to estimate its mean and covariance used in the solution~~\eqref{signal_rec_prob_form_sol}. 

We summarize the different cases in table~\ref{subproblem_one}.

\begin{table}[H]
\begin{center}
 \begin{adjustbox}{max width=\textwidth}
\begin{tabular}{c c c}
\toprule
\textbf{Case} & Distribution $\boldsymbol y$ & Reference operators\\[0.5cm]
\midrule

Case 1 & $y_j=\varepsilon_jf_{0}'(v_j)$,  $\varepsilon_j\sim \operatorname{Ber}(p)$ & $\mathbf{W}^{norm}, \mathbf{P},\bar{\mathbf{P}},\mathbf{T},\bar{\mathbf{T}}$,
$\mathbf{W}_{sym}^{norm}, \mathbf{P}_{sym},\mathbf{T}_{sym}$\\[0.5cm]

Case 2 & $y_j=\varepsilon_j f_{0}'(v_j)$,  $\varepsilon_j\sim \operatorname{Ber}(\alpha \pi_j)$, $\sum_j\mathbb{E}\{\varepsilon_j\}=pN'$ & $\mathbf{W}^{norm}, \mathbf{P},\bar{\mathbf{P}},\mathbf{T},\bar{\mathbf{T}},\mathbf{W}_{sym}^{norm}, \mathbf{P}_{sym},\mathbf{T}_{sym}$\\[0.5cm]
       \bottomrule
\end{tabular} 
\end{adjustbox}
\end{center}
 \caption{Summary table of the different cases of the subproblem 1.}
 \label{subproblem_one}
\end{table}

 \begin{figure}[H]
 \includegraphics[width=1.0\columnwidth]{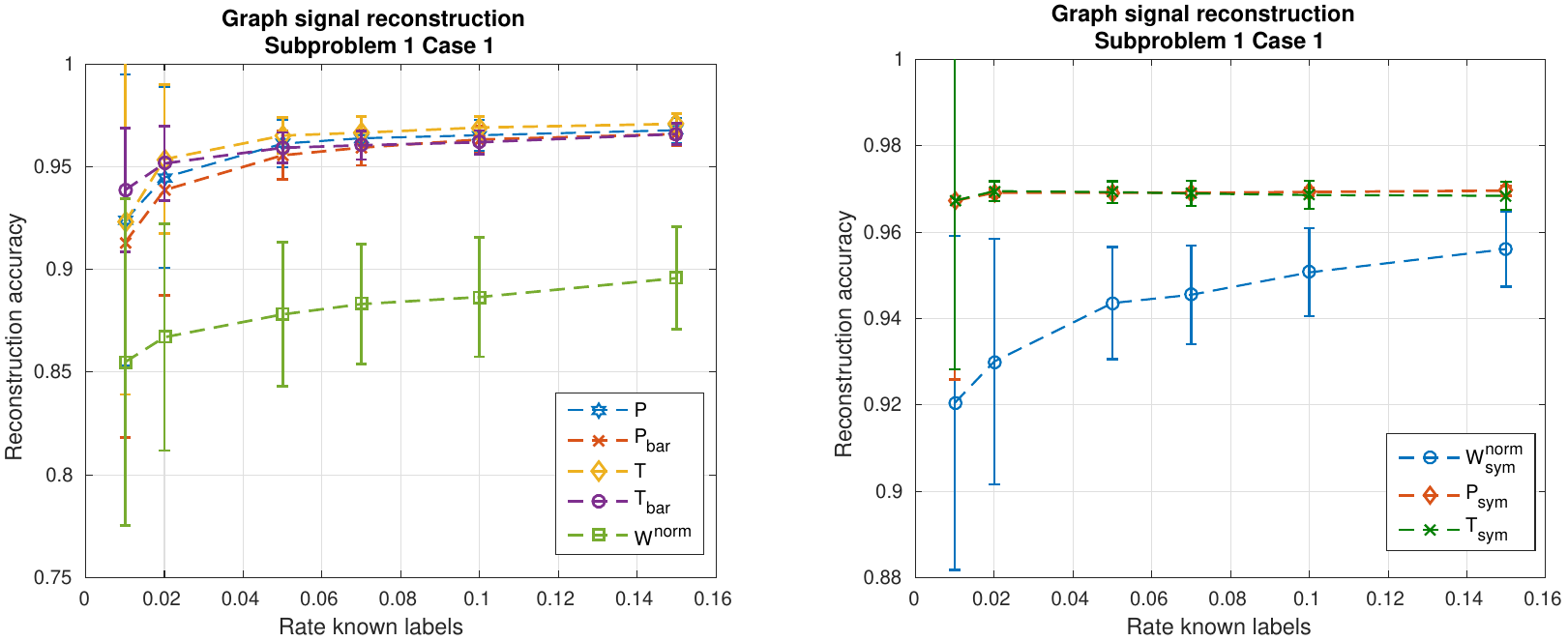}
 \caption{(Subproblem 1, Case 1) Left: Reconstruction of the graph signal $\boldsymbol f_0'$ on $\mathcal{G}'$. Right: Reconstruction of the graph signal $\boldsymbol f_0'$ on $\mathcal{G}_{sym}'$.The notation $\operatorname{P_{bar}}$ (resp. $\operatorname{T_{bar}}$) correspond to $\bar{\mathbf{P}}$ (resp.$\bar{\mathbf{T}}$).}
\label{fig:lscc_rec_unif}
 \end{figure}
 
\paragraph{Numerical simulations: Case 1}
We evaluate the reconstruction performance of the graph signal $\boldsymbol f_0'$ in figure \ref{fig:lscc_rec_unif}. 
For all proportions $p$ of known labels, the average rate of correctly reconstructed labels using a graph filter baseed on $\mathbf{P}$ or $\bar{\mathbf{P}}$ is significantly better than when using a graph filter based on $\mathbf{W}_{norm}$. This holds on the subgraph $\mathcal{G}'$ ands its symmetrized version $\mathcal{G}_{\operatorname{sym}}'$. Furthermore, the reconstruction performance using a filter based on $\mathbf{P}$ is slightly better than using a filter based on $\bar{\mathbf{P}}$. The reconstruction performance using a filter based on $\mathbf{T}$ is the best overall. For small proportions $p$ of known labels, the average rate of correctly reconstructed labels from filter based on $\bar{\mathbf{T}}$ is slightly better. We also notice that the reconstruction performance using a filter based on $\mathbf{T}_{\operatorname{sym}}$ is identical to using a filter based on $\mathbf{P}_{\operatorname{sym}}$ for all proportions $p$ of known labels.

\begin{figure}[H]
 \includegraphics[width=1.0\columnwidth]{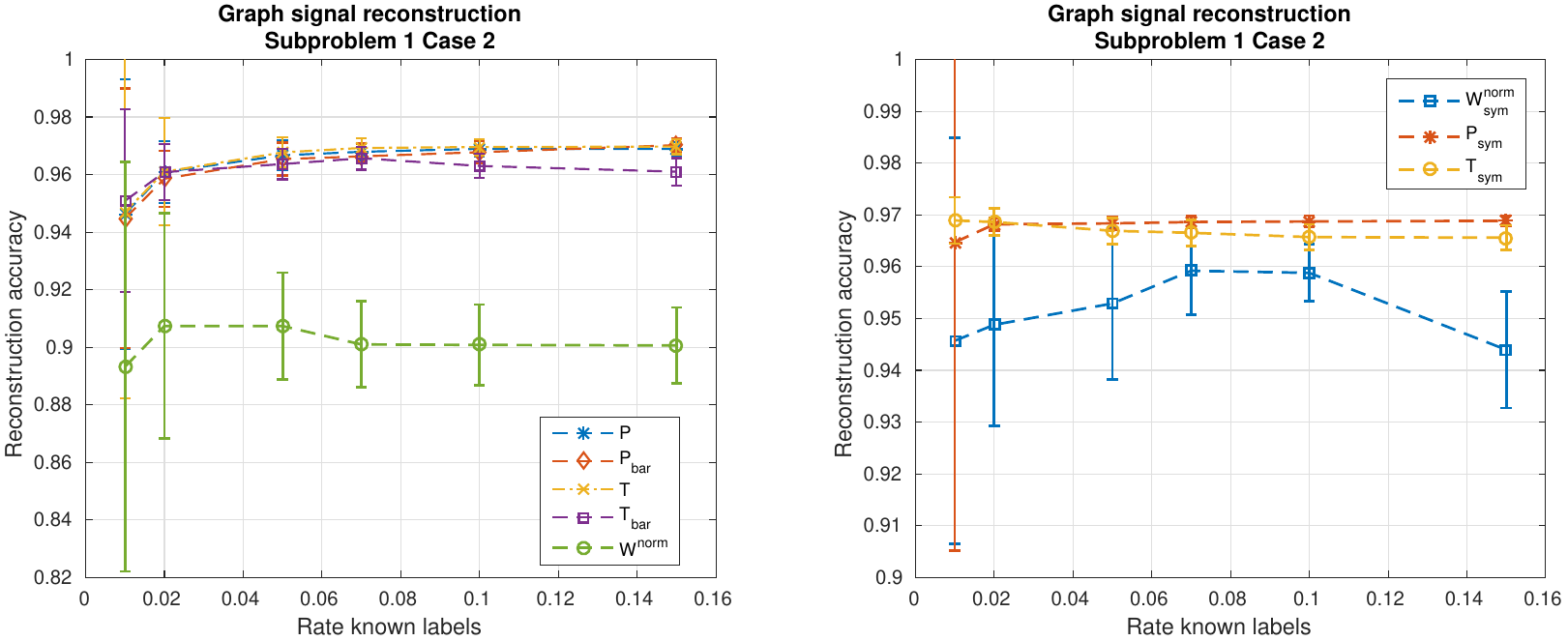}
 \caption{ (Subproblem 1, Case 2) Left: Reconstruction of the graph signal $\boldsymbol f_0'$ on the subgraph $\mathcal{G}'$. Right: Reconstruction of the graph signal $\boldsymbol f_0'$ on $\mathcal{G}_{sym}'$. The notation $\operatorname{P_{bar}}$ (resp. $\operatorname{T_{bar}}$) correspond to $\bar{\mathbf{P}}$ (resp.$\bar{\mathbf{T}}$).}
 \label{fig:rec_lscg_pii_f}
 \end{figure}
 
 \paragraph{Numerical simulations: Case 2}
The reconstruction performance is shown in figure \ref{fig:rec_lscg_pii_f}. Filters based on $\mathbf{W}_{norm}$ perform significantly worse than any other choice both on the subgraph $\mathcal{G}'$ or its symmetrized version $\mathcal{G}_{\operatorname{sym}}'$. Also, the reconstruction accuracy when using $\mathbf{W}_{norm}$ has a much larger variability and the average performance even decreases for larger $p$. Among other methods, the differences are not significant except for filters based on $\bar{\mathbf{T}}$ which perform slightly worse than others for large $p$.  
\\

\paragraph{Case 1 vs. Case 2}
The reconstruction performance using a filter based on $\mathbf{P}$ or $\bar{\mathbf{P}}$ is slightly better when the samples in $\boldsymbol y$ are selected according to a distribution proportional to the stationary distribution $\pi$ (case 2) than when the distribution is uniform (case 1). It is also the case when using a filter based on $\mathbf{T}$ or $\bar{\mathbf{T}}$. This is expected as vertices with larger $\pi_k$ values correspond to better-connected blogs, which are likely to have an influence on a larger number of other blogs. Although the reconstruction performance using a filter based on $\mathbf{W}_{norm}$ is better in case 2 than in case 1, the large variability and the poorer performance at large $p$ in case 2 suggest that these filters provide generally poorer models of the signal. The better reconstruction performance using a filter based on $\mathbf{T}$ or $\bar{\mathbf{T}}$ with respect to using a filter based on $\mathbf{P}$ or $\bar{\mathbf{P}}$ suggests that for this application it seems more suitable for the graph signal to live in $\ell^2(\mathcal{V})$ than $\ell^2(\mathcal{V},\pi)$. 
\\

\paragraph{Subproblem 2}
Here we consider the same problem~\eqref{signal_rec_prob_form} as in the previous section but now compare the performance using reference operators  $\bar{\mathbf{P}}_{\alpha}\in \bar{\mathcal{P}}$ and their equivalents $\bar{\mathbf{T}}_{\alpha}=\boldsymbol\Pi^{1/2}\bar{\mathbf{P}}_{\alpha}\boldsymbol\Pi^{-1/2}$.  The random variables $y_j$, $j=1,\ldots,N'$ are distributed according to one of the two following cases:
\begin{itemize}
\item uniform sampling: $y_j=\varepsilon_jf_{0}'(v_j)$ where $\varepsilon_j\sim \operatorname{Ber}(p)$ with $p$ the proportion of known labels.
 \item $\pi$-weighted sampling:  $y_j=\varepsilon_j f_{0}'(v_j)$ where $\varepsilon_j\sim \operatorname{Ber}(\alpha \pi_j)$, such that $\sum_j\mathbb{E}\{\varepsilon_j\}=pN'$ with $p$ the proportion of known labels.
\end{itemize}

We summarize the different cases in~\cref{subproblem_two}. 
\begin{table}[H]
\begin{center} \begin{adjustbox}{max width=\textwidth}
\begin{tabular}{c c c}
\toprule
\textbf{Case} & Distribution $\boldsymbol y$ & Reference operators\\[0.5 cm]
\midrule
 
Case 1 & $y_j=\varepsilon_jf_{0}'(v_j)$,  $\varepsilon_j\sim \operatorname{Ber}(p)$ & $\bar{\mathbf{P}}_{\alpha}\in \bar{\mathcal{P}},  \bar{\mathbf{T}}_{\alpha}\in\bar{\mathcal{T}}$ \\[0.5 cm]

Case 2 & $y_j=\varepsilon_j f_{0}'(v_j)$,  $\varepsilon_j\sim \operatorname{Ber}(\alpha \pi_j)$, $\sum_j\mathbb{E}\{\varepsilon_j\}=pN'$ & $\bar{\mathbf{P}}_{\alpha}\in \bar{\mathcal{P}}, \bar{\mathbf{T}}_{\alpha}\in\bar{\mathcal{T}}$\\[0.5 cm]
       \bottomrule
\end{tabular} 
\end{adjustbox}
\end{center}
 \caption{Summary table of the different cases of the subproblem 2.}
 \label{subproblem_two}
\end{table}

\begin{figure}
 \includegraphics[width=1.0\columnwidth]{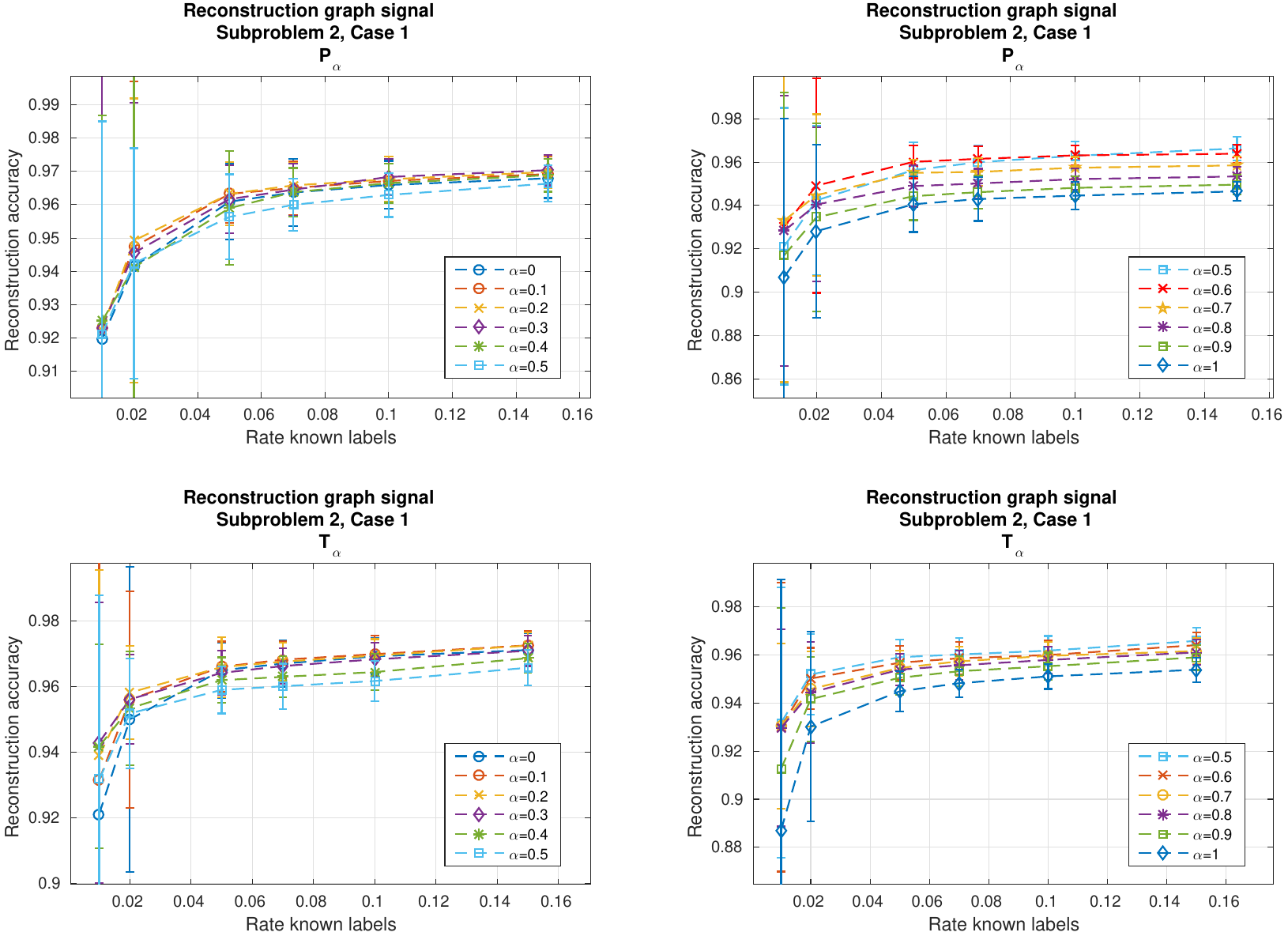}
 \caption{(Subproblem 2, Case 1) Up: Reconstruction of the graph signal $\boldsymbol f_0'$ on $\mathcal{G}'$, $\bar{\mathbf{P}}_{\alpha}\in \bar{\mathcal{P}}$. Down:  Reconstruction of the graph signal $\boldsymbol f_0'$ on $\mathcal{G}'$, $\bar{\mathbf{T}}_{\alpha}\in \bar{\mathcal{T}}$.}
 \label{fig:rec_signal_case_on_subp_two}
 \end{figure}
 
\paragraph{Numerical simulations: Case 1}

We evaluate the reconstruction performance of $\boldsymbol f_0'$ using  filters either based on $\bar{\mathbf{P}}_{\alpha}, \alpha\in[0,1]$ at the top of~\cref{fig:rec_signal_case_on_subp_two} or either based on $\bar{\mathbf{T}}_{\alpha}, \alpha\in[0,1]$ at the bottom of \cref{fig:rec_signal_case_on_subp_two}. We first consider the top of \cref{fig:rec_signal_case_on_subp_two}. For small proportions $p$ of known labels, we notice that the best reconstruction performance using a filter based on $\bar{\mathbf{P}}_{\alpha}$ is obtained for $\alpha\simeq 0.2$ while for higher proportions $p$, the best performance is obtained for $\alpha\simeq 0.3$. We notice poorer performance overall when $\alpha$ increases from $\alpha=0.5$ to $\alpha=1$. At the bottom of~\cref{fig:rec_signal_case_on_subp_two}, the reconstruction performances using a filter based on $\bar{\mathbf{T}}_{\alpha}, \alpha\simeq 0.2$ are the best except for the rate $p=0.01$. We also notice poorer performance overall when $\alpha$ increases from $\alpha=0.5$ to $\alpha=1$. 

  For all proportions $p$ of known labels, filters based on $\bar{\mathbf{T}}_{\alpha}$ appear to perform slightly better than filters based on $\bar{\mathbf{P}}_{\alpha}$. This again suggests that this graph signal $\boldsymbol f_0'$ is better represented as living in $\ell^2(\mathcal{V})$ than $\ell^2(\mathcal{V},\pi)$. Furthermore, the better performance when $\alpha\in(0.5,1)$ suggests that, for this signal, graph filters  $\bar{\mathbf{P}}_{\alpha}$ perform better when the convex combination of $\mathbf{P}$ and $\mathbf{P}^*$ involves more the random walk $\mathbf{P}$, i.e. $\alpha\in(0,0.5)$, than the time-reversed random walk $\mathbf{P}^{*}$.  
\\

 \begin{figure}
 \includegraphics[width=1.0\columnwidth]{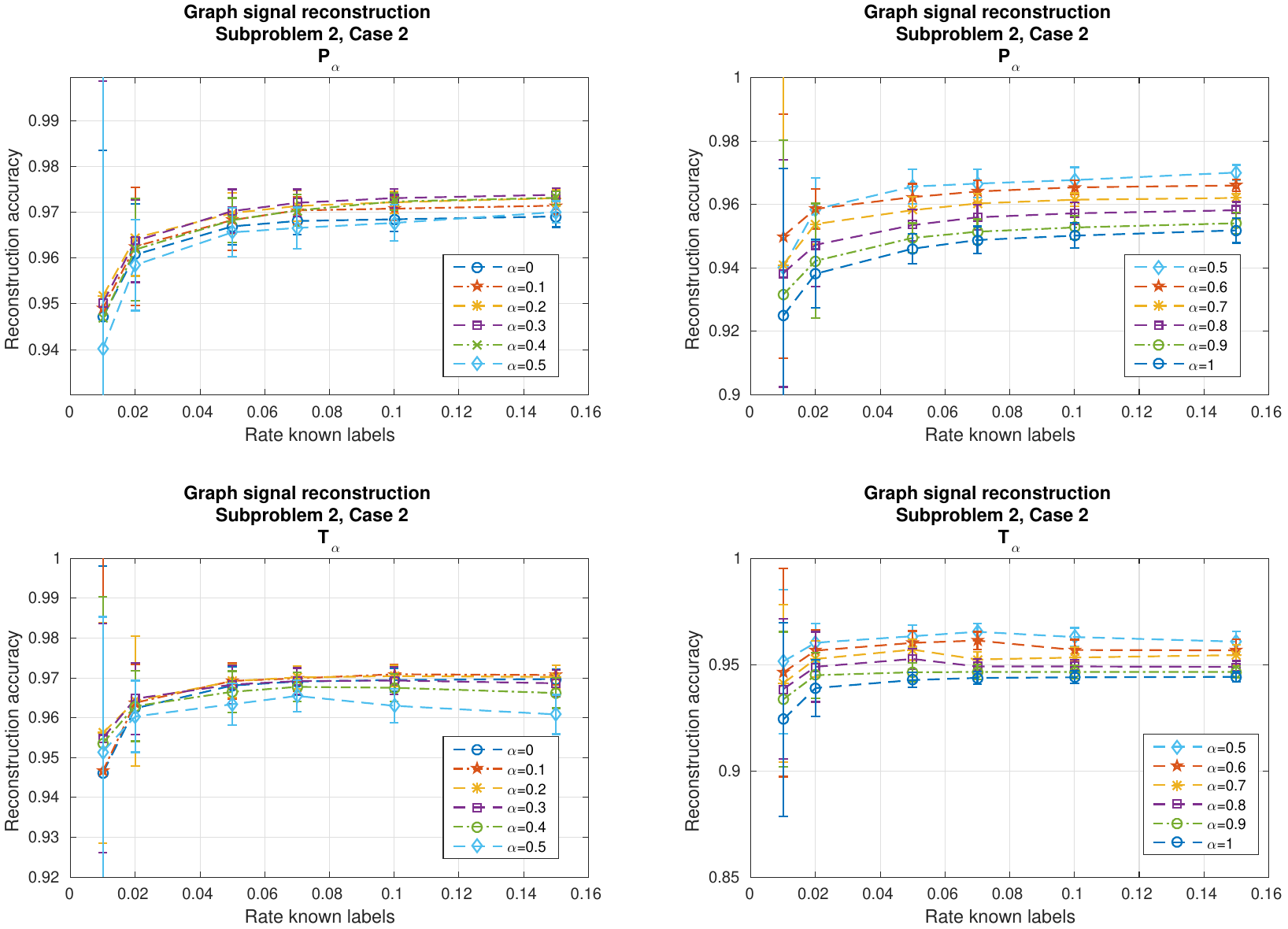}
 \caption{(Subproblem 2, case 2) Up: Reconstruction of the graph signal $\boldsymbol f_0'$ on $\mathcal{G}'$, $\bar{\mathbf{P}}_{\alpha}\in \bar{\mathcal{P}}$. Down:  Reconstruction of the graph signal $\boldsymbol f_0'$ on $\mathcal{G}'$, $\bar{\mathbf{T}}_{\alpha}\in \bar{\mathcal{T}}$.}
 \label{fig:rec_signal_case_two_subp_two}
 \end{figure}

 \paragraph{Numerical simulations: Case 2}
 We evaluate the reconstruction performance of $\boldsymbol f_0'$ using  filters either based on $\bar{\mathbf{P}}_{\alpha}, \alpha\in[0,1]$ at the top of~\cref{fig:rec_signal_case_two_subp_two} or based on $\bar{\mathbf{T}}_{\alpha}, \alpha\in[0,1]$ at the bottom of~\cref{fig:rec_signal_case_two_subp_two}. Let us consider the top of~\cref{fig:rec_signal_case_two_subp_two}. We notice that the reconstruction performances using a filter based on $\bar{\mathbf{P}}_{\alpha},\alpha\simeq 0.3$ are the best except for the rates $p=\{0.01,0.02\}$. We notice poorer performance overall when $\alpha$  increases from $\alpha=0.5$ to $\alpha=1$.  This suggests that, for this signal, graph filters based on $\bar{\mathbf{P}}_{\alpha}$ perform better when the convex combination of $\mathbf{P}$ and $\mathbf{P}^*$ involves more the forward random walk $\mathbf{P}$, i.e. $\alpha\in(0,0.5)$, than the backward random walk $\mathbf{P}^*$, i.e. $\alpha\in(0.5,1)$. At the bottom of~\cref{fig:rec_signal_case_two_subp_two}, the reconstruction performances using a filter based on $\bar{\mathbf{T}}_{\alpha}, \alpha\simeq 0.3$ are the best for all proportions $p$ of known labels. Similarly to the previous cases, we notice poorer performance overall when $\alpha$ increases from $\alpha=0.5$ to $\alpha=1$.

On this example, the reconstruction performance is better using filters based on  $\bar{\mathbf{P}}_{\alpha}$ than  $\bar{\mathbf{T}}_{\alpha}$. This suggests that the graph signal $\boldsymbol f_0'$  is best viewed as living in $\ell^2(\mathcal{V},\pi)$ compared to $\ell^2(\mathcal{V})$ when the samples in $\boldsymbol y$ are selected from a distribution proportional to $\pi$. Previously, we observed that it is more suitable for the graph signal $\boldsymbol f_0'$ to belong to $\ell^2(\mathcal{V})$ than $\ell^2(\mathcal{V},\pi)$ when the distribution in $\boldsymbol y$ is uniform. This suggests that it may be a good idea in practice to adapt the measure of the Hilbert space to the sampling distribution if it is known.
% As a result, the distribution of the graph signal has an influence in the reconstruction performance through the use of learned filters based either on $\bar{\mathbf{P}}_{\alpha}$ or $\bar{\mathbf{T}}_{\alpha}$.
 \\

\paragraph{Subproblem 3}
We now consider the same problem~\eqref{signal_rec_prob_form} on the whole graph $\mathcal{G}$. In this subproblem, the random variables $y_j,j=1\ldots,N$ are distributed as  $y_j=\varepsilon_jf_{0}(v_j), \varepsilon_j\sim \operatorname{Ber}(p)$. For each proportion of known labels $p$, the correct reconstruction rates are obtained by averaging over 500 realizations of $\boldsymbol{y}$.
We summarize the following case in \cref{subproblem_three}.
\\

\begin{table}
\begin{center}
\small
\begin{tabular}{c c c}
\toprule
\textbf{Case} & Distribution $\boldsymbol y$ & Reference operators\\[0.5 cm]
\midrule
Case 1 & $y_j=\varepsilon_jf_{0}'(v_j)$,  $\varepsilon_j\sim \operatorname{Ber}(p)$ & $\mathbf{W}^{norm},\mathbf{P}_{\epsilon},\mathbf{P}_{G},\bar{\mathbf{P}}_{\epsilon},\bar{\mathbf{P}}_{G},\mathbf{T}_{\epsilon},\mathbf{T}_{G},\bar{\mathbf{T}}_{\epsilon},\bar{\mathbf{T}}_{G}$\\[0.5 cm]
\bottomrule
\end{tabular} 
\end{center}
 \caption{Summary table of the case of the subproblem 3.}
 \label{subproblem_three}
\end{table}
\normalsize

 \paragraph*{Numerical simulations}
 The reconstruction performance is shown in~\cref{fig:rec_signal_case_one_subp_three}. For all $p$ values, the reconstruction performance using filters based on $\mathbf{W}_{norm}$ is significantly worse than with all other reference operators. Compared to the other reference operators, $\mathbf{W}_{norm}$ is the only one that can be used without modifying the graph to make it strongly connected. Still, its performance is always worse.
 
Among the other reference operators, we notice a clearly better performance generally when using the reversibilizations $\bar{\mathbf{P}}_{\epsilon}$ and $\bar{\mathbf{P}}_{G}$ compared to the non-reversible random walks ${\mathbf{P}}_{\epsilon}$ and ${\mathbf{P}}_{G}$. Furthermore, the best reconstruction performance is obtained using a filter based on $\bar{\mathbf{T}}_{G}$ and the reconstruction performance using a filter based on $\bar{\mathbf{T}}_{\epsilon}$ is identical to $\bar{\mathbf{T}}_{G}$ for higher proportions of known labels.  
This differs from the results on $\mathcal{G}'$ where both would perform similarly. 

Finally, the reconstruction performance does not depend as much on the approach we use to make the graph strongly connected. It seems though that the Google approach slightly outperforms the rank-one approach, both for the non-reversible random walks and their reversibilizations.
\\
  \begin{figure}
 \includegraphics[width=1\columnwidth]{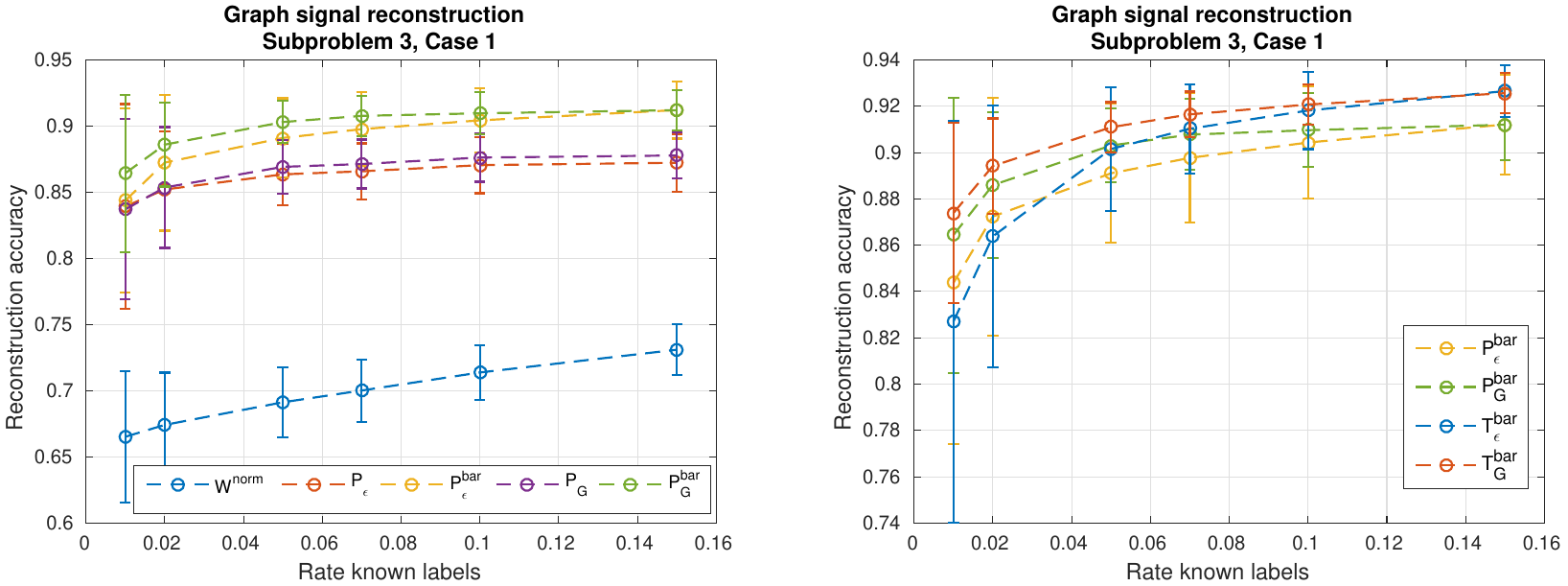}
 \caption{ (Subproblem 3, Case 1) Reconstruction of the graph signal $\boldsymbol f_0$ on $\mathcal{G}$.The notation $\operatorname{P_{\epsilon}^{bar}}$ (resp. $\operatorname{P_{G}^{bar}}$, $\operatorname{T_{G}^{bar}}$, $\operatorname{T_{\epsilon}^{bar}}$) corresponds to $\bar{\mathbf{P}}_{\epsilon}$ (resp. $\bar{\mathbf{P}}_{G}$, $\bar{\mathbf{T}}_{G}$, $\bar{\mathbf{T}}_{\epsilon}$).}
 \label{fig:rec_signal_case_one_subp_three}
 \end{figure}

\paragraph{Subproblem 4}
Here we consider the same problem~\eqref{signal_rec_prob_form} as in the previous section but now compare the performance using reference operators  $\bar{\mathbf{P}}_{\alpha}\in \bar{\mathcal{P}}$  and their respective equivalents $\bar{\mathbf{T}}_{\alpha}=\boldsymbol\Pi^{1/2}\bar{\mathbf{P}}_{\alpha}\boldsymbol\Pi^{-1/2}$. The random walk operator is built from the Google matrix approach (Approach 1). The random variables $y_j, j=1\ldots,N$ are distributed as  $y_j=\varepsilon_jf_{0}(v_j), \varepsilon_j\sim \operatorname{Ber}(p)$. We summarize the case in the~\cref{subproblem_four}.
\begin{table}[H]
\begin{center}
 \begin{adjustbox}{max width=\textwidth}
\begin{tabular}{c c c}
\toprule
\textbf{Case} & Distribution $\boldsymbol y$ & Reference operators\\[0.5cm]
\midrule
Case 1 & $y_j=\varepsilon_jf_{0}(v_j)$, $\varepsilon_j\sim \operatorname{Ber}(p)$ & $\bar{\mathbf{P}}_{\alpha}\in \bar{\mathcal{P}},  \bar{\mathbf{T}}_{\alpha}\in\bar{\mathcal{T}}$ \\[0.5 cm]
\bottomrule
\end{tabular} 
\end{adjustbox}
\end{center}
 \caption{Table of the case of the subproblem 4.}
 \label{subproblem_four}
\end{table}

  \begin{figure}
  \centering
 \includegraphics[width=1\columnwidth]{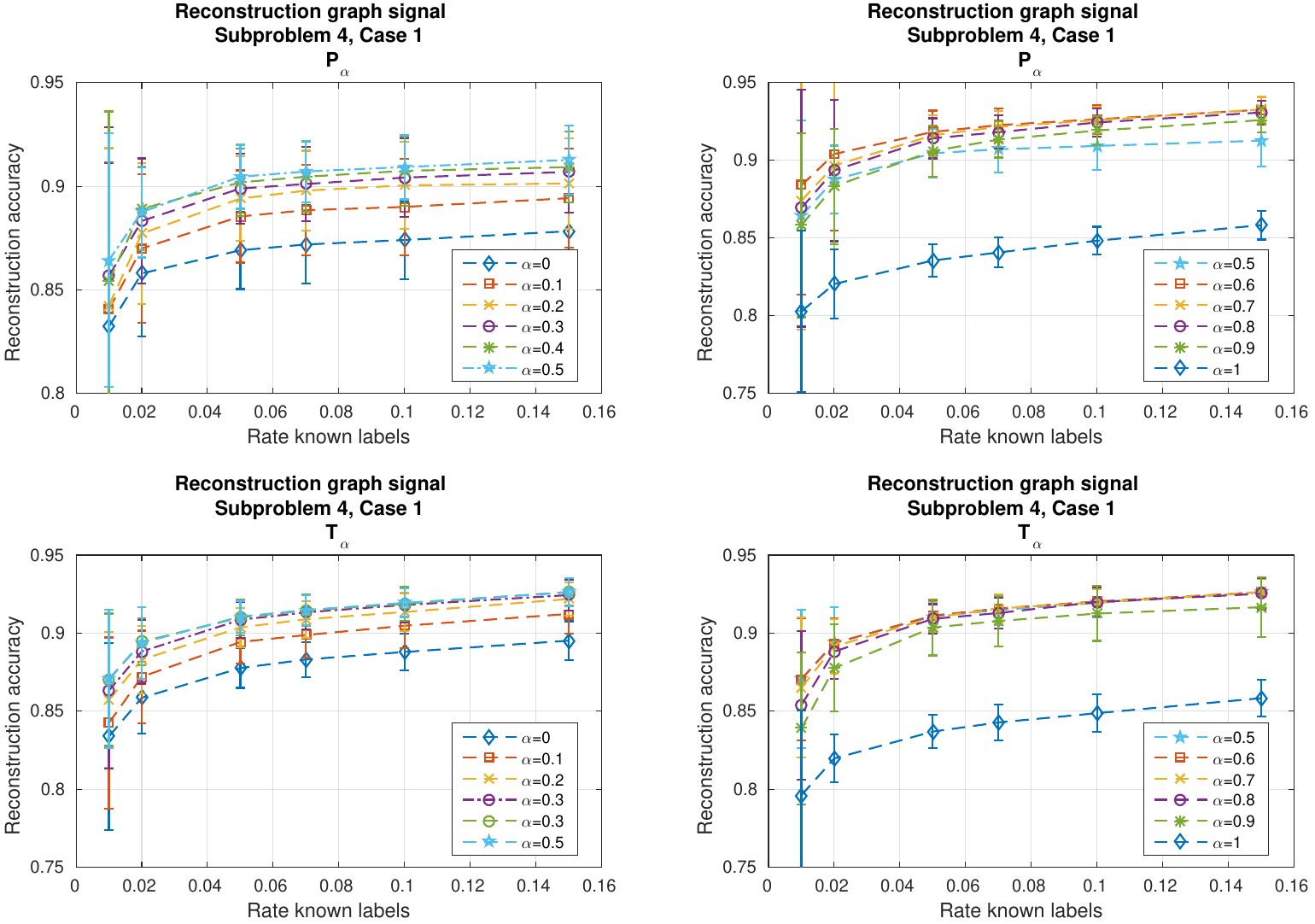}
 \caption{(Subproblem 4, case 1) Up: Reconstruction of the graph signal $\boldsymbol f_0$ on $\mathcal{G}$, $\bar{\mathbf{P}}_{\alpha}\in \bar{\mathcal{P}}$. Down:  Reconstruction of the graph signal $\boldsymbol f_0$ on $\mathcal{G}$, $\bar{\mathbf{T}}_{\alpha}\in \bar{\mathcal{T}}$.}
 \label{fig:rec_signal_case_one_subp_four}
 \end{figure}

\paragraph{Numerical simulations: Case 1}
We evaluate the reconstruction performance of $\boldsymbol f_0'$ using  filters either based on $\bar{\mathbf{P}}_{\alpha}, \alpha\in[0,1]$ at the top of \cref{fig:rec_signal_case_one_subp_four} or either based on $\bar{\mathbf{T}}_{\alpha}, \alpha\in[0,1]$ at the bottom of \cref{fig:rec_signal_case_one_subp_four}. For all proportions $p$ of known labels, the best reconstruction performance using a filter based on $\bar{\mathbf{P}}_{\alpha}$ is obtained for $\alpha\simeq 0.6$. We also notice that the reconstruction performance using a filter with $\alpha \in \{0.7,0.8\}$ is better than with $\alpha\simeq 0.5$. This suggests that, for this signal, graph filters based on $\bar{\mathbf{P}}_{\alpha}$ perform better when the convex combination of $\mathbf{P}$ and $\mathbf{P}^*$ involves more the time-reversed random walk $\mathbf{P}^*$ than the random walk $\mathbf{P}$ (i.e. when $\alpha > 0.5$). At the bottom of~\cref{fig:rec_signal_case_one_subp_four}, the best performance is obtained for a filter $\bar{\mathbf{T}}_{\alpha}$ with $\alpha\simeq 0.5$ overall. The performance is worse when $\alpha$ increases from $\alpha=0.5$ to $\alpha=1$. For all proportions $p$ of known labels, the reconstruction performance are better overall if the graph signal $\boldsymbol f_0$ is in $\ell^2(\mathcal{V})$ than in $\ell^2(\mathcal{V},\pi)$. In other words, the reconstruction performance using filters based on $\bar{\mathbf{T}}_{\alpha}$ is better than using filters based on $\bar{\mathbf{P}}_{\alpha}$. 

To summarize :
\begin{itemize}
\item Both the Hilbert space where we choose for the graph signal and the sampling distribution impact the performance. For the example presented here, we obtain better performance when we consider the Hilbert space whose measure corresponds to the sampling probability.
\item Graph filters based on the random walk operator yield better performance than those based on the adjacency matrix. 
\item The choice of the parameter $\alpha$ in the convex combination between $\mathbf{P}$ and $\mathbf{P}^{*}$ has an influence on the performance. For the example presented here, the optimal convex combination of $\mathbf{P}$ and $\mathbf{P}^{*}$ (resp. $\mathbf{T}$ and $\mathbf{T}^{*}$) always involves more the forward random walk ($\mathbf{P}$) than the backward random walk $\mathbf{P}^{*}$, except in one case.
\end{itemize}

\section{Multi-resolution analyses on directed graphs}
 \label{sect:mra}
 In the previous sections, we proposed a Fourier like basis on directed graphs, as the set of random walk eigenvectors and determined a frequency analysis by studying their smoothness. This allows the definition of low-pass or band-pass filters, hence we are now able to construct multiresolution analyses of functions over directed graphs. In a first instance, we propose wavelet frames made of analysis and synthesis graph filter banks. This multi-scale construction is closely related to the diffusion polynomial frames~\cite{maggioni2008diffusion} and spectral graph wavelets~\cite{hammond2011wavelets}. In a second instance, we propose a critically sampled wavelet construction on directed graphs generalizing the diffusion wavelets framework~\cite{coifman2006diffusion}. For the sake of simplicity, we will present our theoretical constructions in the $\ell^2(\mathcal{V})$ space. The linear transformation $\varphi$ defined in \cref{isometric_op} and its inverse may be used to express the constructed wavelets and scaling functions in $\ell^2(\mathcal{V,\pi})$ or to map measured signals from $\ell^2(\mathcal{V,\pi})$ into $\ell^2(\mathcal{V})$.

  % can be used to map to go into the $\ell^2(\mathcal{V},\pi)$ space. 

\subsection{Redundant wavelet transform on directed graphs}
\label{spectral_wavelet_dg}

 In this section, we propose a redundant wavelet transform on directed graphs. We follow the construction of spectral wavelets on undirected graphs~\cite{hammond2011wavelets} and polynomial diffusion frames~\cite{maggioni2008diffusion}. The novelty is the construction of filters designed in the frequency domain via a frequency response function as a linear combination of projectors onto the associated mono-frequency random walk subspaces. We introduce the necessary elements for the construction of redundant wavelets on directed graphs. 
 
 \subsubsection{Theoretical framework}
Let $\mathcal{G}=(\mathcal{V},\mathcal{E})$ be a strongly connected directed graph with cardinality $|\mathcal{V}|=N$. The directed graph $\mathcal{G}$ is characterized by its adjacency matrix $\mathbf{W}\in\mathbb{R}^{N\times N}_{+}$.
On $\mathcal{G}$, we define a random walk $\mathcal{X}$ characterized by its transition matrix $\mathbf{P}=\mathbf{D}^{-1}\mathbf{W}\in\mathbb{R}^{N\times N}_{+}$. $\mathcal{X}$ is assumed ergodic with stationary distribution $\pi$. Let $\mathbf{T}=\boldsymbol\Pi^{1/2}\mathbf{P}\boldsymbol\Pi^{-1/2}\in \ell^2(\mathcal{V})$ be the operator similar to $\mathbf{P}$ in $\ell^2(\mathcal{V})$. We assume $\mathbf{T}$ diagonalizable. We now define low-pass and band-pass operators as follows. 

\begin{definition}
A low pass operator at dilation $t$, $\mathbf{H}_{t}^{\mathcal{G}}$ on a strongly connected directed graph $\mathcal{G}$ is defined by
\begin{equation*}
\mathbf{H}_{t}^{\mathcal{G}}=\sum_{\ell=1}^p h(t\omega_{\ell})\mathbf{S}_{\ell}, \quad t\in \mathbb{N}.
\end{equation*}
with $h:[0,2]\rightarrow \mathbb{R}$, a function giving a low-pass frequency response and $\{\mathbf{S}_{\ell}\}_{\ell=1}^p$, the set of projectors~\eqref{eq:mono_freq_proj} associated to the frequencies $\omega_\ell$ of the eigenvectors of $\mathbf{T}$. The scaling function at dilation $t$ and translation $y$ is denoted by
\begin{equation*}
\mathfrak{h}_{t,y}^{\mathcal{G}}= \mathbf{H}_{t}^{\mathcal{G}}\boldsymbol\delta_{y}
\end{equation*}
where $\boldsymbol\delta_{y}$ is the Kronecker delta function at vertex $y\in \mathcal{V}$. 
\end{definition}

\begin{definition}
A band pass operator at dilation $t$, $\mathbf{G}_{t}^{\mathcal{G}}$ on a strongly connected directed graph $\mathcal{G}$ is defined as
\begin{equation*}
\mathbf{G}_{t}^{\mathcal{G}}=\sum_{\ell=1}^p g(t\omega_{\ell})\mathbf{S}_{\ell}, \quad t\in \mathbb{N}.
\end{equation*}
with $g:[0,2]\mapsto \mathbb{R}$, a function giving a band-pass frequency response and $\{\mathbf{S}_{\ell}\}_{\ell=1}^p$, the set of projectors~\eqref{eq:mono_freq_proj} associated to the frequencies $\omega_\ell$ of the eigenvectors of $\mathbf{T}$. The wavelet function at dilation $t$ and translation $y$ is denoted by  
\begin{equation*}
\mathfrak{g}_{t,y}^{\mathcal{G}}=\mathbf{G}_{t}^{\mathcal{G}}\boldsymbol\delta_{y}, 
\end{equation*}
where $\boldsymbol\delta_{y}$ is the Kronecker delta function at vertex $y\in \mathcal{V}$. 
\end{definition}

Having defined the low-pass and band-pass operators on directed graphs, we are able to build analysis and synthesis filter banks on directed graphs. 

\begin{definition}
We define a bank of synthesis filters $\mathcal{K}$ as a set comprising a low-pass filter at dilation $t_J$, $\mathbf{H}_{t_J}^{\mathcal{G}}$ and a series of band-pass filters at increasing dilations $\{t_j\}_{j=1}^J$, $\{\mathbf{G}_{t_j}\}_{j=1}^J$ :
\begin{equation*}
\mathcal{K}=\{\mathbf{H}_{t_J}^{\mathcal{G}},\mathbf{G}_{t_1}^{\mathcal{G}},\ldots,\mathbf{G}_{t_J}^{\mathcal{G}}\}.
\end{equation*}
\end{definition}

\begin{definition}
We define a bank of analysis filters $\tilde{\mathcal{K}}$ as a set comprising a
filter at dilation $t_J$, $\tilde{\mathbf{H}}_{t_J}^{\mathcal{G}}$ and a series of filters at increasing dilations $\{t_j\}_{j=1}^J$, $\{\tilde{\mathbf{G}_{t_j}}\}_{j=1}^J$ :

\begin{equation*}
\tilde{\mathcal{K}}=\{\tilde{\mathbf{H}}_{t_J}^{\mathcal{G}},\tilde{\mathbf{G}}_{t_1}^{\mathcal{G}},\ldots,\tilde{\mathbf{G}}_{t_J}^{\mathcal{G}}\},
\end{equation*}
where 
\begin{equation*}
\tilde{\mathbf{H}}_{t_J}^{\mathcal{G}}=\sum_{\ell=1}^p \tilde{h}(t\omega_{\ell})\mathbf{S}_{\ell}\in\mathbb{R}^{N\times N},
\end{equation*}
and
\begin{equation*}
\tilde{\mathbf{G}}_{t_j}^{\mathcal{G}}=\sum_{\ell=1}^p \tilde{g}(t\omega_{\ell})\mathbf{S}_{\ell}\in\mathbb{R}^{N\times N}, \quad \forall j=1,\ldots,J.
\end{equation*}

We also define $\tilde{\mathfrak{h}}_{t,k}$ and $\tilde{\mathfrak{g}}_{t,k}$ as row vectors as follows:
\begin{equation*}
\tilde{\mathfrak{h}}_{t,k}=\boldsymbol\delta_k^{\top}\tilde{\mathbf{H}}_{t}^{\mathcal{G}}, \quad \tilde{\mathfrak{g}}_{t,k}=\boldsymbol\delta_k^{\top}\tilde{\mathbf{G}}_{t}^{\mathcal{G}}.
\end{equation*}

\end{definition}

\begin{prop}
\label{perfect_rec_cond}
Given a fixed set of increasing dilations $\{t_j\}_{j=1}^J$, the synthesis and analysis filters sets respectively $\mathcal{K}$ and $\tilde{\mathcal{K}}$, the perfect reconstruction condition
\begin{equation}
\label{perfect_rec}
\mathbf{H}_{t_J}^{\mathcal{G}}\tilde{\mathbf{H}}_{t_J}^{\mathcal{G}}+\sum_{j=1}^J\mathbf{G}_{t_j}^{\mathcal{G}}\tilde{\mathbf{G}}_{t_j}^{\mathcal{G}}=\mathbf{I}
\end{equation}
is guaranteed if and only if
\begin{equation*}
{h}(t_J\omega_{\ell}) \tilde{h}(t_J\omega_\ell)+\sum_{j=1}^J g(t_j\omega_\ell)\tilde{g}(t_j\omega_\ell)=1, \quad \forall \ell=1,\ldots,p.
\end{equation*}
\end{prop}

 \begin{proof}
 \begin{align*}
 \mathbf{H}_{t_J}^{\mathcal{G}}\tilde{\mathbf{H}}_{t_J}^{\mathcal{G}}+\sum_{j=1}^J\mathbf{G}_{t_j}^{\mathcal{G}}\tilde{\mathbf{G}}_{t_j}^{\mathcal{G}}&=\sum_{\ell,\ell'=1}^p h(t_J\omega_{\ell})\tilde{h}(t_J\omega_{\ell'})\mathbf{S}_{\ell} \mathbf{S}_{\ell'}+\sum_{j=1}^{J} \sum_{\ell,\ell'=1}^p g(t_j\omega_{\ell})\tilde{g}(t_j\omega_{\ell'})\mathbf{S}_{\ell} \mathbf{S}_{\ell'}\\
 &=\sum_{\ell=1}^p h(t_J\omega_{\ell})\tilde{h}(t_J\omega_{\ell})\mathbf{S}_{\ell} +\sum_{j=1}^{J} \sum_{\ell=1}^p g(t_j\omega_{\ell})\tilde{g}(t_j\omega_{\ell})\mathbf{S}_{\ell}\\
&=\sum_{\ell=1}^p \big[h(t_J\omega_{\ell})\tilde{h}(t_J\omega_{\ell}) +\sum_{j=1}^{J}  g(t_j\omega_{\ell})\tilde{g}(t_j\omega_{\ell})\big]\mathbf{S}_{\ell}
 \end{align*}
 
  Consequently, \eqref{perfect_rec} is guaranteed if and only if
 \begin{equation*}
 h(t_J\omega_{\ell})\tilde{h}(t_J\omega_{\ell}) +\sum_{j=1}^{J}  g(t_j\omega_{\ell})\tilde{g}(t_j\omega_{\ell})=1, \quad \forall \ell=1,\ldots,p.\qedhere
 \end{equation*}
 \end{proof}
 
 We now define a frame as follows. 
 \begin{definition}
Let $\mathcal{G}=(\mathcal{V},\mathcal{E})$ be a directed graph and $\mu$ a measure on $\mathcal{V}$. A countable family of elements $\{\boldsymbol f_k\}_{k=1}^n \in \ell^2(\mathcal{V})$ is said to be a frame if for any graph signal $\boldsymbol f \in \ell^2(\mathcal{V})$ we have  

\begin{equation*}
A\|\boldsymbol f\|^2\leq\sum_{k=1}^n|\langle \boldsymbol f,\boldsymbol f_k\rangle|^2\leq B\|\boldsymbol f\|^2,
\end{equation*}
for some constants $0<A\leq B<\infty$ which are called frame bounds.
\end{definition}

We introduce the rectangular matrix $\mathbf{K}=(\mathbf{H}_{t_J}^{\mathcal{G}},\mathbf{G}_{t_1}^{\mathcal{G}},\ldots,\mathbf{G}_{t_J}^{\mathcal{G}})\in\mathbb{R}^{N\times N(J+1)}$ where :
\begin{equation*}
\mathbf{K}=
\begin{pmatrix}
\mathfrak{h}_{t_J,1}^{\mathcal{G}},\ldots,\mathfrak{h}_{t_J,N}^{\mathcal{G}},\mathfrak{g}_{t_1,1}^{\mathcal{G}},\ldots,\mathfrak{g}_{t_J,1}^{\mathcal{G}},\ldots,\mathfrak{g}_{t_J,N}^{\mathcal{G}}
\end{pmatrix}
\in\mathbb{R}^{N\times N(J+1)}.
\end{equation*}

We also introduce the following rectangular matrix $\tilde{\mathbf{K}}=(\tilde{\mathbf{H}}_{t_J}^{\mathcal{G}\top},\tilde{\mathbf{G}}_{t_1}^{\mathcal{G}\top},\ldots,\tilde{\mathbf{G}}_{t_J}^{\mathcal{G}\top})^\top\in\mathbb{R}^{N(J+1)\times N}$  
\begin{equation}
\tilde{\mathbf{K}}=
\begin{pmatrix}
\tilde{\mathfrak{h}}_{t_J,1}^{\mathcal{G}}\\
\vdots\\
\tilde{\mathfrak{h}}_{t_J,N}^{\mathcal{G}}\\
\tilde{\mathfrak{g}}_{t_1,1}^{\mathcal{G}}\\
\vdots\\
\tilde{\mathfrak{g}}_{t_J,1}^{\mathcal{G}}\\
\vdots\\
\tilde{\mathfrak{g}}_{t_J,N}^{\mathcal{G}}
\end{pmatrix}
\in\mathbb{R}^{N(J+1)\times N}.
\end{equation}

\begin{prop}
 Let $\mathbf{K}$ and $\tilde{\mathbf{K}}$  be respectively the synthesis and analysis filter banks. We assume that the filter banks $\mathbf{K}$ and $\tilde{\mathbf{K}}$ verify the perfect reconstruction condition stated in \cref{perfect_rec_cond}. Consequently, $\tilde{\mathbf{K}}$ is a frame for $\ell^2(\mathcal{V})$ with lower frame bound $1/\|\mathbf{K}\|^2$ and upper frame bound $\|\tilde{\mathbf{K}}\|^2$. 
 \end{prop}

\begin{proof}
% By definition of the spectral norm, 
% \begin{equation*}
%   \|\tilde{\mathbf{K}}\boldsymbol f\|^2 \leq \|\tilde{\mathbf{K}}\|^2\|\boldsymbol f\|^2.
% \end{equation*}

The perfect reconstruction condition implies that $\|\boldsymbol f\|^2=\|\mathbf{K}\tilde{\mathbf{K}}\boldsymbol f\|^2\leq \|\mathbf{K}\|^2\|\tilde{\mathbf{K}}\boldsymbol f\|^2$. Hence, we have the frame inequalities
\begin{equation}
  \frac{1}{\|\mathbf{K}\|^2}\|\boldsymbol f\|^2\leq \|\tilde{\mathbf{K}}\boldsymbol f\|^2 \leq \|\tilde{\mathbf{K}}\|^2\|\boldsymbol f\|^2,
\end{equation}
where the second inequality is just the definition of the spectral norm.
$\|\mathbf{K}\|^2$ and $\|\tilde{\mathbf{K}}\|^2$ are finite because of finite dimensions, hence the frame bounds verify $0<\|\mathbf{K}\|^2 
\leq\|\tilde{\mathbf{K}}\|^2<\infty$.
\end{proof}

% We have
% \begin{equation*}
% \|\tilde{\mathbf{K}}\boldsymbol f\|^2=\sum_{k=1}^N|\langle \tilde{\mathfrak{h}}_{t_J,k}^{\mathcal{G}},\boldsymbol f\rangle|^2+\sum_{j=1}^J\sum_{k=1}^N |\langle \tilde{\mathfrak{g}}_{t_j,k}^{\mathcal{G}},\boldsymbol f\rangle|^2.
% \end{equation*}
% Consequently, we can write 
% \begin{align*}
% & A\|\boldsymbol f\|^2\leq\|\tilde{\mathbf{K}}\boldsymbol f\|^2\leq B\|\boldsymbol f\|^2,\\
% &A\|\boldsymbol f\|^2\leq\sum_{k=1}^N|\langle \tilde{\mathfrak{h}}_{t_J,k}^{\mathcal{G}},\boldsymbol f\rangle|^2+\sum_{j=1}^J\sum_{k=1}^N |\langle \tilde{\mathfrak{g}}_{t_j,k}^{\mathcal{G}},\boldsymbol f\rangle|^2\leq B\|\boldsymbol f\|^2,
% \end{align*}
% where $A=1/\|\mathbf{K}\|^2$ et $B=\|\tilde{\mathbf{K}}\|^2$. 

% The right side is obtained by Cauchy-Schwarz inequality applied on $\|\tilde{\mathbf{K}}\boldsymbol f\|^2$. The left side is obtained by the fact that $\|\boldsymbol f\|^2=\|\mathbf{K}\tilde{\mathbf{K}}\boldsymbol f\|^2\leq \|\mathbf{K}\|^2\|\tilde{\mathbf{K}}\boldsymbol f\|^2$ and application of the Cauchy-Schwarz inequality. As a result, $\tilde{\mathbf{K}}$ is a frame.
% \end{proof}

 \begin{definition}[Wavelet decomposition of a signal on a directed graph]
 Any graph signal $\boldsymbol f\in\ell^2(\mathcal{V})$ can be expressed as follows : 
 \begin{equation*}
 \boldsymbol f=\sum_{k=1}^N \langle \boldsymbol f,\tilde{\mathfrak{h}}_{t_J,k}\rangle\,\mathfrak{h}_{t_J,k}+\sum_{j=1}^J\sum_{k=1}^N\langle \boldsymbol f,\tilde{\mathfrak{g}}_{t_j,k}\rangle\,\mathfrak{g}_{t_j,k}.
 \end{equation*}
 \end{definition}

\subsection{Critically sampled wavelet transform on directed graphs}
\label{diffusion_wavelets_dg}

\subsubsection{Notations}
\label{notations}

We first recall relevant and useful notations for a clear understanding of the construction of diffusion wavelets  \cite{coifman2006diffusion,maggioni2005biorthogonal}. Two notations are introduced: one for the representation of linear transformations as matrices and a second for the representation of sets of vectors as matrices where columns correspond to a vector in the set. These notations were used in \cite{coifman2006diffusion,maggioni2005biorthogonal} except that we adopt a column vector convention whereas a row vector convention was used in \cite{coifman2006diffusion,maggioni2005biorthogonal}.

Let $\mathsf{V}_0=\ell^2(\mathcal{V})$ be the space of functions defined over the vertices of a directed graph $\mathcal{G}=(\mathcal{V},\mathcal{E})$. If $\mathbf{L}$ is a linear transformation of $\mathsf{V}_0$ into $\mathsf{V}_0$, $[\mathbf{L}]_{B_1}^{B_2}$ indicates the matrix representing the linear transformation $\mathbf{L}$ with respect to the basis $B_1$ in the domain and $B_2$ in the range. A set of vectors $X$ represented in a given basis $B$ will be written in the matrix form  $[X]_{B}$ where the columns of $[X]_{B}$ are the coordinates of the vectors $X$ in the coordinate system defined by $B$. Generally, $[B_1]_{B_2}=[\mathbf{I}]_{B_1}^{B_2}$ represent the basis vectors $B_1$ in terms of the basis $B_2$. We will note that, for a given basis $B$ if the input and output bases are the same, $[\mathbf{I}]_{B}^{B}=\mathbf{I}$ the identity matrix. We will also abuse this notation if $B_1$ spans a subspace of the space spanned by $B_2$. If $B_2$ spans a subspace of the space spanned by $B_1$ instead, then we use the convention $[\mathbf{I}]_{B_1}^{B_2} = ([\mathbf{I}]_{B_2}^{B_1})^+$, where $\mathbf{X}^+$ denotes the Moore-Penrose pseudo-inverse of $\mathbf{X}$.
If $B_1$ and $B_2$ are two bases, the representations of $X$ in $B_1$ and $B_2$ are related as follows: $[X]_{B_2} = [\mathbf I]_{B_1}^{B_2}[X]_{B_1}$.  If $B_1,B_2,B_3,B_4$ are arbitrary bases, the matrix representation of $\mathbf{L}$ with respect to the basis $B_1$ in the domain and $B_2$ in the range in the different bases is expressed as 
\begin{equation*}
[\mathbf{L}]_{B_1}^{B_2}=[\mathbf{I}]_{B_4}^{B_2}[\mathbf{L}]_{B_3}^{B_4}[\mathbf{I}]_{B_2}^{B_1}. 
\end{equation*}
Furthermore, if $B_1$ and $B_2$ are linearly independent sets of vectors that do not span the whole space $\mathsf{V}_0$ then we will still use the notation $[\mathbf{L}]_{B_1}^{B_2}$, but in that case it will represent the \emph{restriction} of the linear transformation $\mathbf L$ to the domain and range subspaces spanned by $B_1$ and $B_2$.

\subsubsection{Diffusion wavelets}
The construction of the diffusion wavelets~\cite{coifman2006diffusion,maggioni2005biorthogonal} enables a multi-resolution analysis of functions on graphs generalizing the concept of classical multiresolution analysis~\cite{mallat1989theory}. The starting point of this construction is a diffusion operator $\mathbf{T}$. Similarly to the classical multiresolution analysis, diffusion wavelets is characterized by a family of nested subspaces $\mathsf{V}_0\supseteq \mathsf{V}_1\supseteq\cdots\supseteq \mathsf{V}_j\supseteq\cdots$, where each subspace $\mathsf{V}_j $ is spanned by a basis of scaling functions $\Phi_j$. The complement of $\mathsf{V}_{j+1}$ into $\mathsf{V}_{j}$  is called $\mathsf{W}_{j}$ and is spanned by a set of diffusion wavelets $\Psi_j$.

\subsubsection{Construction}
\label{ssec:DiffWaveletConstruction}
The construction of diffusion wavelets proceeds with a diffusion operator $\mathbf{T}$ defined on a directed strongly connected graph $\mathcal{G}=(\mathcal{V},\mathcal{E})$. In the original diffusion wavelets framework by Coifman and Maggioni \cite{coifman2006diffusion}, the graph is assumed undirected and they suggest using the reversible random walk operator $\mathbf{P}$. 

More generally, any operator $\mathbf{T}$ can be used as long as it acts like a low-pass filter. Given a  directed graph $\mathcal{G}$, an appropriate choice seems to be the operator $\mathbf{T}=\boldsymbol\Pi^{1/2}\mathbf{P}\boldsymbol\Pi^{-1/2}$ where $\mathbf{P}$ is the random walk operator. However, the random-walk operator is not guaranteed to have a low-pass frequency response as the highest frequency eigenvectors may have eigenvalues with modulus arbitrarily close to one (as e.g. for the cycle graph). In such a case, a low-pass filter can be defined using a lazy random walk instead: $\mathbf{T}=\boldsymbol\Pi^{1/2}(\gamma\mathbf{I}+(1-\gamma)\mathbf{P})\boldsymbol\Pi^{-1/2}$.

 % with low-pass frequency response with respect to the definition proposed in section~\ref{Fourier_analysis}.

 Assuming an ordering of vertices in $\mathcal{V}$, $\mathbf{T}$ is originally represented on the canonical basis $\Phi_0=\{\boldsymbol\delta_k\}_{k\in\mathcal{V}}$ of $\mathsf{V}_0$ where $\boldsymbol\delta_k$ is the Kronecker delta function corresponding to the vertex $k\in \mathcal{V}$. Using the notations from section \ref{notations}, $[\mathbf{T}]_{\Phi_0}^{\Phi_0}$ is the matrix representation of the linear operator $\mathbf{T}$ with respect to the basis $\Phi_0$ in the domain and $\Phi_0$ in the range.   
 
 The columns of the matrix $[\mathbf{T}]_{\Phi_0}^{\Phi_0}$ can be described as a set of functions $\widetilde{\Phi}_1=\{\widetilde{\boldsymbol\phi}_{1,k}\}_{k\in\mathcal{V}}$ represented in the basis $\Phi_0$. Each element $\widetilde{\boldsymbol\phi}_{1,k},k\in \mathcal{V}$ of $\widetilde{\Phi}_{1}$ corresponds to the action of $\mathbf{T}$ on the Kronecker $\boldsymbol\delta$-function at vertex $k$ and is represented as $[\widetilde{\boldsymbol\phi}_{1,k}]_{\Phi_0}=[\mathbf{T}]_{\Phi_0}^{\Phi_0}\boldsymbol\delta_k$. 
 
The low-pass property of $\mathbf{T}$ leads to the fact that each element $\widetilde{\boldsymbol\phi}_{1,k}$ of $\widetilde{\Phi}_{1}$ is a function localized around its vertex $k$ whose support extends to its close neighbors. In terms of the diffusion wavelets construction, the elements $\widetilde{\boldsymbol\phi}_{1,k},k\in \mathcal{V}$ are therefore interpreted as \emph{scaling} functions. Due to the fact that the support of each element $\widetilde{\boldsymbol\phi}_{1,k}$ covers a small neighborhood around their respective vertices, these elements $\widetilde{\boldsymbol\phi}_{1,k}$ can generally be well approximated by a linear combination of the other functions $\widetilde{\boldsymbol\phi}_{1,l}$ with $l\neq k$. The next stage of the construction is a column subset selection stage \cite{civril2009column}. The aim is to find a small number of columns of $\tilde{\Phi}_1$, forming a set $\mathcal{C}$ such that the residual $\|[\tilde{\Phi}_1]_{\Phi_0}-\Pi_\mathcal{C}[\tilde{\Phi}_1]_{\Phi_0}\|_{\beta}$ is as minimal as possible where $\Pi_{\mathcal{C}}$ is the projection matrix onto the space spanned by the columns of $\mathcal{C}$ and $\beta$ denotes the spectral norm or the Frobenius norm. 

More precisely, this step involves the selection of a subset $\{\widetilde{\boldsymbol\phi}_{1,k},k\in\mathcal{I}_1\}$ from $\widetilde{\Phi}_{1}$, $|\mathcal{I}_1|\leq|\mathcal{V}|$. We select a subset $\{\widetilde{\boldsymbol\phi}_{1,k},k\in \mathcal{I}_1\}$ such that all $\widetilde{\boldsymbol\phi}_{1,k}$ are well-enough approximated by linear combinations of the functions in the subset. In classical signal processing terms, this step is analogous to a subsampling of a set of scaling functions in the classical discrete wavelet transform. Coifman and Maggioni used a greedy approach to build the set $\mathcal{I}_1$ iteratively by using a modified Gram-Schmidt orthogonalization procedure~\cite{coifman2006diffusion}. 

The subset $\{\widetilde{\boldsymbol\phi}_{1,k},k\in \mathcal{I}_1\}$ spans a subspace $\mathsf{V}_1$ which corresponds to the first approximation subspace of the multi-resolution analysis. We will thus denote $\{\widetilde{\boldsymbol\phi}_{1,k},k\in \mathcal{I}_1\}$ as $\Phi_1$, where $\Phi_1$ is by definition a basis of $\mathsf{V}_1$ that is generally not orthogonal. The vectors in $\Phi_1$ are the \emph{scaling functions} at scale 1.

In order to define the next scales of scaling functions, we consider the compression step of the diffusion operator $\mathbf{T}$ on the subspace $\mathsf{V}_1$, that is its restriction on $\mathsf{V}_1$. The latter can be represented in $\Phi_1$ as 

\begin{equation*}
[\mathbf{T}]_{\Phi_1}^{\Phi_1}=[\mathbf{I}]_{\Phi_0}^{\Phi_1}[\mathbf T]_{\Phi_0}^{\Phi_0} [\mathbf{I}]_{\Phi_1}^{\Phi_0}.
\end{equation*}
where $[\mathbf{I}]_{\Phi_0}^{\Phi_1}$ represents the restriction of the signal space $\mathsf{V}_0$ to $\mathsf{V}_1$ (with respective bases $\Phi_0$ and $\Phi_1$) and $[\mathbf{I}]_{\Phi_1}^{\Phi_0}$ represents the embedding of $\mathsf{V}_1$ in $\mathsf{V}_0$.

The next approximation space $\mathsf{V}_2\subset \mathsf{V}_1$ and its associated basis $\Phi_2$ can be obtained in the same way as the definition of $\mathsf{V}_1$ and its basis $\Phi_1$ except that we now consider the operator $\mathbf{T}^2$ restricted to $\mathsf{V}_1$ instead of $\mathbf{T}$ in $\mathsf{V}_0$.
 
  The columns of $[\mathbf{T}^2]_{\Phi_1}^{\Phi_1} \approx ([\mathbf{T}]_{\Phi_1}^{\Phi_1})^2$ can be interpreted as scaling functions at scale 2, $\widetilde{\Phi}_2=\{\widetilde{\boldsymbol\phi}_{2,k}\}$, represented in the basis $\Phi_1$. From these functions we extract a subset $\Phi_2 = \{\widetilde{\boldsymbol\phi}_{2,k},k\in \mathcal{I}_2\}$ such that any function in $\widetilde{\Phi}_2$ is well-approximated by a linear combination of functions of $\Phi_2$. 

After $j$ iterations of this procedure we have defined $j$ approximation subspaces $\mathsf{V}_j\subset \mathsf{V}_{j-1}\subset\cdots\subset \mathsf{V}_1$ with corresponding bases $\Phi_j, \Phi_{j-1},\ldots, \Phi_1$. At each step the basis $\Phi_j$ is defined by its representation in the basis $\Phi_{j-1}$ based on the restriction of the operator $\mathbf{T}^{2^j}$ to $\mathsf{V}_{j-1}$. In order to represent these functions in the original basis $\Phi_0$ of $\mathsf{V}_0$ we can use
\begin{equation*}
  [\Phi_j]_{\Phi_0} = [\mathbf{I}]_{\Phi_j}^{\Phi_0}=[\mathbf{I}]_{\Phi_{0}}^{\Phi_0} [\mathbf{I}]_{\Phi_{1}}^{\Phi_0}\cdots [\mathbf{I}]_{\Phi_{j-1}}^{\Phi_{j-2}}[\mathbf{I}]_{\Phi_{j}}^{\Phi_{j-1}}.
\end{equation*}
 
 Since every function in $\Phi_0$ is defined on $\mathsf{V}_0$, so is every function on $\Phi_j$. Hence any function in the approximation space $\mathsf{V}_j$ can be extended naturally to the whole space $V_0$.
  
Regarding the construction of the wavelets, we propose to construct the wavelet bases $\Psi_j$ for the subspaces $\mathsf{W}_j$ by selecting a subset of the columns of the band pass operator $[\mathbf{I}_{\mathsf{V}_j}-\Phi_{j+1}\Phi_{j+1}^{\dagger}]_{\Phi_j}^{\Phi_j}$ which is the orthogonal projector on the complement of $\mathsf{V}_{j+1}$ into $\mathsf{V}_j$. The wavelets capture the detail lost from going from $\mathsf{V}_j$ to $\mathsf{V}_{j+1}$. As our framework falls into a bi-orthogonal scope, we need to build the dual wavelet bases. For a scale $j$, we have a wavelet base $\Psi_j$ and we need to construct the associated dual base $\widehat{\Psi}_j$ obtained such that we have the following relation 
\begin{equation*}
\widehat{\Psi}_j^{\dagger}\Psi_j=\mathbf{I}. 
\end{equation*}
where $\widehat{\Psi}_j^{\dagger}$ is the pseudo-inverse of $\Psi_j$. We mention that in~\cite{maggioni2005biorthogonal}, biorthogonal wavelet transform is proposed but only the scaling functions are actually defined. Our construction is identical but we also propose a definition for wavelets.

\subsubsection{A generalization using more arbitrary scaling operators}
We propose a generalization of the diffusion wavelet framework that enables it to be combined for instance with the wavelet construction presented in Section~\ref{diffusion_wavelets_dg}. The idea is merely to replace the powers of $\mathbf{T}$ as multiresolution scaling operators by more arbitrary low-pass filters. More precisely, where the operator $\mathbf{T}^{2^j}$ is used to define the approximation space $\mathsf{V}_{j+1}$ in section~\ref{ssec:DiffWaveletConstruction}, we propose to use instead a low-pass graph filter $\mathbf{H}_j$. If the graph filters  $\mathbf{H}_j$ correspond to the scaling operators defined in~\cref{spectral_wavelet_dg} with appropriately increasing scales, this approach provides a way to reduce the redundancy of the sets of scaling functions. Similarly, the same approach could be used to reduce the redundancy of the wavelet functions defined in Section~\ref{diffusion_wavelets_dg}.

\subsection{Applications}

 This section is devoted to illustrate our multiresolution analysis framework through the visualization of wavelets and scaling functions on regular-type structures and semi-supervised learning on directed graphs. The application framework will be the same as in section~\ref{sect:applications_one}.
 
 \subsubsection{Multiresolution analysis on the directed cycle graph}
We show an example of multi-resolution analysis on the directed cycle graph $\mathcal{C}_N$ with $N=256$. We use the same assumptions made in section~\ref{fa_dcg}. We construct both orthogonal and biorthogonal multi-resolution analyses on $\mathcal{C}_N$ through the framework of the diffusion wavelets applied on the dyadic powers of $\mathbf{T}$, i.e. $\{\mathbf{T}^{2^j}\}_{j=1}^J$. We set the number of scales at $J=6$. 

 \begin{figure}
 \includegraphics[width=1.0\columnwidth]
 {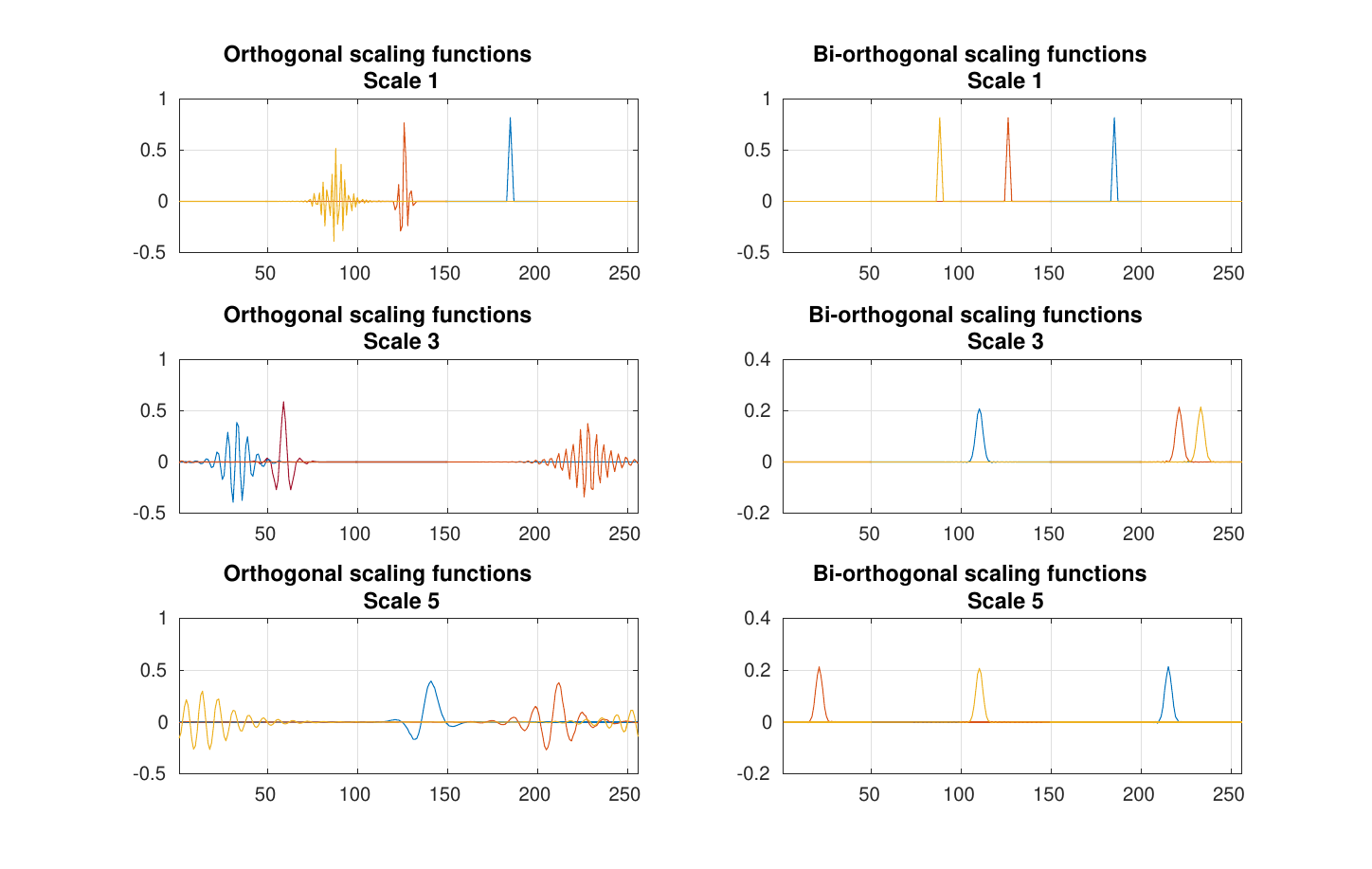}
 \caption{Orthogonal and biorthogonal scaling functions on the directed cycle graph $\mathcal{C}_{256}$.}
 \label{fig:ort_biort_sc_f}
 \end{figure}

 \begin{figure}
 \includegraphics[width=1.0\columnwidth]
 {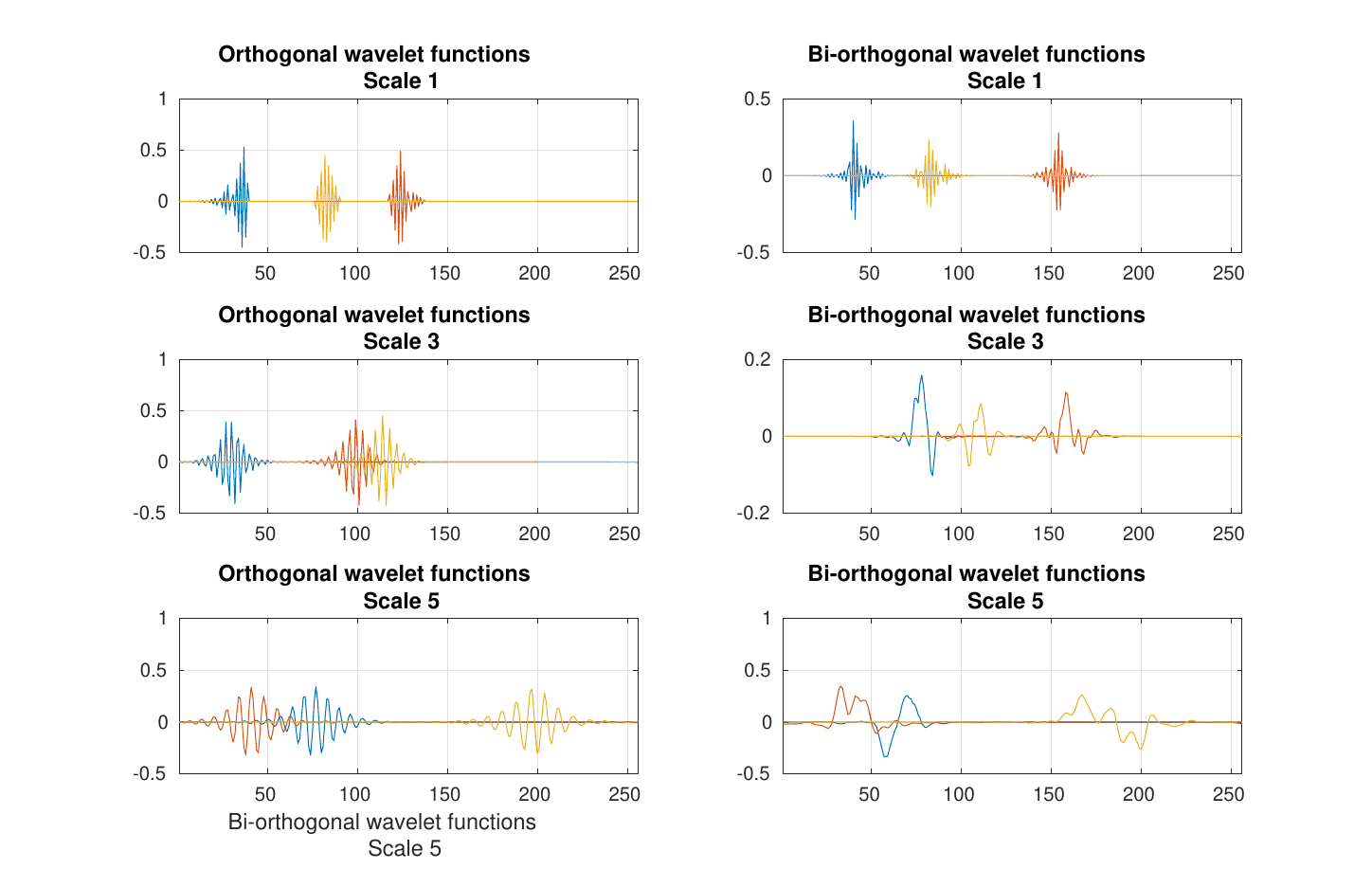}
 \caption{Orthogonal and biorthogonal wavelet functions on the directed cycle graph $\mathcal{C}_{256}$.}
 \label{fig:ort_biort_wav_f}
 \end{figure}

\Cref{fig:ort_biort_sc_f} shows the orthogonal and biorthogonal scaling functions at scales 1,3 and 5. At each scale, we represent 3 or 4 scaling functions. 
\Cref{fig:ort_biort_wav_f} shows the orthogonal and biorthogonal wavelet functions also at scales 1,3 and 5. 

We note at each scale that the support of the orthogonal scaling functions is larger than the support of the biorthogonal scaling functions. We also note at scale 3 and 5 that biorthogonal wavelet functions have support slightly smaller than the orthogonal wavelets functions. 
Furthermore, we note that orthogonal scaling functions as well orthogonal wavelet functions have more oscillations than their biorthogonal counterparts. That means that orthogonal scaling and wavelet functions have a poorer frequency localization. The conclusions about the spatial localization of the scaling functions are identical to the conclusions in~\cite{bremer2006diffusion}. The novelty here is we work on the directed circle graph and we show the orthogonal and biorthogonal wavelets functions in addition to the orthogonal and biorthogonal scaling functions. Finally, the scaling and wavelets functions are not centered around the vertex where they have been selected. They instead propagate along the graph towards the parent vertices of the selected vertex. We can also note that the orthogonal wavelet transform is more robust with respect to perturbations than the biorthogonal wavelet transform. Let $\mathcal{W}_{ortho}=[\Psi_{1},\ldots,\Psi_{5},\Phi_{5}]$ be the inverse wavelet transform and $\mathcal{W}_{biorth}$ its biorthogonal counterpart. The condition numbers on this example are $\kappa(\mathcal{W}_{ortho})\approx 33$ and $\kappa(\mathcal{W}_{biortho})\approx 2\times 10^{4}$. Consequently, the biorthogonal wavelet transform is much more sensitive to perturbations than its orthogonal counterpart.

 \subsubsection{Multiresolution analysis on the directed Watts-Strogatz graph}
 \label{sec:mra_ws}
 
In this section, we show examples of multi-resolution analyses on the directed version of a graph from the Watts-Strogatz model. The Watts-Strogatz model~\cite{watts1998collective} is an undirected random graph model exhibiting small world properties including short average path lengths and high clustering~\cite{newman2003mej}. The construction of a directed graph from the Watts-Strogatz model starts with a directed cycle graph with $N$ vertices. Each node is connected to its $k$ next nodes following the direction of the directed cycle graph. For the sake of simplicity, we consider the $k$ next nodes connected to a given node $i$ following the direction of the directed cycle graph as its ``closest'' neighbors. Starting from an arbitrary vertex, we apply the following procedure to each vertex in a clockwise manner. At vertex $i$, the edge that connects $i$ to each of its next nodes is randomly rewired with probability $p$ or remains untouched with probability $1-p$. This procedure is repeated cyclically for each successive vertex until the vertex $i$ is selected again. We denote by $\mathcal{G}\sim\operatorname{DWS}(N,K,\beta)$ a graph constructed following a directed Watts-Strogatz model with $N$ vertices, $K$ nearest neighbors and rewiring probability $\beta$.

In a first instance, we analyze the scaling functions built from low-pass filters based on $\bar{\mathbf{T}}_{\alpha}$ at a given scale of $\mathcal{G}_1\sim\operatorname{DWS}(64,2,0)$. That is a special case of the directed Watts-Strogatz model with no rewiring. In a second instance, we analyze some scaling functions built from low-pass filters based on $\bar{\mathbf{T}}_{\alpha}$  at a given scale of $\mathcal{G}_2\sim\operatorname{DWS}(64,2,0.02)$. Lastly, we show both orthogonal and biorthogonal scaling and wavelets functions built following the construction in section~\ref{diffusion_wavelets_dg}. We summarize the following cases in table~\ref{mra_cases}.

% In a third instance, we compare some scaling functions built from low-pass filters based on the additive reversibilization of the random walk i.e. $\bar{\mathbf{T}}$ and a random walk built from the symmetrization of $\mathcal{G}_2$. 

\begin{table}[H]
\begin{center}
 \begin{adjustbox}{max width=\textwidth}
\begin{tabular}{c c c}
\toprule
\textbf{Case} & Graph & Reference operators\\
\midrule
Case 1 : Scaling functions built wrt. sect.~\ref{spectral_wavelet_dg} & $\mathcal{G}_1\sim\operatorname{DWS}(64,2,0)$ & $\bar{\mathbf{T}}_{\alpha}\in\bar{\mathcal{T}}$\\[0.5 cm]

Case 2: Scaling functions built wrt. sect.~\ref{spectral_wavelet_dg} & $\mathcal{G}_2\sim\operatorname{DWS}(64,2,0.02)$ & $\bar{\mathbf{T}}_{\alpha}\in\bar{\mathcal{T}}$\\ [0.5cm]

Case 3: Scaling functions built wrt. sect.~\ref{diffusion_wavelets_dg} & $\mathcal{G}_2\sim\operatorname{DWS}(64,2,0.02)$ & $\bar{\mathbf{T}}$\\ [0.5 cm]
\bottomrule
\end{tabular} 
\end{adjustbox}
\end{center}
 \caption{Table of the different cases for the directed Watts-Strogatz graph.}
 \label{mra_cases}
\end{table}

\paragraph{Case 1}
 In this case, we analyze some scaling functions built from low-pass filters based on $\bar{\mathbf{P}}_{\alpha}$ at scale $t=2^4$ of a directed Watts-Strogatz graph $\mathcal{G}_2\sim\operatorname{DWS}(64,2,0)$. We consider the following low-pass filters
\begin{equation}
\label{ref_op_l_two}
\bar{\mathbf{T}}_{\alpha}=\boldsymbol\Pi^{1/2}\bar{\mathbf{P}}_{\alpha}\boldsymbol\Pi^{-1/2},\quad \forall \alpha\in[0,1]. 
\end{equation}
and we construct a collection of low-pass graph filters at a given scale following the construction presented in~\cref{spectral_wavelet_dg}. More precisely, we build the filters $\mathbf{H}_{\alpha}$ as follows
\begin{equation*}
\mathbf{H}_{\alpha}=\sum_{\omega\in\boldsymbol{\omega}} h(t\omega)\mathbf{S}_{\omega,\alpha},
\end{equation*}
with $t=2^{4}$ and $h(x)=\exp(-x)$. 

 \begin{figure}
 \includegraphics[width=1.0\columnwidth]
 {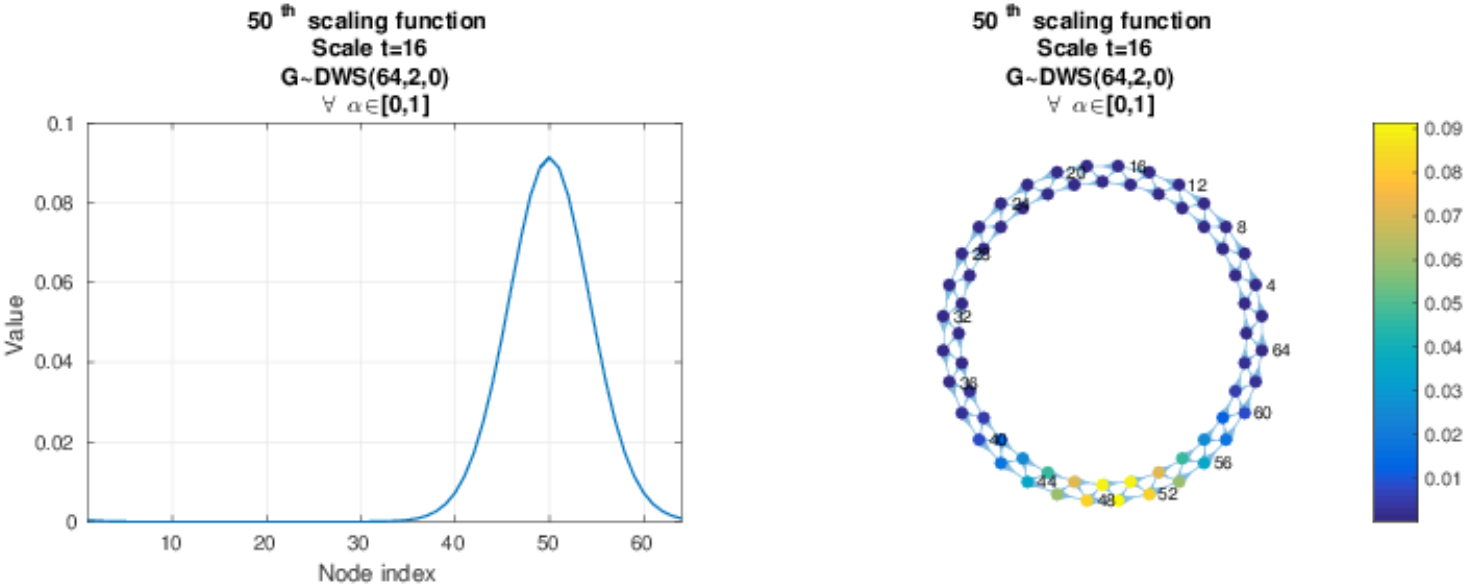}
 \caption{(Case 1) $50^{th}$ scaling function at scale t=16 on the graph $\mathcal{G}\sim\operatorname{DWS}(64,2,0)$, for all $\alpha\in[0,1]$.}
 \label{mra_case_one}
 \end{figure}
 
 \Cref{mra_case_one} represents the $50^{th}$ scaling function $\mathfrak{h}_{\alpha,50}=\mathbf{H}_{\alpha}\boldsymbol\delta_{50}$ at scale $t=2^4$ built from low-pass filters $\mathbf{T}_{\alpha}$, for all $\alpha\in[0,1]$ according to the construction~\cref{spectral_wavelet_dg} on the directed Watts-Strogatz graph $\mathcal{G}\sim\operatorname{DWS}(64,2,0)$. The particularity of this case is that the following graph has no rewiring. Consequently, the associated adjacency matrix is circulant.
All the low-pass filters $\mathbf{T}_{\alpha}$ admit the same discrete Fourier basis and all corresponding frequencies are the same. Hence the associated scaling functions $\mathfrak{h}_{\alpha,50}$ are exactly the same. We observe that the scaling function is centered around the node 50 and has a symmetric shape.
In this case, for non symmetric graph filters $\mathbf{H}_{\alpha}$, the scaling functions have symmetric shape and are centered around their given node.

\paragraph{Case 2} 
In this case, we analyze some scaling functions built from low-pass filters based on $\bar{\mathbf{P}}_{\alpha}$ at scale 10 of $\mathcal{G}_2\sim\operatorname{DWS}(64,2,0.02)$. We also consider the following low-pass filters
\begin{equation}
\bar{\mathbf{T}}_{\alpha}=\boldsymbol\Pi^{1/2}\bar{\mathbf{P}}_{\alpha}\boldsymbol\Pi^{-1/2},\quad \forall \alpha\in\{0,0.5,1\}. 
\end{equation}
and we construct a collection of low-pass graph filters at a given scale following the construction presented in~\cref{spectral_wavelet_dg}. More precisely, we build the filters $\mathbf{H}_{\alpha}$ as follows
\begin{equation*}
\mathbf{H}_{\alpha}=\sum_{\omega\in\boldsymbol{\omega}} h(t\omega)\mathbf{S}_{\omega,\alpha}
\end{equation*}
with $t=2^{4}$ and $h(x)=\exp(-x)$. We observe the $50^{th}$ scaling functions at scale $t=2^{4}$, that is $\mathfrak{h}_{\alpha,50}= \mathbf{H}_{\alpha}\boldsymbol\delta_{50}$.

 \begin{figure}
 \includegraphics[width=1.0\columnwidth]
 {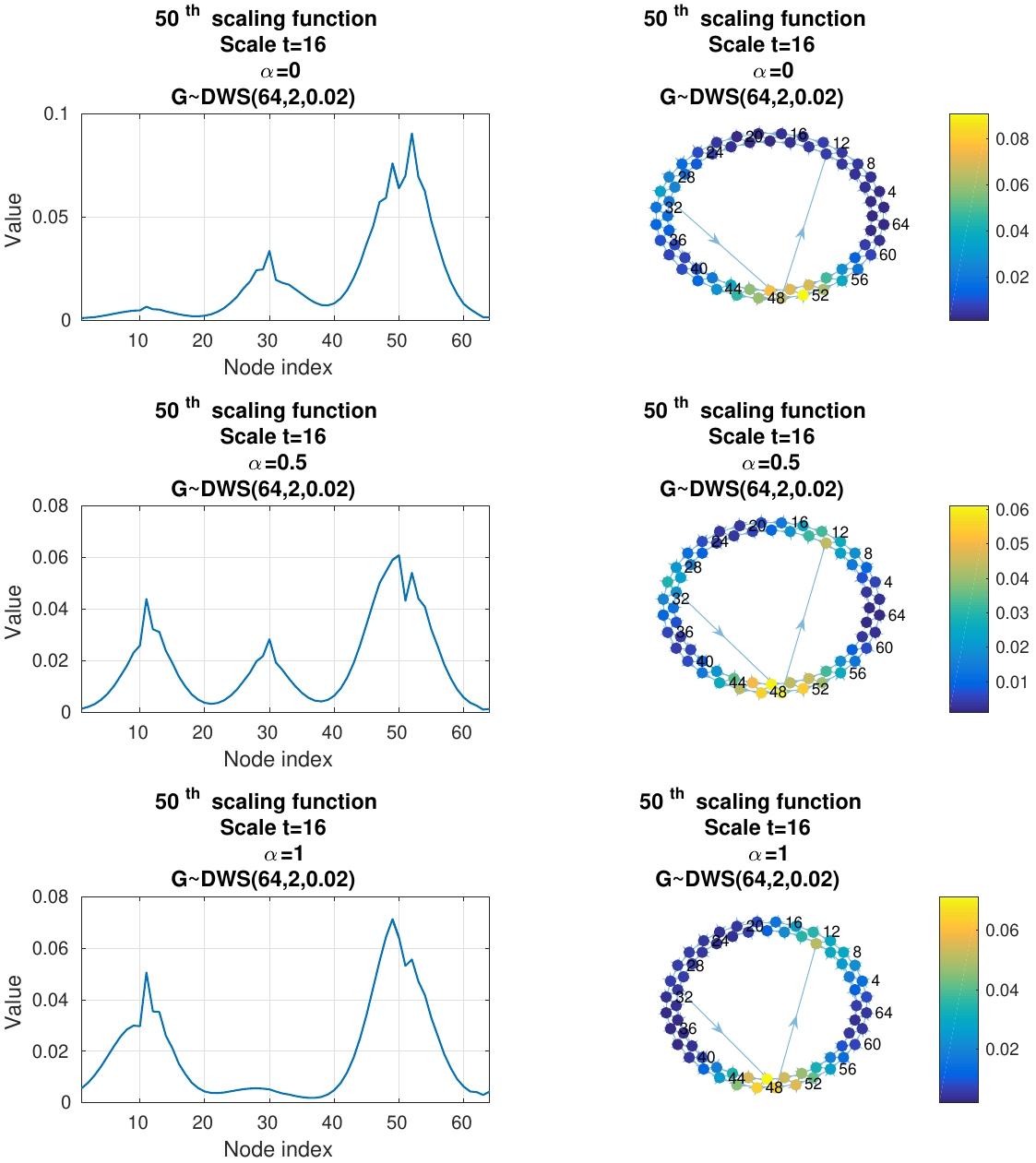}
 \caption{$50^{th}$ scaling function at scale 4 on a graph $\mathcal{G}\sim\operatorname{DWS} (64,2,0.02)$, $\alpha\in\{0,0.5,1\}$,~\cref{ref_op_l_two}.}
 \label{mra_case_two_part_one}
 \end{figure}
 
The $50^{th}$ scaling function at the scale $t=2^4$, for different $\alpha$, $\mathfrak{h}_{\alpha,50}$, are represented in \cref{mra_case_two_part_one}. The graph $\mathcal{G}_2$ admits a directed edge from node $50$ to node $11$ and a directed edge from node $30$ to node $50$. As $\alpha$ increases towards $1$, we observe that the scaling function $\mathfrak{h}_{\alpha,50}$ propagates around the child node $11$. This means that the more $\alpha$ increases, the more the influence of the backward random walk $\mathbf{P}^{*}$ is important. For $\alpha=0.5$, the scaling function $\mathfrak{h}_{\alpha,50}$ diffuses both ways and is centered around the child node $11$ and the parent node $30$. Finally, as $\alpha$ goes to 0, we observe that $\mathfrak{h}_{\alpha,50}$ mostly propagates towards the parent node $30$.

\paragraph{Case 3}
Here, we consider orthogonal and biorthogonal scaling functions on a graph $\mathcal{G}\sim \operatorname{DWS}(64,2,0.02)$ used on the Case 2. These orthogonal and biorthogonal scaling functions are built from the diffusion wavelets procedure. We start the procedure with the low-pass filter $\mathbf{H}=\mathbf{T}^{2}$ and we look at the dyadic powers of $\mathbf{H}$  i.e. $\{\mathbf{H}^{2^j}\}_{j=1}^J$. We set the number of scales at $J=5$. We observe the orthogonal an biorthogonal scaling functions at node $49$ at scale $3$ obtained by the diffusion wavelets procedure and we compare to the scaling function at node $49$ built from a graph filter based on $\bar{\mathbf{T}}$ at the scale $t=16$. 

 \begin{figure}
 \includegraphics[width=1.0\columnwidth]
 {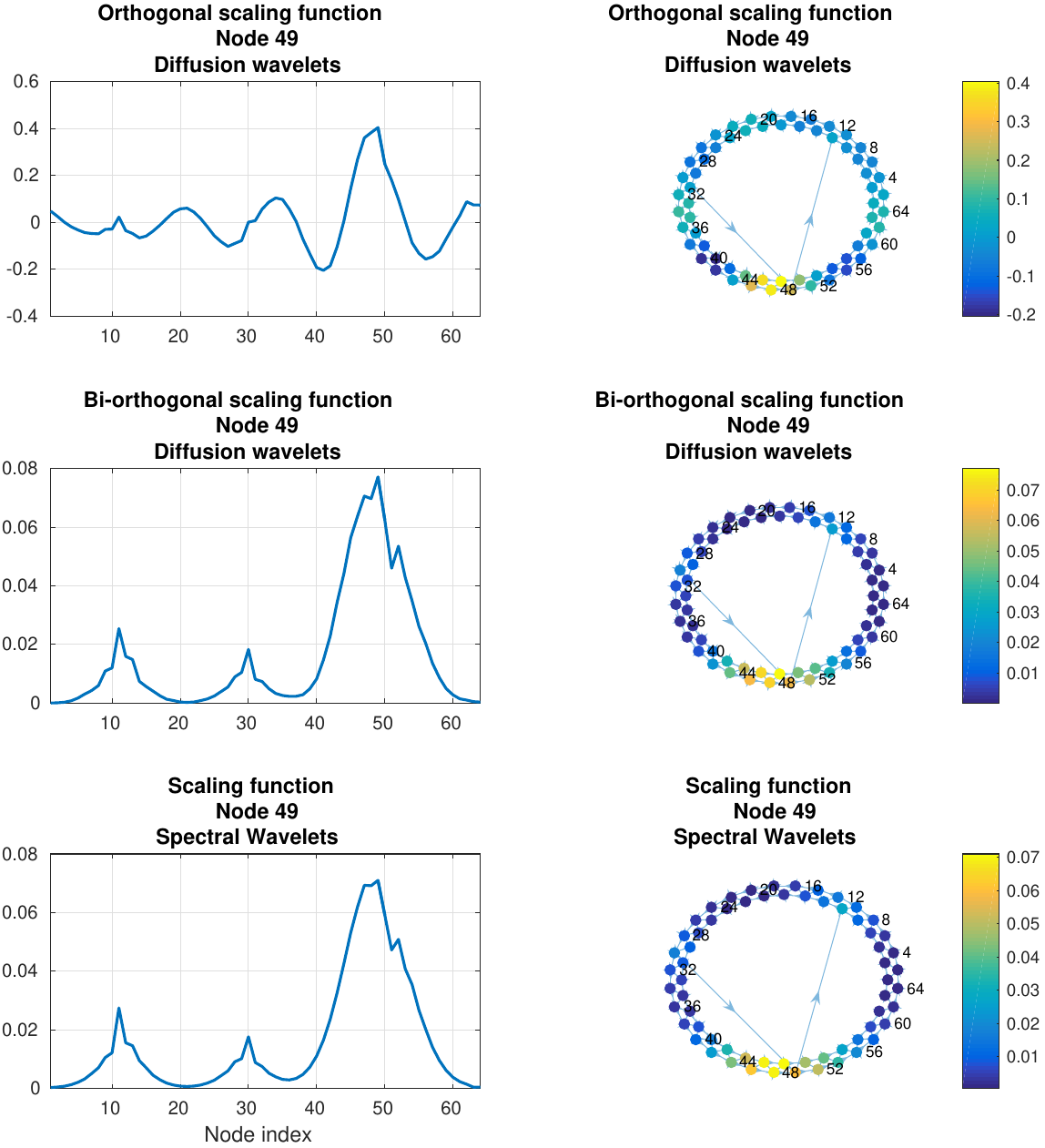}
 \caption{Orthogonal and bi-orthogonal scaling functions built w.r.t the diffusion wavelet framework versus scaling function built w.r.t spectral wavelets framework.}
 \label{fig:sc_ws_case_three}
 \end{figure}

\Cref{fig:sc_ws_case_three} shows the orthogonal and bi-orthogonal scaling functions at node $49$ built using the diffusion wavelets framework and the scaling function at node $49$ built using the spectral graph wavelets framework. We note that the biorthogonal scaling function and the scaling function built by the spectral graph wavelet framework are similar, i.e. well localized around nodes $11$ and $30$ and localized around $49$ with a large size support. By contrast, the orthogonal scaling function has a poor spatial localization around the nodes  $11$ and $30$.

 \subsubsection{Semi-supervised learning on directed graphs with $\ell_1$-regularization}
\label{sec:semi_sup_l_one}
 
We discuss the semi-supervised learning approach in the case of functions over directed graphs with $\ell_1$ regularization on the wavelet coefficients. Our aim is to show that the performance of the semi-supervised learning problem with $\ell_1$ regularization  is competitive compared to the existing approaches, e.g. the semi-supervised learning problem studied in section~\ref{sect:applications_one}. 
% We discuss the following semi-supervised learning problem, again on the example of the political blogs dataset~\cite{adamic2005political}. 
We use the same notations as in~\cref{sect:applications_one} and work on the subgraph $\mathcal{G}'$. The $\ell_1$-regularized solution for the semi-supervised learning problem is 

\begin{equation}
\label{semi_sup_l_one}
\boldsymbol w^{*}=\underset{\boldsymbol{w}}{\operatorname{argmin}}\:\|\tilde{\boldsymbol  y}-\mathbf{M}\mathbf{K}\boldsymbol w\|_2^2+\lambda\|\boldsymbol w\|_1,\quad \lambda\in\mathbb{R}^{+}.
\end{equation}

The graph signal $\tilde{\boldsymbol  y}=\mathbf{M}\boldsymbol y$ is the partially labeled graph signal with $\mathbf{M}=\{m_{ij}\}_{\{1\leq i,j\leq N\}}$ is the mask operator, i.e. the diagonal matrix where $m_{ii}=\mathbb{1}_{v_i\in\mathcal{O}}$ where $\mathcal{O}\subset\mathcal{V}$ is the subset of known labels. 

The matrix $\mathbf{K}=(\mathbf{H}_J,\mathbf{G}_1,\cdots,\mathbf{G}_J)$ is the synthesis filter bank. If we set $\mathbf{X}=\mathbf{MK}$, the equation \eqref{semi_sup_l_one} can be rewritten as
\begin{equation*}
\boldsymbol w^{*}=\underset{\boldsymbol{w}}{\operatorname{argmin}}\:\|\tilde{\boldsymbol y}-\mathbf{X}\boldsymbol w\|_2^2+\lambda\|\boldsymbol w\|_1,\quad \lambda\in\mathbb{R}.
\end{equation*}

The formulation \eqref{semi_sup_l_one} is identical to the formulation of the problem of signal restoration with redundant wavelet transforms in~\cite{selesnick2009signal}, except that $\mathbf{K}$ is the synthesis wavelet transform for functions defined over directed graphs. Furthermore, previous approaches of semi-supervised learning on undirected graphs have been investigated using overcomplete graph wavelets~\cite{shuman2011semi} or critically sampled spline graph wavelets~\cite{ekambaram2013wavelet}. The $\ell_1$-regularized synthesis semi-supervised learning problem is convex and can be solved efficiently using e.g. proximal splitting methods~\cite{combettes2011proximal,beck2009fast,becker2011nesta}. From the solution $\boldsymbol{w}^{*}$ of \cref{semi_sup_l_one}, we define the restored signal $\boldsymbol f^{*}$ as
\begin{equation*}
\boldsymbol f^{*}=\operatorname{sign}(\mathbf{K}\boldsymbol{w}^{*}). 
\end{equation*}
% where $\mathcal{L}$ is the directed normalized Laplacian defined at~\eqref{eq:dir_norm_laplacian} and the graph bandpass filters $\mathbf{G}_1,\mathbf{G}_2$ are

We compare the performance of formulations \eqref{semi_sup_l_one} and \eqref{semi_sup_dir_zhu}. Formulation \eqref{semi_sup_l_one} requires a synthesis filter bank $\mathbf{K}$. We choose a two scales synthesis graph filter bank $\mathbf{K}=(\mathbf{H}_2,\mathbf{G}_1,\mathbf{G}_2)$  based on a heat kernel construction~\cite{barlow2017random}
\begin{equation*}
\mathbf{H}_j=e^{-\mathcal{L}t_j},\quad t_j=2^j, \quad j=1,2.
\end{equation*}
with $\mathcal{L}$ the directed Laplacian defined at \eqref{eq:dir_norm_laplacian}
and the graph bandpass filters $\mathbf{G}_1,\mathbf{G}_2$ as
\begin{equation*}
\mathbf{G}_2=\mathbf{H}_{1}-\mathbf{H}_2,\quad \mathbf{G}_1=\mathbf{I}-\mathbf{H}_1.
\end{equation*}

\begin{figure}
\includegraphics[width=0.7\columnwidth]
{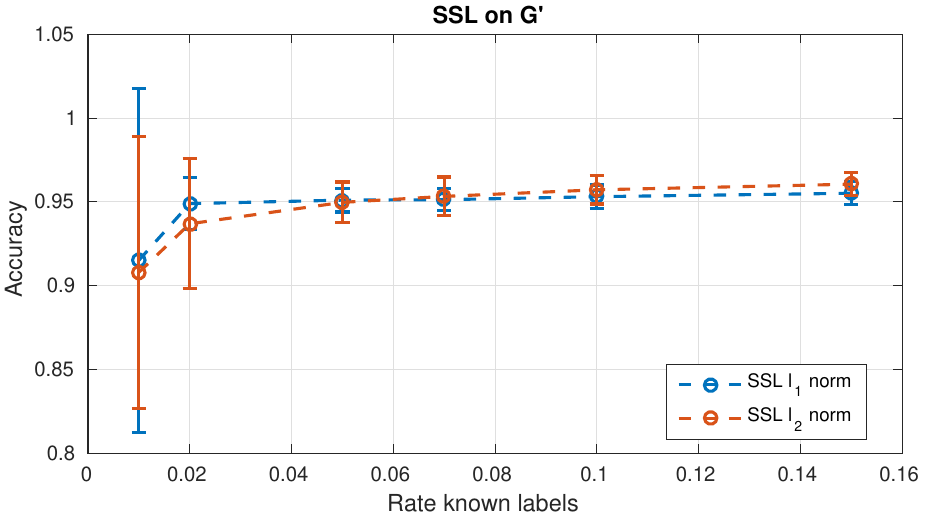}
\caption{SSL on the largest connected graph of US political blogs.}
\label{ssl_wav_vs_dir}
\end{figure}

\Cref{ssl_wav_vs_dir} shows the performance of semi-supervised learning formulations~\eqref{semi_sup_l_one} and~\eqref{semi_sup_dir_zhu} on the largest strongly connected graph obtained from US political blogs. 
The performance is obtained by averaging 200 realizations and choosing the parameters $\lambda$ and $\gamma$ that maximize the performance. On~\cref{ssl_wav_vs_dir}, the performance based on the $\ell_1$ penalization of redundant graph wavelets coefficients on directed graphs is competitive against the problem with $\ell_2$ Dirichlet regularization term for all percentage of known labels. Normally, one should expect better performance in the $\ell_1$ case since the signal should be rather sparse in the wavelet basis. However, the wavelet representation of the signal here is not that sparse because our graph wavelets are not that good in terms of spatial localization.

% This is a consequence of using a construction based on a heat kernel instead of a polynomial of $\mathcal{L}$ with a low order.
% Thus, in order to improve the performance of the $\ell_1$ version, it could be wise to use wavelets with a better space-frequency characterization, and choose the best suitable number of scales with respect to the characteristics of the data set. 

% More specifically, for a low rate of known labels, the performance rate based on $\ell_1$ penalization  is slightly better than the $\ell_2$ problem and as the rate of known labels increases, the $\ell_1$ problem \eqref{semi_sup_l_one} is competitive compared to the $\ell_2$ problem \eqref{semi_sup_dir_zhu}. 

\section{Conclusion} 
\label{sec:ccl}
We introduced a novel harmonic analysis on directed graphs. First, we proposed a frequency analysis for functions defined on directed graphs based on the eigenvectors of the random walk operator on a directed graph. From this Fourier-type frequency interpretation, we showed how to construct redundant wavelets on directed graphs as well as critically sampled wavelets by generalizing the diffusion wavelets framework. Finally, we illustrated our harmonic analysis through examples of semi-supervised learning and graph signal modeling on directed graphs and showed the relevance of our framework compared to existing approaches. 

\section{Acknowledgements}
The authors would like to thank Romain Cosentino, Randall Balestriero, Yanis Bahroun, Benjamin Girault and Antonio Ortega for the useful discussions.

\section{References} 

\bibliographystyle{elsarticle-num} 
\bibliography{bibliography.bib}

\begin{thebibliography}{10}
\expandafter\ifx\csname url\endcsname\relax
  \def\url#1{\texttt{#1}}\fi
\expandafter\ifx\csname urlprefix\endcsname\relax\def\urlprefix{URL }\fi
\expandafter\ifx\csname href\endcsname\relax
  \def\href#1#2{#2} \def\path#1{#1}\fi

\bibitem{sevi2017multiresolution}
H.~Sevi, G.~Rilling, P.~Borgnat, Multiresolution analysis of functions on
  directed networks, in: Wavelets and Sparsity XVII, Vol. 10394, International
  Society for Optics and Photonics, 2017, p. 103941Q.

\bibitem{sevi:GlobalSIP:2018}
H.~Sevi, G.~Rilling, P.~Borgnat, Modeling signals over directed graphs through
  filtering, in: 2018 IEEE Global Conference on Signal and Information
  Processing (GlobalSIP), 2018, pp. 718--722.

\bibitem{barthelemy2011spatial}
M.~Barth{\'e}lemy, Spatial networks, Physics Reports 499~(1-3) (2011) 1--101.

\bibitem{barrat2008dynamical}
A.~Barrat, M.~Barthelemy, A.~Vespignani, Dynamical processes on complex
  networks, Cambridge university press, 2008.

\bibitem{belkin2003laplacian}
M.~Belkin, P.~Niyogi, Laplacian eigenmaps for dimensionality reduction and data
  representation, Neural computation 15~(6) (2003) 1373--1396.

\bibitem{belkin2004regularization}
M.~Belkin, I.~Matveeva, P.~Niyogi, Regularization and semi-supervised learning
  on large graphs, in: International Conference on Computational Learning
  Theory, Springer, 2004, pp. 624--638.

\bibitem{gine2006empirical}
E.~Gin{\'e}, V.~Koltchinskii, et~al., Empirical graph laplacian approximation
  of laplace--beltrami operators: Large sample results, in: High dimensional
  probability, Institute of Mathematical Statistics, 2006, pp. 238--259.

\bibitem{coifman2006diffusionmaps}
R.~R. Coifman, S.~Lafon, Diffusion maps, Applied and computational harmonic
  analysis 21~(1) (2006) 5--30.

\bibitem{von2007tutorial}
U.~Von~Luxburg, A tutorial on spectral clustering, Statistics and computing
  17~(4) (2007) 395--416.

\bibitem{rosenberg1997laplacian}
S.~Rosenberg, The Laplacian on a Riemannian manifold: an introduction to
  analysis on manifolds, Vol.~31, Cambridge University Press, 1997.

\bibitem{chung1997spectral}
F.~R. Chung, Spectral graph theory, no.~92, American Mathematical Soc., 1997.

\bibitem{lablee2015spectral}
O.~Labl{\'e}e, Spectral Theory in Riemannian Geometry, 2015.

\bibitem{shuman2013emerging}
D.~I. Shuman, S.~K. Narang, P.~Frossard, A.~Ortega, P.~Vandergheynst, The
  emerging field of signal processing on graphs: Extending high-dimensional
  data analysis to networks and other irregular domains, IEEE Signal Processing
  Magazine 30~(3) (2013) 83--98.

\bibitem{ortega2017graph}
A.~Ortega, P.~Frossard, J.~Kovacevic, J.~M.~F. Moura, P.~Vandergheynst, Graph
  signal processing: Overview, challenges, and applications, Proceedings of the
  IEEE 106~(5) (2018) 808--828.

\bibitem{belkin2007convergence}
M.~Belkin, P.~Niyogi, Convergence of laplacian eigenmaps, in: Advances in
  Neural Information Processing Systems, 2007, pp. 129--136.

\bibitem{hein2007graph}
M.~Hein, J.-Y. Audibert, U.~v. Luxburg, Graph laplacians and their convergence
  on random neighborhood graphs, Journal of Machine Learning Research 8~(Jun)
  (2007) 1325--1368.

\bibitem{RICAUD2019474}
B.~Ricaud, P.~Borgnat, N.~Tremblay, P.~Gon\c{c}alves, P.~Vandergheynst, Fourier
  could be a data scientist: From graph fourier transform to signal processing
  on graphs, Comptes Rendus Physique 20~(5) (2019) 474--488, special section
  ``{F}ourier and the science of today''.

\bibitem{lovasz1993random}
L.~Lov{\'a}sz, Random walks on graphs, Combinatorics, Paul erdos is eighty
  2~(1-46) (1993) 4.

\bibitem{coulhon1998random}
T.~Coulhon, A.~Grigoryan, Random walks on graphs with regular volume growth,
  Geometric and Functional Analysis 8~(4) (1998) 656--701.

\bibitem{aldous2002reversible}
D.~Aldous, J.~Fill, Reversible markov chains and random walks on graphs (2002).

\bibitem{mallat1989theory}
S.~G. Mallat, A theory for multiresolution signal decomposition: the wavelet
  representation, IEEE transactions on pattern analysis and machine
  intelligence 11~(7) (1989) 674--693.

\bibitem{coifman2006diffusion}
R.~R. Coifman, M.~Maggioni, Diffusion wavelets, Applied and Computational
  Harmonic Analysis 21~(1) (2006) 53--94.

\bibitem{maggioni2008diffusion}
M.~Maggioni, H.~Mhaskar, Diffusion polynomial frames on metric measure spaces,
  Applied and Computational Harmonic Analysis 24~(3) (2008) 329--353.

\bibitem{hammond2011wavelets}
D.~K. Hammond, P.~Vandergheynst, R.~Gribonval, Wavelets on graphs via spectral
  graph theory, Applied and Computational Harmonic Analysis 30~(2) (2011)
  129--150.

\bibitem{rustamov2013wavelets}
R.~Rustamov, L.~J. Guibas, Wavelets on graphs via deep learning, in: Advances
  in neural information processing systems, 2013, pp. 998--1006.

\bibitem{dong2017sparse}
B.~Dong, Sparse representation on graphs by tight wavelet frames and
  applications, Applied and Computational Harmonic Analysis 42~(3) (2017)
  452--479.

\bibitem{bremer2006diffusion}
J.~C. Bremer, R.~R. Coifman, M.~Maggioni, A.~D. Szlam, Diffusion wavelet
  packets, Applied and Computational Harmonic Analysis 21~(1) (2006) 95--112.

\bibitem{cohen1992biorthogonal}
A.~Cohen, I.~Daubechies, J.-C. Feauveau, Biorthogonal bases of compactly
  supported wavelets, Communications on pure and applied mathematics 45~(5)
  (1992) 485--560.

\bibitem{sweldens1996lifting}
W.~Sweldens, The lifting scheme: A custom-design construction of biorthogonal
  wavelets, Applied and computational harmonic analysis 3~(2) (1996) 186--200.

\bibitem{kingsbury2001complex}
N.~Kingsbury, Complex wavelets for shift invariant analysis and filtering of
  signals, Applied and computational harmonic analysis 10~(3) (2001) 234--253.

\bibitem{candes2004new}
E.~J. Cand{\`e}s, D.~L. Donoho, New tight frames of curvelets and optimal
  representations of objects with piecewise c2 singularities, Communications on
  pure and applied mathematics 57~(2) (2004) 219--266.

\bibitem{peyre2008orthogonal}
G.~Peyr{\'e}, S.~Mallat, Orthogonal bandelet bases for geometric images
  approximation, Communications on Pure and Applied Mathematics 61~(9) (2008)
  1173--1212.

\bibitem{schroder1995spherical}
P.~Schr{\"o}der, W.~Sweldens, Spherical wavelets: Efficiently representing
  functions on the sphere, in: Proceedings of the 22nd annual conference on
  Computer graphics and interactive techniques, ACM, 1995, pp. 161--172.

\bibitem{antoine1999wavelets}
J.-P. Antoine, P.~Vandergheynst, Wavelets on the 2-sphere: A group-theoretical
  approach, Applied and Computational Harmonic Analysis 7~(3) (1999) 262--291.

\bibitem{wiaux2008exact}
Y.~Wiaux, J.~McEwen, P.~Vandergheynst, O.~Blanc, Exact reconstruction with
  directional wavelets on the sphere, Monthly Notices of the Royal Astronomical
  Society 388~(2) (2008) 770--788.

\bibitem{maggioni2005biorthogonal}
M.~Maggioni, J.~C. Bremer, R.~R. Coifman, A.~D. Szlam, Biorthogonal diffusion
  wavelets for multiscale representations on manifolds and graphs, in: Wavelets
  XI, Vol. 5914, International Society for Optics and Photonics, 2005, p.
  59141M.

\bibitem{coulhon2012heat}
T.~Coulhon, G.~Kerkyacharian, P.~Petrushev, Heat kernel generated frames in the
  setting of dirichlet spaces, Journal of Fourier Analysis and Applications
  18~(5) (2012) 995--1066.

\bibitem{geller2009continuous}
D.~Geller, A.~Mayeli, Continuous wavelets on compact manifolds, Mathematische
  Zeitschrift 262~(4) (2009) 895.

\bibitem{rahman2005multiscale}
I.~U. Rahman, I.~Drori, V.~C. Stodden, D.~L. Donoho, P.~Schr{\"o}der,
  Multiscale representations for manifold-valued data, Multiscale Modeling \&
  Simulation 4~(4) (2005) 1201--1232.

\bibitem{crovella2003graph}
M.~Crovella, E.~Kolaczyk, Graph wavelets for spatial traffic analysis, in:
  INFOCOM 2003. Twenty-Second Annual Joint Conference of the IEEE Computer and
  Communications. IEEE Societies, Vol.~3, IEEE, 2003, pp. 1848--1857.

\bibitem{jansen2009multiscale}
M.~Jansen, G.~P. Nason, B.~W. Silverman, Multiscale methods for data on graphs
  and irregular multidimensional situations, Journal of the Royal Statistical
  Society: Series B (Statistical Methodology) 71~(1) (2009) 97--125.

\bibitem{tremblay2014graph}
N.~Tremblay, P.~Borgnat, Graph wavelets for multiscale community mining, IEEE
  Transactions on Signal Processing 62~(20) (2014) 5227--5239.

\bibitem{narang2012perfect}
S.~K. Narang, A.~Ortega, Perfect reconstruction two-channel wavelet filter
  banks for graph structured data, IEEE Transactions on Signal Processing
  60~(6) (2012) 2786--2799.

\bibitem{de2019localized}
B.~de~Loynes, F.~Navarro, B.~Olivier, Localized fourier analysis for graph
  signal processing, arXiv preprint arXiv:1906.04529.

\bibitem{szlam2005diffusion}
A.~D. Szlam, M.~Maggioni, R.~R. Coifman, J.~C. Bremer, Diffusion-driven
  multiscale analysis on manifolds and graphs: top-down and bottom-up
  constructions, in: Wavelets XI, Vol. 5914, International Society for Optics
  and Photonics, 2005, p. 59141D.

\bibitem{gavish2010multiscale}
M.~Gavish, B.~Nadler, R.~R. Coifman, Multiscale wavelets on trees, graphs and
  high dimensional data: Theory and applications to semi supervised learning.,
  in: ICML, 2010, pp. 367--374.

\bibitem{irion2014hierarchical}
J.~Irion, N.~Saito, Hierarchical graph laplacian eigen transforms, JSIAM
  Letters 6 (2014) 21--24.

\bibitem{tremblay2016graph}
N.~Tremblay, P.~Borgnat, Subgraph-based filterbanks for graph signals, IEEE
  Transactions on Signal Processing 64~(15) (2016) 3827--3840.

\bibitem{rustamov2011average}
R.~M. Rustamov, Average interpolating wavelets on point clouds and graphs,
  arXiv preprint arXiv:1110.2227.

\bibitem{shuman2016multiscale}
D.~I. Shuman, M.~J. Faraji, P.~Vandergheynst, A multiscale pyramid transform
  for graph signals, IEEE Transactions on Signal Processing 64~(8) (2016)
  2119--2134.

\bibitem{gobel2018construction}
F.~G{\"o}bel, G.~Blanchard, U.~von Luxburg, Construction of tight frames on
  graphs and application to denoising, in: Handbook of Big Data Analytics,
  Springer, 2018, pp. 503--522.

\bibitem{behjat2016signal}
H.~Behjat, U.~Richter, D.~Van De~Ville, L.~S{\"o}rnmo, Signal-adapted tight
  frames on graphs, IEEE Transactions on Signal Processing 64~(22) (2016)
  6017--6029.

\bibitem{avena2017intertwining}
L.~Avena, F.~Castell, A.~Gaudilli{\`e}re, C.~M{\'e}lot, Intertwining wavelets
  or multiresolution analysis on graphs through random forests, arXiv preprint
  arXiv:1707.04616.

\bibitem{sandryhaila2013discrete}
A.~Sandryhaila, J.~M. Moura, Discrete signal processing on graphs, IEEE
  transactions on signal processing 61~(7) (2013) 1644--1656.

\bibitem{sandryhaila2014discrete}
A.~Sandryhaila, J.~M. Moura, Discrete signal processing on graphs: Frequency
  analysis., IEEE Trans. Signal Processing 62~(12) (2014) 3042--3054.

\bibitem{TayOrtega2017}
D.~B.~H. Tay, A.~Ortega, Bipartite graph filter banks: Polyphase analysis and
  generalization, IEEE Transactions on Signal Processing 65~(18) (2017)
  4833--4846.

\bibitem{sardellitti2017graph}
S.~Sardellitti, S.~Barbarossa, P.~D. Lorenzo, On the {{Graph Fourier
  Transform}} for {{Directed Graphs}}, IEEE Journal of Selected Topics in
  Signal Processing 11~(6) (2017) 796--811, 00039.

\bibitem{shafipour2017digraph}
R.~Shafipour, A.~Khodabakhsh, G.~Mateos, E.~Nikolova, A {{Directed Graph
  Fourier Transform With Spread Frequency Components}}, IEEE Transactions on
  Signal Processing 67~(4) (2019) 946--960, 00008.

\bibitem{marques2020signal}
A.~G. Marques, S.~Segarra, G.~Mateos, Signal processing on directed graphs: The
  role of edge directionality when processing and learning from network data,
  IEEE Signal Processing Magazine 37~(6) (2020) 99--116.

\bibitem{mhaskar2016unified}
H.~N. Mhaskar, A unified framework for harmonic analysis of functions on
  directed graphs and changing data, Applied and Computational Harmonic
  Analysis 44~(3) (2018) 611--644.

\bibitem{furutani2019graph}
S.~Furutani, T.~Shibahara, M.~Akiyama, K.~Hato, M.~Aida, Graph signal
  processing for directed graphs based on the hermitian laplacian, in: Joint
  European Conference on Machine Learning and Knowledge Discovery in Databases,
  Springer, 2019, pp. 447--463.

\bibitem{colin2013magnetic}
Y.~Colin~de Verdi{\`e}re, Magnetic interpretation of the nodal defect on
  graphs, Analysis \& PDE 6~(5) (2013) 1235--1242.

\bibitem{berkolaiko2013nodal}
G.~Berkolaiko, Nodal count of graph eigenfunctions via magnetic perturbation,
  Analysis \& PDE 6~(5) (2013) 1213--1233.

\bibitem{fanuel2017magnetic}
M.~Fanuel, C.~M. Alaiz, J.~A. Suykens, Magnetic eigenmaps for community
  detection in directed networks, Physical Review E 95~(2) (2017) 022302.

\bibitem{zhang2021magnet}
X.~Zhang, N.~Brugnone, M.~Perlmutter, M.~Hirn, Magnet: A magnetic neural
  network for directed graphs, arXiv preprint arXiv:2102.11391.

\bibitem{guo2017hermitian}
K.~Guo, B.~Mohar, Hermitian adjacency matrix of digraphs and mixed graphs,
  Journal of Graph Theory 85~(1) (2017) 217--248.

\bibitem{cucuringu2020hermitian}
M.~Cucuringu, H.~Li, H.~Sun, L.~Zanetti, Hermitian matrices for clustering
  directed graphs: insights and applications, in: International Conference on
  Artificial Intelligence and Statistics, PMLR, 2020, pp. 983--992.

\bibitem{laenen2020higher}
S.~Laenen, H.~Sun, Higher-order spectral clustering of directed graphs, in:
  H.~Larochelle, M.~Ranzato, R.~Hadsell, M.~F. Balcan, H.~Lin (Eds.), Advances
  in Neural Information Processing Systems, Vol.~33, Curran Associates, Inc.,
  2020, pp. 941--951.

\bibitem{bauer2012normalized}
F.~Bauer, Normalized graph laplacians for directed graphs, Linear Algebra and
  its Applications 436~(11) (2012) 4193--4222.

\bibitem{langville2011google}
A.~N. Langville, C.~D. Meyer, Google's PageRank and beyond: The science of
  search engine rankings, Princeton University Press, 2011.

\bibitem{mugnolo2016semigroup}
D.~Mugnolo, Semigroup methods for evolution equations on networks, Springer,
  2016.

\bibitem{bremaud2013markov}
P.~Br{\'e}maud, Markov chains: Gibbs fields, Monte Carlo simulation, and
  queues, Vol.~31, Springer Science \& Business Media, 2013.

\bibitem{montenegro2006mathematical}
R.~Montenegro, P.~Tetali, et~al., Mathematical aspects of mixing times in
  markov chains, Foundations and Trends{\textregistered} in Theoretical
  Computer Science 1~(3) (2006) 237--354.

\bibitem{golub2012matrix}
G.~H. Golub, C.~F. Van~Loan, Matrix computations, Vol.~3, JHU Press, 2012.

\bibitem{atchade2009adaptive}
Y.~Atchad{\'e}, G.~Fort, E.~Moulines, P.~Priouret, Adaptive markov chain monte
  carlo: theory and methods, Preprint.

\bibitem{fill1991eigenvalue}
J.~A. Fill, Eigenvalue bounds on convergence to stationarity for nonreversible
  markov chains, with an application to the exclusion process, The annals of
  applied probability (1991) 62--87.

\bibitem{anis2016efficient}
A.~Anis, A.~Gadde, A.~Ortega, Efficient sampling set selection for bandlimited
  graph signals using graph spectral proxies, IEEE Transactions on Signal
  Processing 64~(14) (2016) 3775--3789.

\bibitem{girault2018irregularity}
B.~Girault, A.~Ortega, S.~S. Narayanan, Irregularity-aware graph fourier
  transforms, IEEE Transactions on Signal Processing 66~(21) (2018) 5746--5761.

\bibitem{chung2005laplacians}
F.~Chung, Laplacians and the cheeger inequality for directed graphs, Annals of
  Combinatorics 9~(1) (2005) 1--19.

\bibitem{butler2007interlacing}
S.~Butler, Interlacing for weighted graphs using the normalized laplacian,
  Electronic Journal of Linear Algebra 16~(1) (2007) 8.

\bibitem{kubrusly2003hilbert}
C.~S. Kubrusly, Hilbert space operators, in: Hilbert Space Operators, Springer,
  2003, pp. 13--22.

\bibitem{terras1999fourier}
A.~Terras, Fourier analysis on finite groups and applications, Vol.~43,
  Cambridge University Press, 1999.

\bibitem{kaveh2013optimal}
A.~Kaveh, Optimal analysis of structures by concepts of symmetry and
  regularity, Springer, 2013.

\bibitem{zhou2005learning}
D.~Zhou, J.~Huang, B.~Sch{\"o}lkopf, Learning from labeled and unlabeled data
  on a directed graph, in: Proceedings of the 22nd international conference on
  Machine learning, ACM, 2005, pp. 1036--1043.

\bibitem{zhu2005semi}
X.~Zhu, J.~Lafferty, R.~Rosenfeld, Semi-supervised learning with graphs, Ph.D.
  thesis, Carnegie Mellon University, language technologies institute, school
  of computer science Pittsburgh, PA (2005).

\bibitem{adamic2005political}
L.~A. Adamic, N.~Glance, The political blogosphere and the 2004 us election:
  divided they blog, in: Proceedings of the 3rd international workshop on Link
  discovery, ACM, 2005, pp. 36--43.

\bibitem{RevModPhys.87.1261}
L.~Ermann, K.~M. Frahm, D.~L. Shepelyansky,
  \href{https://link.aps.org/doi/10.1103/RevModPhys.87.1261}{Google matrix
  analysis of directed networks}, Rev. Mod. Phys. 87 (2015) 1261--1310.
\newblock \href {http://dx.doi.org/10.1103/RevModPhys.87.1261}
  {\path{doi:10.1103/RevModPhys.87.1261}}.
\newline\urlprefix\url{https://link.aps.org/doi/10.1103/RevModPhys.87.1261}

\bibitem{civril2009column}
A.~Civril, Column subset selection for approximating data matrices, Ph.D.
  thesis, Rensselaer Polytechnic Institute (2009).

\bibitem{watts1998collective}
D.~J. Watts, S.~H. Strogatz, Collective dynamics of "small-world" networks,
  nature 393~(6684) (1998) 440.

\bibitem{newman2003mej}
M.~Newman, The structure and function of complex networks, Siam Rev. 45 (2003)
  167.

\bibitem{selesnick2009signal}
I.~W. Selesnick, M.~A. Figueiredo, Signal restoration with overcomplete wavelet
  transforms: Comparison of analysis and synthesis priors, in: Wavelets XIII,
  Vol. 7446, International Society for Optics and Photonics, 2009, p. 74460D.

\bibitem{shuman2011semi}
D.~I. Shuman, M.~Faraji, P.~Vandergheynst, Semi-supervised learning with
  spectral graph wavelets, in: Proceedings of the International Conference on
  Sampling Theory and Applications (SampTA), no. EPFL-CONF-164765, 2011.

\bibitem{ekambaram2013wavelet}
V.~N. Ekambaram, G.~Fanti, B.~Ayazifar, K.~Ramchandran, Wavelet-regularized
  graph semi-supervised learning, in: Global Conference on Signal and
  Information Processing (GlobalSIP), 2013 IEEE, IEEE, 2013, pp. 423--426.

\bibitem{combettes2011proximal}
P.~L. Combettes, J.-C. Pesquet, Proximal splitting methods in signal
  processing, in: Fixed-point algorithms for inverse problems in science and
  engineering, Springer, 2011, pp. 185--212.

\bibitem{beck2009fast}
A.~Beck, M.~Teboulle, A fast iterative shrinkage-thresholding algorithm for
  linear inverse problems, SIAM journal on imaging sciences 2~(1) (2009)
  183--202.

\bibitem{becker2011nesta}
S.~Becker, J.~Bobin, E.~J. Cand{\`e}s, Nesta: A fast and accurate first-order
  method for sparse recovery, SIAM Journal on Imaging Sciences 4~(1) (2011)
  1--39.

\bibitem{barlow2017random}
M.~T. Barlow, Random walks and heat kernels on graphs, Vol. 438, Cambridge
  University Press, 2017.

\end{thebibliography}

%\bibliographystyle{unsrt}
% \printbibliography

\end{document}